\documentclass[leqno,11pt]{imsart}

\usepackage{amsthm,amsmath}
\usepackage[margin=1.1in]{geometry}
\usepackage{scrextend}

\usepackage{amssymb,amsmath,latexsym}
\usepackage[dvips]{graphics}
\usepackage[pdflatex]{graphicx}
\usepackage{epsfig}
\usepackage[T1]{fontenc}
\usepackage[latin1]{inputenc}

\usepackage[mathscr]{eucal}
\usepackage{graphicx}
\usepackage{latexsym}
\usepackage{amssymb}
\usepackage{psfrag}
\usepackage{epsfig}
\usepackage{MnSymbol}
\usepackage{enumerate}
\usepackage{multirow}
\DeclareMathAlphabet{\mathpzc}{OT1}{pzc}{m}{it}

\usepackage{amsfonts}
\usepackage{graphicx}
\usepackage{amsfonts}
\usepackage{color}
\usepackage[usenames,dvipsnames]{xcolor}
\usepackage[textsize=small]{todonotes}
\usepackage{caption,subcaption}
\usepackage{soul}
\usepackage{hyperref}

\usepackage[numbers]{natbib}

\begin{document}

\renewcommand*{\bibfont}{\small}
\newtheorem{proposition}{Proposition}[section]
\newtheorem{theorem}[proposition]{Theorem}
\newtheorem{corollary}[proposition]{Corollary}
\newtheorem{lemma}[proposition]{Lemma}
\newtheorem{conjecture}[proposition]{Conjecture}
\newtheorem{question}[proposition]{Question}
\newtheorem{definition}[proposition]{Definition}
\newtheorem{comment}[proposition]{Comment}
	\newtheorem{example}[proposition]{Example}
\newtheorem{algorithm}[proposition]{Algorithm}
\newtheorem{assumption}[proposition]{Assumption}
\newtheorem{condition}[proposition]{Condition}
\numberwithin{equation}{section}
\numberwithin{proposition}{section}

\theoremstyle{remark}
\newtheorem{remark}{Remark}[section]

\captionsetup[table]{format=plain,labelformat=simple,labelsep=period}

\newcommand{\skp}{\vspace{\baselineskip}}
\newcommand{\noi}{\noindent}
\newcommand{\osc}{\mbox{osc}}
\newcommand{\lfl}{\lfloor}
\newcommand{\rfl}{\rfloor}

\newcommand{\leb}{\mbox{Leb}}
\newcommand{\veps}{\varepsilon}
\newcommand{\img}{\imath}
\newcommand{\iy}{\infty}
\newcommand{\eps}{\varepsilon}
\newcommand{\del}{\delta}
\newcommand{\Rk}{\mathbb{R}^k}
\newcommand{\RR}{\mathbb{R}}
\newcommand{\spa}{\vspace{.2in}}
\newcommand{\V}{\mathcal{V}}
\newcommand{\E}{\mathbb{E}}
\newcommand{\I}{\mathbb{I}}
\newcommand{\PP}{\mathbb{P}}
\newcommand{\sgn}{\mbox{sgn}}
\newcommand{\ti}{\tilde}
\newcommand{\GG}{H}
\newcommand{\GGnew}{A}

\newcommand{\QQ}{\mathbb{Q}}

\newcommand{\XX}{\mathbb{X}}
\newcommand{\XXz}{\mathbb{X}^0}

\newcommand{\lan}{\langle}
\newcommand{\ran}{\rangle}
\newcommand{\llan}{\llangle}
\newcommand{\rran}{\rrangle}
\newcommand{\lf}{\lfloor}
\newcommand{\rf}{\rfloor}
\def\wh{\widehat}
\newcommand{\defn}{\stackrel{def}{=}}
\newcommand{\txb}{\tau^{\eps,x}_{B^c}}
\newcommand{\tyb}{\tau^{\eps,y}_{B^c}}
\newcommand{\tilxb}{\tilde{\tau}^\eps_1}
\newcommand{\pxeps}{\mathbb{P}_x^{\eps}}
\newcommand{\non}{\nonumber}
\newcommand{\dist}{\mbox{dist}}

\newcommand{\Om}{\mathnormal{\Omega}}
\newcommand{\om}{\omega}
\newcommand{\vph}{\varphi}
\newcommand{\Del}{\mathnormal{\Delta}}
\newcommand{\Gam}{\mathnormal{\Gamma}}
\newcommand{\Sig}{\mathnormal{\Sigma}}

\newcommand{\tilyb}{\tilde{\tau}^\eps_2}
\newcommand{\beq}{\begin{eqnarray*}}
\newcommand{\eeq}{\end{eqnarray*}}
\newcommand{\beqn}{\begin{eqnarray}}
\newcommand{\eeqn}{\end{eqnarray}}
\newcommand{\ink}{\rule{.5\baselineskip}{.55\baselineskip}}

\newcommand{\bt}{\begin{theorem}}
\newcommand{\et}{\end{theorem}}
\newcommand{\deps}{\Del_{\eps}}
\newcommand{\dbl}{\mathbf{d}_{\tiny{\mbox{BL}}}}

\newcommand{\be}{\begin{equation}}
\newcommand{\ee}{\end{equation}}
\newcommand{\ac}{\mbox{AC}}
\newcommand{\hs}{\tiny{\mbox{HS}}}
\newcommand{\BB}{\mathbb{B}}
\newcommand{\VV}{\mathbb{V}}
\newcommand{\DD}{\mathbb{D}}
\newcommand{\KK}{\mathbb{K}}
\newcommand{\HH}{\mathbb{H}}
\newcommand{\TT}{\mathbb{T}}
\newcommand{\CC}{\mathbb{C}}
\newcommand{\ZZ}{\mathbb{Z}}
\newcommand{\SSS}{\mathbb{S}}
\newcommand{\EE}{\mathbb{E}}
\newcommand{\NN}{\mathbb{N}}
\newcommand{\MM}{\mathbb{M}}

\newcommand{\clg}{\mathcal{G}}
\newcommand{\clb}{\mathcal{B}}
\newcommand{\cls}{\mathcal{S}}
\newcommand{\clc}{\mathcal{C}}
\newcommand{\clj}{\mathcal{J}}
\newcommand{\clm}{\mathcal{M}}
\newcommand{\clx}{\mathcal{X}}
\newcommand{\cld}{\mathcal{D}}
\newcommand{\cle}{\mathcal{E}}
\newcommand{\clv}{\mathcal{V}}
\newcommand{\clu}{\mathcal{U}}
\newcommand{\clr}{\mathcal{R}}
\newcommand{\clt}{\mathcal{T}}
\newcommand{\cll}{\mathcal{L}}
\newcommand{\clz}{\mathcal{Z}}
\newcommand{\clq}{\mathcal{Q}}
\newcommand{\clo}{\mathcal{O}}

\newcommand{\cli}{\mathcal{I}}
\newcommand{\clp}{\mathcal{P}}
\newcommand{\cla}{\mathcal{A}}
\newcommand{\clf}{\mathcal{F}}
\newcommand{\clh}{\mathcal{H}}
\newcommand{\N}{\mathbb{N}}
\newcommand{\Q}{\mathbb{Q}}
\newcommand{\bfx}{{\boldsymbol{x}}}
\newcommand{\bfa}{{\boldsymbol{a}}}
\newcommand{\bfh}{{\boldsymbol{h}}}
\newcommand{\bfs}{{\boldsymbol{s}}}
\newcommand{\bfm}{{\boldsymbol{m}}}
\newcommand{\bff}{{\boldsymbol{f}}}
\newcommand{\bfb}{{\boldsymbol{b}}}
\newcommand{\bfw}{{\boldsymbol{w}}}
\newcommand{\bfz}{{\boldsymbol{z}}}
\newcommand{\bfu}{{\boldsymbol{u}}}
\newcommand{\bfell}{{\boldsymbol{\ell}}}
\newcommand{\bfn}{{\boldsymbol{n}}}
\newcommand{\bfd}{{\boldsymbol{d}}}
\newcommand{\bfbeta}{{\boldsymbol{\beta}}}
\newcommand{\bfzeta}{{\boldsymbol{\zeta}}}
\newcommand{\bfnu}{{\boldsymbol{\nu}}}
\newcommand{\bfvarphi}{{\boldsymbol{\varphi}}}

\newcommand{\curvz}{{\bf \mathpzc{z}}}
\newcommand{\curvx}{{\bf \mathpzc{x}}}
\newcommand{\curvi}{{\bf \mathpzc{i}}}
\newcommand{\curvs}{{\bf \mathpzc{s}}}
\newcommand{\blip}{\mathbb{B}_1}
\newcommand{\loc}{\text{loc}}

\newcommand{\BM}{\mbox{BM}}

\newcommand{\tac}{\mbox{\scriptsize{AC}}}

\newcommand{\Erdos}{Erd\H{o}s-R\'enyi }
\newcommand{\Meleard}{M\'el\'eard }
\newcommand{\Frechet}{Fr\'echet }
\def\blue{\textcolor{blue}}
\def\red{\textcolor{red}}
\newcommand{\set}[1]{\left\{#1\right\}}
\newcommand{\half}{{\frac{1}{2}}}
\newcommand{\quarter}{{\frac{1}{4}}}
\newcommand{\Rd}{{\Rmb^d}}
\newcommand{\intR}{\int_\Rmb}
\newcommand{\intRd}{\int_\Rd}
\definecolor{amet}{rgb}{0.8, 0.2, 0.8}
\newcommand{\cg}{\textcolor{amet}}

\newcommand{\Amb}{{\mathbb{A}}}
\newcommand{\Bmb}{{\mathbb{B}}}
\newcommand{\Cmb}{{\mathbb{C}}}
\newcommand{\Dmb}{{\mathbb{D}}}
\newcommand{\Emb}{{\mathbb{E}}}
\newcommand{\Fmb}{{\mathbb{F}}}
\newcommand{\Gmb}{{\mathbb{G}}}
\newcommand{\Hmb}{{\mathbb{H}}}
\newcommand{\Imb}{{\mathbb{I}}}
\newcommand{\Jmb}{{\mathbb{J}}}
\newcommand{\Kmb}{{\mathbb{K}}}
\newcommand{\Lmb}{{\mathbb{L}}}
\newcommand{\Mmb}{{\mathbb{M}}}
\newcommand{\Nmb}{{\mathbb{N}}}
\newcommand{\Omb}{{\mathbb{O}}}
\newcommand{\Pmb}{{\mathbb{P}}}
\newcommand{\Qmb}{{\mathbb{Q}}}
\newcommand{\Rmb}{{\mathbb{R}}}
\newcommand{\Smb}{{\mathbb{S}}}
\newcommand{\Tmb}{{\mathbb{T}}}
\newcommand{\Umb}{{\mathbb{U}}}
\newcommand{\Vmb}{{\mathbb{V}}}
\newcommand{\Wmb}{{\mathbb{W}}}
\newcommand{\Xmb}{{\mathbb{X}}}
\newcommand{\Ymb}{{\mathbb{Y}}}
\newcommand{\Zmb}{{\mathbb{Z}}}

\newcommand{\Amc}{{\mathcal{A}}}
\newcommand{\Bmc}{{\mathcal{B}}}
\newcommand{\Cmc}{{\mathcal{C}}}
\newcommand{\Dmc}{{\mathcal{D}}}
\newcommand{\Emc}{{\mathcal{E}}}
\newcommand{\Fmc}{{\mathcal{F}}}
\newcommand{\Gmc}{{\mathcal{G}}}
\newcommand{\Hmc}{{\mathcal{H}}}
\newcommand{\Imc}{{\mathcal{I}}}
\newcommand{\Jmc}{{\mathcal{J}}}
\newcommand{\Kmc}{{\mathcal{K}}}
\newcommand{\lmc}{{\mathcal{l}}}\newcommand{\Lmc}{{\mathcal{L}}}
\newcommand{\Mmc}{{\mathcal{M}}}
\newcommand{\Nmc}{{\mathcal{N}}}
\newcommand{\Omc}{{\mathcal{O}}}
\newcommand{\Pmc}{{\mathcal{P}}}
\newcommand{\Qmc}{{\mathcal{Q}}}
\newcommand{\Rmc}{{\mathcal{R}}}
\newcommand{\Smc}{{\mathcal{S}}}
\newcommand{\Tmc}{{\mathcal{T}}}
\newcommand{\Umc}{{\mathcal{U}}}
\newcommand{\Vmc}{{\mathcal{V}}}
\newcommand{\Wmc}{{\mathcal{W}}}
\newcommand{\Xmc}{{\mathcal{X}}}
\newcommand{\Ymc}{{\mathcal{Y}}}
\newcommand{\Zmc}{{\mathcal{Z}}}

\newcommand{\Abf}{{\mathbf{A}}}
\newcommand{\Bbf}{{\mathbf{B}}}
\newcommand{\Cbf}{{\mathbf{C}}}
\newcommand{\Dbf}{{\mathbf{D}}}
\newcommand{\Ebf}{{\mathbf{E}}}
\newcommand{\Fbf}{{\mathbf{F}}}
\newcommand{\Gbf}{{\mathbf{G}}}
\newcommand{\Hbf}{{\mathbf{H}}}
\newcommand{\Ibf}{{\mathbf{I}}}
\newcommand{\Jbf}{{\mathbf{J}}}
\newcommand{\Kbf}{{\mathbf{K}}}
\newcommand{\Lbf}{{\mathbf{L}}}
\newcommand{\Mbf}{{\mathbf{M}}}
\newcommand{\Nbf}{{\mathbf{N}}}
\newcommand{\Obf}{{\mathbf{O}}}
\newcommand{\Pbf}{{\mathbf{P}}}
\newcommand{\Qbf}{{\mathbf{Q}}}
\newcommand{\Rbf}{{\mathbf{R}}}
\newcommand{\Sbf}{{\mathbf{S}}}
\newcommand{\Tbf}{{\mathbf{T}}}
\newcommand{\Ubf}{{\mathbf{U}}}
\newcommand{\Vbf}{{\mathbf{V}}}
\newcommand{\Wbf}{{\mathbf{W}}}
\newcommand{\Xbf}{{\mathbf{X}}}
\newcommand{\Ybf}{{\mathbf{Y}}}
\newcommand{\Zbf}{{\mathbf{Z}}}

\newcommand{\zero}{{\boldsymbol{0}}}
\newcommand{\one}{{\boldsymbol{1}}}
\newcommand{\abd}{{\boldsymbol{a}}}\newcommand{\Abd}{{\boldsymbol{A}}}
\newcommand{\betabd}{{\boldsymbol{\beta}}}
\newcommand{\bbd}{{\boldsymbol{b}}}\newcommand{\Bbd}{{\boldsymbol{B}}}
\newcommand{\cbd}{{\boldsymbol{c}}}\newcommand{\Cbd}{{\boldsymbol{C}}}
\newcommand{\dbd}{{\boldsymbol{d}}}\newcommand{\Dbd}{{\boldsymbol{D}}}
\newcommand{\ebd}{{\boldsymbol{e}}}\newcommand{\Ebd}{{\boldsymbol{E}}}
\newcommand{\etabd}{{\boldsymbol{\eta}}}
\newcommand{\fbd}{{\boldsymbol{f}}}\newcommand{\Fbd}{{\boldsymbol{F}}}
\newcommand{\gbd}{{\boldsymbol{g}}}\newcommand{\Gbd}{{\boldsymbol{G}}}
\newcommand{\hbd}{{\boldsymbol{h}}}\newcommand{\Hbd}{{\boldsymbol{H}}}
\newcommand{\ibd}{{\boldsymbol{i}}}\newcommand{\Ibd}{{\boldsymbol{I}}}
\newcommand{\jbd}{{\boldsymbol{j}}}\newcommand{\Jbd}{{\boldsymbol{J}}}
\newcommand{\kbd}{{\boldsymbol{k}}}\newcommand{\Kbd}{{\boldsymbol{K}}}
\newcommand{\lbd}{{\boldsymbol{l}}}\newcommand{\Lbd}{{\boldsymbol{L}}}
\newcommand{\mbd}{{\boldsymbol{m}}}\newcommand{\Mbd}{{\boldsymbol{M}}}
\newcommand{\mubd}{{\boldsymbol{\mu}}}
\newcommand{\nbd}{{\boldsymbol{n}}}\newcommand{\Nbd}{{\boldsymbol{N}}}
\newcommand{\nubd}{{\boldsymbol{\nu}}}
\newcommand{\Nalpha}{{\boldsymbol{N_\alpha}}}
\newcommand{\Nbeta}{{\boldsymbol{N_\beta}}}
\newcommand{\Ngamma}{{\boldsymbol{N_\gamma}}}
\newcommand{\obd}{{\boldsymbol{o}}}\newcommand{\Obd}{{\boldsymbol{O}}}
\newcommand{\pbd}{{\boldsymbol{p}}}\newcommand{\Pbd}{{\boldsymbol{P}}}
\newcommand{\phibd}{{\boldsymbol{\phi}}}
\newcommand{\phihatbd}{{\boldsymbol{\hat{\phi}}}}
\newcommand{\psibd}{{\boldsymbol{\psi}}}
\newcommand{\psihatbd}{{\boldsymbol{\hat{\psi}}}}
\newcommand{\qbd}{{\boldsymbol{q}}}\newcommand{\Qbd}{{\boldsymbol{Q}}}
\newcommand{\rbd}{{\boldsymbol{r}}}\newcommand{\Rbd}{{\boldsymbol{R}}}
\newcommand{\rhobd}{{\boldsymbol{\rho}}}
\newcommand{\sbd}{{\boldsymbol{s}}}\newcommand{\Sbd}{{\boldsymbol{S}}}
\newcommand{\tbd}{{\boldsymbol{t}}}\newcommand{\Tbd}{{\boldsymbol{T}}}
\newcommand{\taubd}{{\boldsymbol{\tau}}}
\newcommand{\ubd}{{\boldsymbol{u}}}\newcommand{\Ubd}{{\boldsymbol{U}}}
\newcommand{\vbd}{{\boldsymbol{v}}}\newcommand{\Vbd}{{\boldsymbol{V}}}
\newcommand{\varphibd}{{\boldsymbol{\varphi}}}
\newcommand{\wbd}{{\boldsymbol{w}}}\newcommand{\Wbd}{{\boldsymbol{W}}}
\newcommand{\xbd}{{\boldsymbol{x}}}\newcommand{\Xbd}{{\boldsymbol{X}}}
\newcommand{\xibd}{{\boldsymbol{\xi}}}
\newcommand{\ybd}{{\boldsymbol{y}}}\newcommand{\Ybd}{{\boldsymbol{Y}}}
\newcommand{\zbd}{{\boldsymbol{z}}}\newcommand{\Zbd}{{\boldsymbol{Z}}}

\newcommand{\abar}{{\bar{a}}}\newcommand{\Abar}{{\bar{A}}}
\newcommand{\Amcbar}{{\bar{\Amc}}}
\newcommand{\bbar}{{\bar{b}}}\newcommand{\Bbar}{{\bar{B}}}
\newcommand{\bbdbar}{{\bar{\bbd}}}
\newcommand{\betabar}{{\bar{\beta}}}
\newcommand{\betabdbar}{{\bar{\betabd}}}
\newcommand{\cbar}{{\bar{c}}}\newcommand{\Cbar}{{\bar{C}}}
\newcommand{\ebar}{{\bar{e}}}\newcommand{\Ebar}{{\bar{E}}}
\newcommand{\ebdbar}{{\bar{\ebd}}}
\newcommand{\Embbar}{{\bar{\Emb}}}
\newcommand{\Ebfbar}{{\bar{\Ebf}}}
\newcommand{\etabar}{{\bar{\eta}}}
\newcommand{\fbar}{{\bar{f}}}\newcommand{\Fbar}{{F}}
\newcommand{\Fmcbar}{{\bar{\Fmc}}}
\newcommand{\gbar}{{\bar{g}}}\newcommand{\Gbar}{{\bar{G}}}
\newcommand{\Gammabar}{{\bar{\Gamma}}}
\newcommand{\Gmcbar}{{\bar{\Gmc}}}
\newcommand{\Hbar}{{\bar{H}}}
\newcommand{\ibar}{{\bar{i}}}\newcommand{\Ibar}{{\bar{I}}}
\newcommand{\jbar}{{\bar{j}}}\newcommand{\Jbar}{{\bar{J}}}
\newcommand{\Jmcbar}{{\bar{\Jmc}}}
\newcommand{\kbar}{{\bar{k}}}\newcommand{\Kbar}{{\bar{K}}}
\newcommand{\lbar}{{\bar{l}}}\newcommand{\Lbar}{{\bar{L}}}
\newcommand{\mbar}{{\bar{m}}}\newcommand{\Mbar}{{\bar{M}}}
\newcommand{\mubar}{{\bar{\mu}}}
\newcommand{\Mmbbar}{{\bar{\Mmb}}}
\newcommand{\nbar}{{\bar{n}}}\newcommand{\Nbar}{{\bar{N}}}
\newcommand{\nubar}{{\bar{\nu}}}
\newcommand{\obar}{{\bar{o}}}\newcommand{\Obar}{{\bar{O}}}
\newcommand{\omegabar}{{\bar{\omega}}}
\newcommand{\Omegabar}{{\bar{\Omega}}}
\newcommand{\pbar}{{\bar{p}}}\newcommand{\Pbar}{{\bar{P}}}
\newcommand{\Pbdbar}{{\bar{\Pbd}}}
\newcommand{\Phibar}{{\bar{\Phi}}}
\newcommand{\pibar}{{\bar{\pi}}}
\newcommand{\Pmbbar}{{\bar{\Pmb}}}
\newcommand{\Pmcbar}{{\bar{\Pmc}}}
\newcommand{\psibar}{{\bar{\psi}}}
\newcommand{\qbar}{{\bar{q}}}\newcommand{\Qbar}{{\bar{Q}}}
\newcommand{\rbar}{{\bar{r}}}\newcommand{\Rbar}{{\bar{R}}}
\newcommand{\Rmcbar}{{\bar{\Rmc}}}
\newcommand{\sbar}{{\bar{s}}}\newcommand{\Sbar}{{\bar{S}}}
\newcommand{\sigmabar}{{\bar{\sigma}}}
\newcommand{\Smcbar}{{\bar{\Smc}}}
\newcommand{\tbar}{{\bar{t}}}\newcommand{\Tbar}{{\bar{T}}}
\newcommand{\taubar}{{\bar{\tau}}}
\newcommand{\taubdbar}{{\bar{\taubd}}}
\newcommand{\Thetabar}{{\bar{\Theta}}}
\newcommand{\thetabar}{{\bar{\theta}}}
\newcommand{\ubar}{{\bar{u}}}\newcommand{\Ubar}{{\bar{U}}}
\newcommand{\varphibar}{{\bar{\varphi}}}
\newcommand{\vbar}{{\bar{v}}}\newcommand{\Vbar}{{\bar{V}}}
\newcommand{\wbar}{{\bar{w}}}\newcommand{\Wbar}{{\bar{W}}}
\newcommand{\xbar}{{\bar{x}}}\newcommand{\Xbar}{{\bar{X}}}
\newcommand{\xibar}{{\bar{\xi}}}
\newcommand{\Xitbar}{{\bar{X}_t^i}}
\newcommand{\Xjtbar}{{\bar{X}_t^j}}
\newcommand{\Xktbar}{{\bar{X}_t^k}}
\newcommand{\Xisbar}{{\bar{X}_s^i}}
\newcommand{\ybar}{{\bar{y}}}\newcommand{\Ybar}{{\bar{Y}}}
\newcommand{\zbar}{{\bar{z}}}\newcommand{\Zbar}{{\bar{Z}}}
\newcommand{\zetabar}{{\bar{\zeta}}}

\newcommand{\Ahat}{{\hat{A}}}
\newcommand{\bhat}{{\hat{b}}}\newcommand{\Bhat}{{\hat{B}}}
\newcommand{\Chat}{{\hat{C}}}
\newcommand{\etahat}{{\hat{\eta}}}
\newcommand{\fhat}{{\hat{f}}}
\newcommand{\ghat}{{\hat{g}}}
\newcommand{\hhat}{{\hat{h}}}
\newcommand{\Jhat}{{\hat{J}}}
\newcommand{\Jmchat}{{\hat{\Jmc}}}
\newcommand{\lambdahat}{{\hat{\lambda}}}
\newcommand{\lhat}{{\hat{l}}}
\newcommand{\muhat}{{\hat{\mu}}}
\newcommand{\nuhat}{{\hat{\nu}}}
\newcommand{\phihat}{{\hat{\phi}}}
\newcommand{\psihat}{{\hat{\psi}}}
\newcommand{\rhohat}{{\hat{\rho}}}
\newcommand{\Smchat}{{\hat{\Smc}}}
\newcommand{\Tmchat}{{\hat{\Tmc}}}
\newcommand{\Xhat}{{\hat{X}}}
\newcommand{\Yhat}{{\hat{Y}}}

\newcommand{\atil}{{\tilde{a}}}\newcommand{\Atil}{{\tilde{A}}}
\newcommand{\btil}{{\tilde{b}}}\newcommand{\Btil}{{\tilde{B}}}
\newcommand{\ctil}{{\tilde{c}}}\newcommand{\Ctil}{{\tilde{C}}}
\newcommand{\dtil}{{\tilde{d}}}\newcommand{\Dtil}{{\tilde{D}}}
\newcommand{\etil}{{\tilde{e}}}\newcommand{\Etil}{{\tilde{E}}}
\newcommand{\ebdtil}{{\tilde{\ebd}}}
\newcommand{\etatil}{{\tilde{\eta}}}
\newcommand{\Embtil}{{\tilde{\Emb}}}
\newcommand{\ftil}{{\tilde{f}}}\newcommand{\Ftil}{{\tilde{F}}}
\newcommand{\Fmctil}{{\tilde{\Fmc}}}
\newcommand{\gtil}{{\tilde{g}}}\newcommand{\Gtil}{{\tilde{G}}}
\newcommand{\gammatil}{{\tilde{\gamma}}}
\newcommand{\Gmctil}{{\tilde{\Gmc}}}
\newcommand{\htil}{{\tilde{h}}}\newcommand{\Htil}{{\tilde{H}}}
\newcommand{\itil}{{\tilde{i}}}\newcommand{\Itil}{{\tilde{I}}}
\newcommand{\jtil}{{\tilde{j}}}\newcommand{\Jtil}{{\tilde{J}}}
\newcommand{\Jmctil}{{\tilde{\Jmc}}}
\newcommand{\ktil}{{\tilde{k}}}\newcommand{\Ktil}{{\tilde{K}}}
\newcommand{\ltil}{{\tilde{l}}}\newcommand{\Ltil}{{\tilde{L}}}
\newcommand{\mtil}{{\tilde{m}}}\newcommand{\Mtil}{{\tilde{M}}}
\newcommand{\mutil}{{\tilde{\mu}}}
\newcommand{\ntil}{{\tilde{n}}}\newcommand{\Ntil}{{\tilde{N}}}
\newcommand{\nutil}{{\tilde{\nu}}}
\newcommand{\nubdtil}{{\tilde{\nubd}}}
\newcommand{\otil}{{\tilde{o}}}\newcommand{\Otil}{{\tilde{O}}}
\newcommand{\omegatil}{{\tilde{\omega}}}
\newcommand{\Omegatil}{{\tilde{\Omega}}}
\newcommand{\ptil}{{\tilde{p}}}\newcommand{\Ptil}{{\tilde{P}}}
\newcommand{\Pmbtil}{{\tilde{\Pmb}}}
\newcommand{\phitil}{{\tilde{\phi}}}
\newcommand{\pitil}{{\tilde{\pi}}}
\newcommand{\psitil}{{\tilde{\psi}}}
\newcommand{\qtil}{{\tilde{q}}}\newcommand{\Qtil}{{\tilde{Q}}}
\newcommand{\rtil}{{\tilde{r}}}\newcommand{\Rtil}{{\tilde{R}}}
\newcommand{\rbdtil}{{\tilde{\rbd}}}
\newcommand{\Rmctil}{{\tilde{\Rmc}}}
\newcommand{\rhotil}{{\tilde{\rho}}}
\newcommand{\stil}{{\tilde{s}}}\newcommand{\Stil}{{\tilde{S}}}
\newcommand{\sigmatil}{{\tilde{\sigma}}}
\newcommand{\tautil}{{\tilde{\tau}}}
\newcommand{\ttil}{{\tilde{t}}}\newcommand{\Ttil}{{\tilde{T}}}
\newcommand{\Tmctil}{{\tilde{\Tmc}}}
\newcommand{\thetatil}{{\tilde{\theta}}}
\newcommand{\util}{{\tilde{u}}}\newcommand{\Util}{{\tilde{U}}}
\newcommand{\vtil}{{\tilde{v}}}\newcommand{\Vtil}{{\tilde{V}}}
\newcommand{\Vmctil}{{\tilde{\Vmc}}}
\newcommand{\varphitil}{{\tilde{\varphi}}}
\newcommand{\wtil}{{\tilde{w}}}\newcommand{\Wtil}{{\tilde{W}}}
\newcommand{\xtil}{{\tilde{x}}}\newcommand{\Xtil}{{\tilde{X}}}
\newcommand{\xitil}{{\tilde{\xi}}}
\newcommand{\Xittil}{{\tilde{X}^{i}_t}}
\newcommand{\Xistil}{{\tilde{X}^{i}_s}}
\newcommand{\ytil}{{\tilde{y}}}\newcommand{\Ytil}{{\tilde{Y}}}
\newcommand{\ztil}{{\tilde{z}}}\newcommand{\Ztil}{{\tilde{Z}}}
\newcommand{\zetatil}{{\tilde{\zeta}}}
\newcommand{\Zittil}{{\tilde{Z}^{i,N}_t}}
\newcommand{\Zistil}{{\tilde{Z}^{i,N}_s}}

\newcommand{\newset}{\mathcal{A}}
\newcommand{\newsetRW}{F}
\newcommand{\newdomain}{E}





\begin{frontmatter}
\title {Many-Server Asymptotics for Join-the-Shortest-Queue: Large Deviations and Rare Events 	
}

 \runtitle{LDP for JSQ in the Many-Server Limit}

\begin{aug}
\author{Amarjit Budhiraja\thanks{Research  supported in part by the National Science Foundation (DMS-1814894, DMS-1853968).}, Eric Friedlander, and Ruoyu Wu \\ \ \\
}
\end{aug}

March 19, 2020

\skp

\begin{abstract}
	The Join-the-Shortest-Queue routing policy is studied in an asymptotic regime where the number of processors $n$ scales with the arrival rate.
	A large deviation principle (LDP) for the occupancy process is established, as $n\to \infty$, in a suitable infinite-dimensional path space. 
	Model features that present technical challenges include,  Markovian dynamics with discontinuous statistics,  a diminishing rate property of the transition probability rates, and an infinite-dimensional state space. 
	The difficulty is in the proof of the Laplace lower bound which requires establishing the uniqueness of solutions of certain infinite-dimensional systems of controlled ordinary differential equations. 
	The LDP gives information on the rate of decay of probabilities of various types of rare events associated with the system.
	We illustrate this by establishing explicit exponential decay rates for probabilities of long queues. 
	In particular, denoting by 
	$E_j^n(T)$  the event that there is at least one queue with $j$ or more jobs at some time instant 
	over $[0,T]$, we show that, in the critical case, for large $n$ and $T$,
	$\Pmb(E^n_j(T)) \approx \exp\left [-\frac{n (j-2)^2}{4T}\right].$
	\\
	
\noi {\bf AMS 2000 subject classifications:} 60F10, 90B15, 91B70, 60J75, 34H05

\noi{\bf Keywords:} large deviations, load balancing, discontinuous statistics, diminishing rates, JSQ,  jump-Markov processes in infinite dimensions, calculus of variations, infinite-dimensional Skorokhod problem, golden ratio.
\end{abstract}

\end{frontmatter}

\section{Introduction}
Consider a system of $n$ parallel processors, each processing jobs  in its queue at rate $1$.
Jobs enter the system at rate $n\lambda_n$ with $\lambda_n\to\lambda \in (0,\infty)$ as $n\to\iy$.
Service times and inter-arrival times are exponentially distributed and are mutually independent.
Upon arriving, each job is routed to the shortest available queue by a central dispatcher. This is known as the 
Join-the-Shortest-Queue (JSQ) routing policy and is a popular model for load balancing among distributed resources in parallel-processing systems that arise in applications of cloud computing, file transfers, database look-ups etc.\ (see the survey article \cite{van2018scalable} and references therein). Denote by $X^n_i(t)$ the proportion of queues at time instant $t$ with $i$ or more jobs. 
This occupancy process
$X^n(t) \doteq (X^n_1(t),X^n_2(t), \cdots)$ is a convenient state descriptor for this system. In this work we establish a large deviation principle (LDP) for $X^n$ in the path space $\Dmb_{\Rmb^{\infty}}$, where for a Polish space $S$,  $\Dmb_S$  denotes the space of all maps from $[0,T]$ to $S$ that are right continuous and have left limits, equipped with the usual 
Skorokhod topology. 
This result gives a characterization of exponential decay rates for events of the form $\Pmb(X^n\in \GG)$, where $\GG$ is a 
suitable Borel set in $\Dmb_{\Rmb^{\infty}}$, in terms of the associated rate function (see Theorem \ref{thm:mainResult}
for a precise statement). 
The rate function takes a variational form and is given as the value function of an infinite-dimensional  deterministic optimal control problem (see \eqref{eqn:JSQRateFunction}). 

In general this control problem is intractable and in order to obtain useful information, beyond the fact that certain probabilities of interest converge to $0$ at an exponential rate, one needs  approximations, e.g.\ by computing costs for sub-optimal control actions. 
Nevertheless, for some events of interest, one can say more. We illustrate this by studying the decay rate of probabilities of long queues. For this, we restrict attention to the critical case $\lambda_n\to1$ and initial configuration $X^n_j(0) = \one_{\{j=1\}}$ (i.e.\ all queues are length 1 at time 0). Consider the set $E_j^n(T)$ that represents the event that there is at least one queue with $j$ or more jobs at some time instant 
over $[0,T]$. In Theorem \ref{thm:largeqs} we give an explicit characterization of the exponential decay rate of the probability of such events for $j\ge 3$.
In particular, when $j=3$ and $T=1$ (or, more generally, when $j-2=T$), we obtain the following formula (see \eqref{eq:eq332}) in terms of the golden ratio
$$\Pmb(E^n_j(T)) \approx \exp\left [-n T\left(\ell\left( \frac{1+\sqrt{5}}{2} \right) + 
\ell\left( \frac{-1+\sqrt{5}}{2} \right) \right)\right], \mbox{ for large } n,$$
where $\ell(x) \doteq x\log x - x +1$ for $x\ge 0$.
For long time horizons, the decay rates take an even more simple form, namely we show that for large $n$ and large $T$,
\begin{equation}\label{eq:asympform}
	\Pmb(E^n_j(T)) \approx \exp\left [-\frac{n (j-2)^2}{4T}\right].\end{equation}
Although not pursued in this work, techniques used to establish the above explicit asymptotic rates are expected to be useful for other types of events as well, see Conjecture \ref{rem:otherevent}. 

There are several technical challenges in establishing the LDP on the path space (namely Theorem \ref{thm:mainResult}). These stem from three key features of the model that are illustrated in Figure~\ref{fig1}.\\
\begin{figure}
	\includegraphics[width=.6\textwidth]{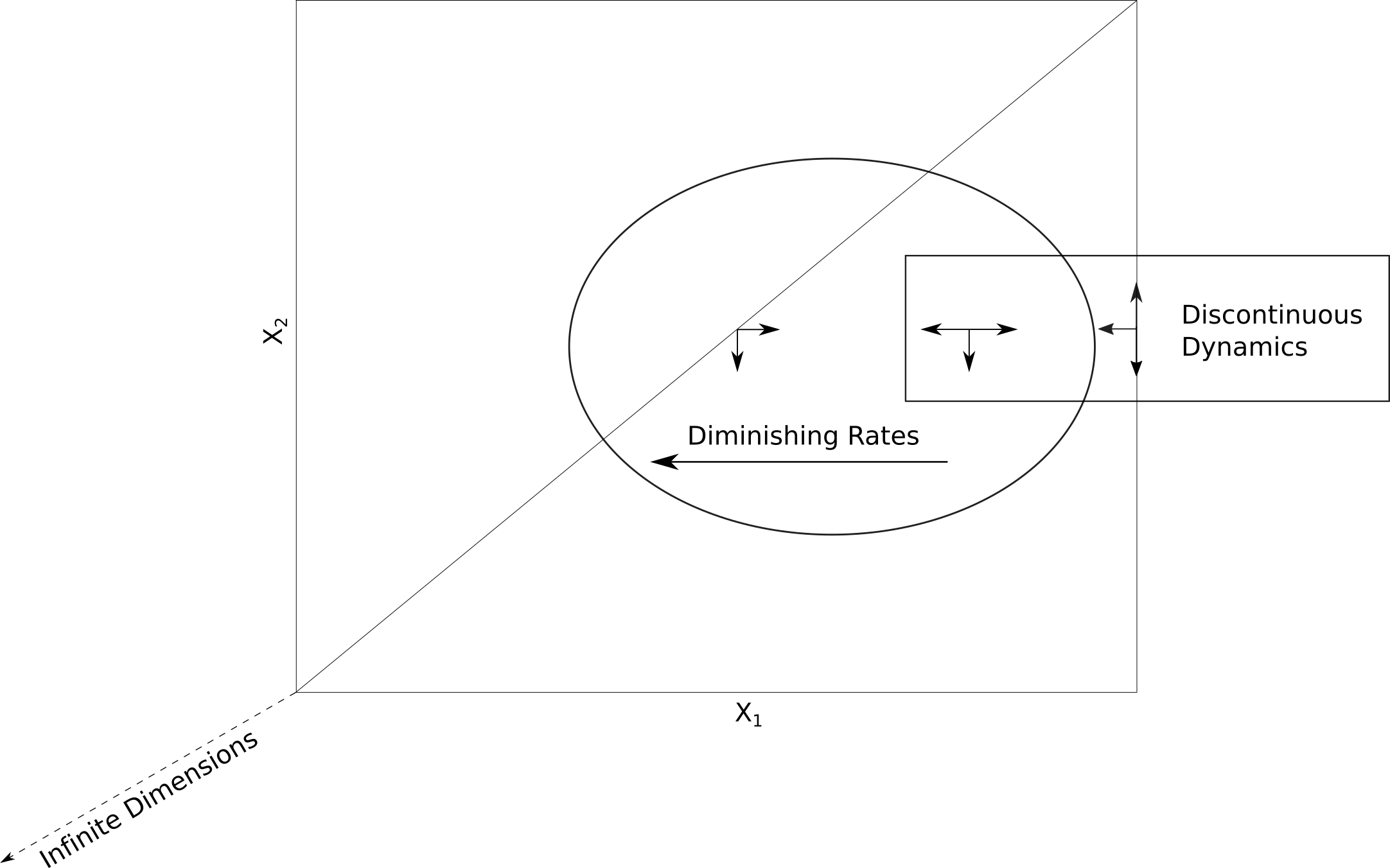}
	\caption{Three key features of the model: infinite dimensions, discontinuous dynamics, diminishing rates.}\label{fig1}
\end{figure}
\begin{addmargin}[1em]{2em}
{\em Infinite-dimensional dynamics.}  The state process $X^n$ is an infinite-dimensional Markov process and there is a non-trivial coupling between the different coordinates of the process. As a consequence, the associated rate function in the large deviation analysis is characterized through an infinite-dimensional control problem. 
Some recent works that have studied large deviation properties of  jump-Markov processes in infinite dimensions include \cite{BhamidiBudhirajaDupuisWu2019rare, budhiraja2013large, budwuptrf}.

\noindent {\em Discontinuous dynamics.}
The model considered here falls in the class of Markov processes with discontinuous statistics. Roughly speaking, this means that the transition rates change discontinuously at the interface of different regions of the state space. Large deviation analysis of such systems is technically challenging and has been the focus of several works \cite{alahaj, atadup1, dupell1, dupell2, dupell3, ignatiouk-robert2000, ign2, puhvla}.
In the current work an additional challenge is that there are infinitely many regions across which transition behavior changes discontinuously.

\noindent {\em Diminishing rates.} In the study of large deviation properties of Markov processes for which the transition probability rates decrease (continuously) to zero along some directions, one is led to local rate functions that have poor regularity properties. Some works that have treated large deviation properties of such systems include \cite{tib1, dupramwu, leo1, agazzi2018, BhamidiBudhirajaDupuisWu2019rare}. The model considered in the current work has similar features, in fact the setting here is more challenging in that the process $X^n$ may switch among infinitely many regions in which rates diminish along different directions.
\end{addmargin}

We note that all the papers referenced above, except \cite{BhamidiBudhirajaDupuisWu2019rare}, include only one of the three features noted above, while the paper 
\cite{BhamidiBudhirajaDupuisWu2019rare} has the first and third feature but not the second.  The combination of the three features described above is the main technical challenge in this work and  most of the arguments in Section \ref{sec:LowerBound}, which is the heart of this work,  are devoted to overcoming these  challenges in establishing a certain uniqueness property.

A LDP for a JSQ system has been obtained in \cite{puhvla} (see also \cite{ridshw}). 
However, the scaling regime considered in these works is very different from the one of interest here. 
Specifically, they consider a setting with a \textit{fixed} number of queues for which the arrival and service rates are scaled up by a factor of $n$, and the LDP is established for the finite dimensional queue length process scaled down by a factor of $\frac{1}{n}$. 
In this regime, \cite{puhvla} is in fact able to allow general arrival time distributions, different service time distribution parameters for different queues, and a weighted version of the JSQ policy. 
For the scaling regime considered here, in which the arrival rate and number of queues $n$ approach infinity, we establish a LDP  for the occupancy process (equivalent to the empirical distribution) of $n$ queue length processes.
While sacrificing some of the generality of \cite{puhvla}, this approach provides insights into questions and situations not addressed in \cite{puhvla}.
Specifically, the results here shed light on how probabilities of rare events decay as the size of the system is increased and the load on the system remains constant.
Suppose, for example, one is interested in deciding the size of a JSQ system (i.e.\ the number  of servers $n$), in which each queue has a finite buffer of size $k$, so that the system experiences an overflow event over the time horizon $[0,T]$ (namely, over the interval $[0,T]$, an arrival occurs to at least one of the  queues with a full buffer) with probability at most $e^{-6}$. Then, 
in the critical case $\lambda_n\to1$ and for an initial configuration where all queues are length 1 at time 0,
a calculation based on the asymptotic formula \eqref{eq:asympform} tells us that, roughly, one should take $n \approx 24T/(k-1)^2$. Such questions cannot be readily analyzed from the LDP  for a fixed size JSQ system established in \cite{puhvla}. Indeed, since the queue lengths in the analysis of \cite{puhvla} are scaled down by a factor of $n$, the results there will say that, for any fixed size system, the asymptotic (under their scaling) probability that at least one queue will attain a length of $m$ is  the same for all $m>0$.
%
We also note that the proof techniques here are very different than those employed in \cite{puhvla}. 
In particular, as noted previously, unlike \cite{puhvla}, the state descriptor here is infinite-dimensional
and the transition rates approach zero in certain directions, which requires a careful analysis of an infinite-dimensional Skorokhod problem and a delicate analysis of an important uniqueness property.

Queuing systems with many parallel servers in the regime where the arrival rate scales with the number of servers have been studied extensively.
A significant portion of this body of work concerns the setting in which,  upon arrival, jobs join a global queue and a large pool of servers processes jobs from this queue in a FIFO fashion.
One of the first works on such queuing systems is by Halfin and Whitt \cite{halwhi1} in which the number of servers and arrival rate increase to infinity while the load $\rho_n$ approaches one from below such that $(1-\rho_n)\sqrt{n} \to \beta \in (0,\infty)$.
Thereafter, this type of scaling has been often referred to as the Halfin-Whitt regime. The setting considered in the current work is different from  Halfin-Whitt type queuing systems with a global queue and is motivated by applications in which load balancing is an important concern. In terms of analysis, the setting considered here requires tracking an infinite-dimensional state instead of the size of a single queue. Some of the works that have considered the asymptotics of a JSQ system under a scaling of the form considered in the current work are \cite{eschenfeldt2018join, mukboretal, braverman2018steady, banerjee2018join}. In particular, \cite{eschenfeldt2018join} proves a central limit theorem under the heavy traffic scaling $(1-\lambda_n)\sqrt{n} \to \beta \in (0,\infty)$ while \cite{mukboretal} gives a law of large numbers (LLN) result under the more general condition $\lambda_n \to \lambda \in (0,\infty)$.

Many other load balancing policies have been studied in the literature (see \cite{vvedenskaya1996queueing, mukherjee2015universality,mitzenmacher2001power,bramson2012asymptotic,stolyar2015pull,   graham2000chaoticity,gupta2017load,budhiraja2017diffusion} and refereces therein) and a survey of recent advances can be found in \cite{van2018scalable}. In particular, \cite{mukherjee2015universality} considers the JIQ routing policy in which an incoming job is routed to an idle server, if available, and according to another routing policy (e.g.\ uniformly at random) if there are no idle servers in the system.
The authors use a coupling argument to show that, under the heavy traffic condition of \cite{eschenfeldt2018join}, JIQ and JSQ behave the same under the diffusion scaling.
While JSQ and JIQ look similar in the LLN limit and under the diffusion scaling, the statistical tail behavior of the two systems is expected to be quite different. 
\begin{figure}
	\includegraphics[width=.5\textwidth]{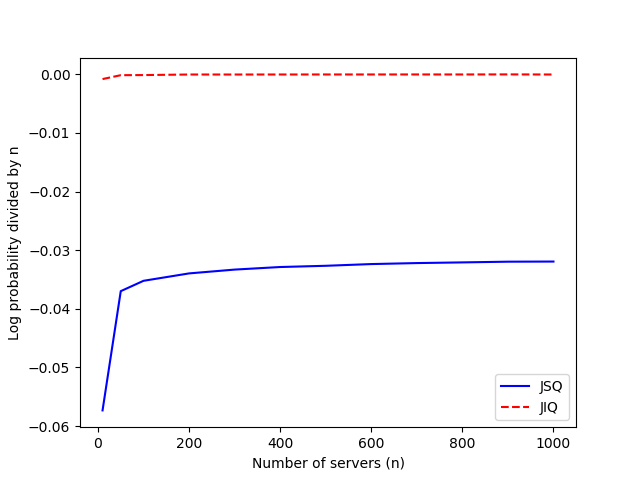}
	\caption{Monte Carlo Estimates for the rate of decay of the probability of $E^n_3(10)$ with $\lambda_n=.99$ for JSQ and JIQ starting with all queues of length 1.}\label{fig2}
\end{figure}
For example, Figure \ref{fig2} gives Monte Carlo  estimates of $\frac{1}{n}\log\PP(E^n_3(10))$ for JSQ and JIQ with $\lambda_n=.99$, starting with all queues of length 1. As is clear from this figure, the performance differences between the two policies are more clearly revealed when  one considers extreme tail events viewed under a large deviation scaling (see \cite{gupta2017load} for an alternative point of view for differentiating the performance of JSQ and JIQ systems). The LDP established in this work characterizes the tail statistical behavior of the JSQ system.
Establishing a similar result for the JIQ system presents significant new challenges that arise from the analysis of events that include time instants where no idle servers are present. This study will be taken up in a future work. Although not considered here, the large deviation principle given in this work also provides a starting point for developing efficient importance sampling schemes for estimating probabilities of rare events in a fixed size JSQ system (see also Remark \ref{rem:opttraj}).  

We now  comment on the proof idea of our main result, Theorem \ref{thm:mainResult}. 
The starting point of our analysis is to introduce a convenient representation for the evolution of the state process $X^n$. 
This is given in \eqref{eqn:stateProc1p}--\eqref{eqn:stateProc2p} through an infinite collection of i.i.d.\ Poisson random measures (PRM) where each of the PRM corresponds to a stream of events (i.e.\ arrivals or job completions).
The use of PRM  allows us to introduce controls on the rates of these events.
Roughly speaking, one can view these controls as pushing the state process away from its LLN limit to a ``rare-event trajectory'' while
the cost incurred by the control for perturbing the state process will determine the exponential decay rate of probabilities for such rare paths.
The state process has an equivalent and simpler description given in \eqref{eq:equivrepn} through which it can be viewed as a solution of an infinite-dimensional
Skorokhod problem for a {\em free process} $Y^n$ with sample paths in $\Dmb_{\Rmb^{\infty}}$. 
The existence and uniqueness of solutions
of this Skorokhod problem is established in Section \ref{sec:skorokhod}. 
The main results of this work, Theorem \ref{thm:mainResult},  will not only give a LDP for $X^n$ but in fact for the pair $(X^n, Y^n)$ in $\Dmb_{\Rmb^{\infty}\times \Rmb^{\infty}}$. 
For the proof of the LDP we consider its equivalent formulation in terms of a Laplace principle (see e.g.\ \cite{dupuis2011weak}). 
Specifically, to prove the LDP in Theorem \ref{thm:mainResult}, it suffices to show that every continuous and bounded function $G$ on the path space $\Dmb_{\Rmb^{\infty}\times \Rmb^{\infty}}$ satisfies the upper and lower bounds associated with this Laplace principle given in \eqref{eqn:LaplaceUpperBound} and \eqref{eqn:LaplaceLowerBound}, respectively, and the function $I_T$ in these equations has compact sub-level sets making it a rate function.

The key ingredient in establishing these results is a variational representation for expected values of exponential functionals of PRM established in \cite{budhiraja2011variational}. 
This result can be applied to give variational formulas for the expected values on the left sides of  the Laplace principle bounds, \eqref{eqn:LaplaceUpperBound} and \eqref{eqn:LaplaceLowerBound}.
These formulas are given in terms of controlled analogues of the PRM and state process.
Recalling that the uncontrolled PRM correspond to streams of job arrivals or completions, each controlled PRM is constructed using a control process that suitably modulates the rate at which the corresponding stream of events occurs.
Using controls $\varphi = (\varphi_i)_{i\in \Nmb_0}$, where $\varphi_i$ is the control associated with the $i$-th PRM, the controlled state process $(\bar X^n, \bar Y^n)$ is defined through \eqref{eq:maincontproc}.
The variational representation from  \cite{budhiraja2011variational} allows us to express the expected values of interest in 
\eqref{eqn:LaplaceUpperBound} and \eqref{eqn:LaplaceLowerBound} in terms of an infimum of costs associated with these controls and the associated controlled state process (see Lemma \ref{lem:varRep}).
The utility of this representation is twofold.
First, it reduces the majority of the proof of the LDP to arguing tightness and characterization of limits of sequences of carefully chosen controls and controlled state processes.
In particular, the limits points $(\zeta, \psi)$ of the controlled state processes can be characterized
as solutions to a system of controlled ordinary differential equations (ODE), expressed in \eqref{eq:psi1}--\eqref{eq:psii} driven by limit control processes $\{\varphi_i\}$.
Second, the representation hints at the form of the rate function $I_T$ for the LDP.
In view of the above characterization of limiting controlled state processes and in comparing the infimum in Lemma \ref{lem:varRep} to the desired right side of Laplace asymptotics in \eqref{eqn:LaplaceUpperBound} and \eqref{eqn:LaplaceLowerBound}, namely,
$$\inf_{(\zeta,\psi)\in \Cmc_T}\{I_T(\zeta,\psi)+G(\zeta,\psi)\}$$
where the space $\Cmc_T$ of trajectories in $\Dmb_{\Rmb^{\infty}\times \Rmb^{\infty}}$ is described in Section \ref{sec:rateFunction}, 
 it is natural to conjecture the following form for the rate function evaluated on a given pair of trajectories
 $(\zeta,\psi)\in \Cmc_T$. Consider the class $\cls_T(\zeta, \psi)$ of all controls $\{\varphi_i\}$ for which
 the given pair $(\zeta,\psi)$ solves the controlled ODE in \eqref{eq:psi1}--\eqref{eq:psii}. Then the rate function, $I_T(\zeta,\psi)$
 suggested by the above considerations is the infimum of the cost, 
$$\sum_{i=0}^\iy\int_{[0,T]\times[0,1]}\vartheta_i\ell(\varphi_i(s,y))\,ds\,dy,$$
over all controls $\{\varphi_i\} \in \cls_T(\zeta, \psi)$, where $\vartheta_i$ is defined above \eqref{eqn:JSQRateFunction}.

In order to prove the upper bound we consider, for each $n$, a near-optimal control and controlled process for the infimum in Lemma
\ref{lem:varRep}, establish the tightness of the sequence of such  processes in a suitable space, and characterize the weak limit points of the sequence. This is done in Sections \ref{sec:tight} -- \ref{sec:weakconv}. From these properties the Laplace upper bound in \eqref{eqn:LaplaceUpperBound} follows by standard arguments that use Fatou's lemma and certain lower semicontinuity properties, as shown
in Section \ref{sec:UpperBound}.
The proof that the function $I_T$ has compact sublevel sets and, thus, is a rate function has many similarities to the proof of the Laplace upper bound and is provided in Section \ref{sec:compactSets}.

The main technical challenge in this work is in the proof of the Laplace lower bound \eqref{eqn:LaplaceLowerBound}. 
For this, one starts with the variational expression on the right side of the inequality, namely,
$$\inf_{(\zeta,\psi)\in \Cmc_T}\{I_T(\zeta,\psi)+G(\zeta,\psi)\}.$$
The basic idea is to select a trajectory $(\zeta^*, \psi^*)$ in $\Cmc_T$ that is $\varepsilon$-optimal for the above infimum and then select a control
$\varphi^*\in \Smc_T(\zeta^*, \psi^*)$ driving this trajectory which is $\varepsilon$-optimal for the rate function evaluated at $(\zeta^*, \psi^*)$, i.e.\ $I_T(\zeta^*, \psi^*)$. 
In view of the variational representation of the Laplace functional of interest, given in Lemma \ref{lem:varRep},
one would like to construct a sequence of controlled state processes of the form in \eqref{eq:maincontproc} that converge to 
$(\zeta^*, \psi^*)$ such that the associated cost converges in an appropriate manner as well.  There is a natural choice for a sequence of controlled processes that one can attempt to implement for this purpose and it is relatively easy to show that this sequence of controlled processes has the needed tightness properties and that the weak limit points $(\bar \varphi, \bar \zeta, \bar \psi)$ of the controls and controlled state processes satisfy 
$\bar \varphi = \varphi^*  \in \Smc_T(\bar \zeta, \bar \psi)$.

However this is where one faces the main obstacle. From the above characterizations of the limit points it is not possible to deduce, in general, that $(\bar \zeta, \bar \psi) = (\zeta^*, \psi^*)$.
The issue is regarding the uniqueness of solutions of the infinite system of controlled ODE described by 
\eqref{eq:psi1}--\eqref{eq:psii} (considered with $(\zeta, \psi, \varphi)$ replaced with $(\zeta^*, \psi^*, \varphi^*)$).
Namely, if one is able to say that for a given $\varphi^*$ there is a unique pair $(\zeta^*, \psi^*)$ satisfying this system of equations then one obtains the desired result 
$(\bar \zeta, \bar \psi) = (\zeta^*, \psi^*)$. Showing uniqueness of such a system of equations is hard in general and in fact may not hold.
Most of Section \ref{sec:LowerBound} is devoted to overcoming this challenge. The main result is Lemma \ref{lem:uniqueness} which says
that one can replace $(\zeta^*, \psi^*)$ by a nearby pair of trajectories $(\zeta, \psi)$ for which the desired uniqueness property does hold with an appropriate near optimal control
$\varphi$. Key ingredients in the proof of this lemma are Lemmas \ref{lem:uniqueness_preparation_2}--\ref{lem:uniqueness_preparation}
and together these three lemmas, which perform several delicate surgeries on the original infinite dimensional trajectory $(\zeta^*, \psi^*)$
and the associated control $\varphi^*$, are at the technical heart of this work.  Additional comments on the proofs of these lemmas are given in Section \ref{sec:LowerBound}.

The paper is organized as follows.  Section \ref{sec:JSQmodel} introduces the state dynamics in terms of an infinite collection of PRM  and also gives an equivalent representation through an infinite-dimensional Skorokhod map. The properties of this map are studied in Section \ref{sec:skorokhod}. In preparation for the main result, in Section \ref{sec:rateFunction}, we introduce the rate function that governs the LDP.
The main result, Theorem \ref{thm:mainResult}, is then given in Section \ref{sec:mainResult}. This section also presents Theorem \ref{thm:largeqs} which gives our main result on exponential decay rates for probabilities of long queues as an illustration of applications of Theorem \ref{thm:mainResult}. Other possible applications of this result are discussed briefly in Conjecture \ref{rem:otherevent}. Section \ref{sec:varWeakCon} introduces the main variational representation that is the starting point of our analysis and establishes preliminary tightness and limit characterization results that are used in both the Laplace upper bound and lower bound proofs. Proof of the Laplace upper bound (i.e.\ \eqref{eqn:LaplaceUpperBound}) is completed in Section \ref{sec:UpperBound} while the lower bound (i.e.\ \eqref{eqn:LaplaceLowerBound}) is taken up in Section \ref{sec:LowerBound}.
Section \ref{sec:compactSets} shows that the function $I_T$ introduced in Section \ref{sec:rateFunction} is indeed a rate function. The results of Sections \ref{sec:UpperBound}, \ref{sec:LowerBound}, and \ref{sec:compactSets} together complete the proof of Theorem
\ref{thm:mainResult}. Finally Section \ref{sec:examples}  gives the proof of Theorem \ref{thm:largeqs}.

\subsection{Notation}
The following notation will be used. 
Fix $T < \infty$. All stochastic processes will be considered over the time horizon $[0,T]$.  
We denote the Lebesgue measure on a Euclidean space as $\leb$.
Let $S$ be a Polish space.
For a set $B\in S$ we denote the closure of $B$ as $\bar B$.
The Borel $\sigma$-field on $S$ will be denoted as $\clb(S)$.
Denote by $\Dmb_S$  the collection of all maps from $[0,T]$ to $S$ that are right continuous and have left limits. 
This space is equipped with the usual 
Skorokhod topology. 
Similarly $\Cmb_S$ is the space of all continuous maps from $[0,T]$ to $S$ equipped with the uniform topology. 
A sequence of $\Dmb_S$ valued random variables is said to be $\CC$-tight if it is tight in $\Dmb_S$ and any weak limit point takes values in $\Cmb_S$ a.s.
The space of all continuous and bounded real valued functions on $S$ will be denoted as $\CC_b(S)$.
For a bounded map $x: S \to \Rmb$, let $\|x\|_{\infty} \doteq \sup_{s \in S} |x(s)|$.

Let $(\XX, \|\cdot\|)$ be a metric space. 
We denote by $\XX^\iy$  the set of all sequence $x = \{x_i\}_{i\in\NN}$ such that $x_i\in\XX$ for all $i\in\NN$.
$\XX^\iy$ is equipped with the product topology, which is metrized with
\begin{equation*}
	d(x,y) \doteq \sum_{i=1}^\iy\frac{\|x_i-y_i\|\wedge 1}{2^i}.
\end{equation*}
Let $\ell(z) \doteq z\log(z)-z+1$ for $z \ge 0$.

\section{Model and Results}\label{sec:ModelRes}
In this section we will describe the model of interest and present our main results.
We begin by giving a precise mathematical formulation of the JSQ system in Section \ref{sec:JSQmodel}. The state process can equivalently be described through a certain infinite-dimensional Skorokhod map. This map is introduced and studied in 
Section \ref{sec:skorokhod}. 
In Section \ref{sec:rateFunction} we introduce the large deviation problem of interest and present the  rate function for the associated LDP.
 Section \ref{sec:mainResult}  presents  the main result of this work (Theorem \ref{thm:mainResult}) which, in particular, gives a large deviation principle for the queue occupancy process $X^n$ as the number of servers (and arrival rate) approaches infinity. This LDP can be used to extract information about probabilities of various types of rare events and in Theorem \ref{thm:largeqs} 
 we present one such application that provides estimates for probabilities of occurrence of ``large queues''.

\subsection{Model Description}\label{sec:JSQmodel}
Consider a system of $n$ parallel servers each maintaining its own  queue.
Jobs arrive in the system according to a Poisson process with rate $n\lambda_n$ where $\lambda_n \to\lambda$ for some $\lambda\in(0,\iy)$.
When a job enters the system, a central dispatcher queries each server and routes the job to the server with the shortest queue. If there are multiple shortest queues, then the tie is broken uniformly at random. Each server processes jobs in its queue using the FIFO protocol and the service times are exponential with mean $1$. We assume that the  inter-arrival times and service times are mutually independent.
The state of the system at time $t$ can be represented as $X^n(t) = (X^n_0(t), X^n_1(t),\ldots)$ where $X^n_i(t)$ corresponds to the proportion of queues which are of length $i$ or longer at time $t$.
Note that $X^n_i(t)\in[0,1]$ and $1=X^n_0(t)\geq X^n_1(t)\geq X^n_2(t)\geq \ldots$ for all $t\in[0,T]$.

We will now give a convenient evolution equation for the state process in terms of a collection of Poisson random measures. 
For a locally compact metric space $\SSS$, let $\clm_{FC}(\SSS)$ represent the space of measures $\nu$ on $(\SSS,\clb(\SSS))$ such that $\nu(K)<\iy$ for every compact $K\in\clb(\SSS)$, equipped with the usual vague topology.
This topology can be metrized such that $\clm_{FC}(\SSS)$ is a Polish space (see \cite{budhiraja2016moderate} for one convenient metric).
A PRM $D$ on $\SSS$ with mean measure (or intensity measure) $\nu\in\clm_{FC}(\SSS)$ is an $\clm_{FC}$-valued random variable such that for each $\GG\in\clb(\SSS)$ with $\nu(\GG)<\iy$, $D(\GG)$ is a Poisson random variable with mean $\nu(\GG)$ and for disjoint $\GG_1,\ldots,\GG_k\in\clb(\SSS)$, the random variables $D(\GG_1),\ldots, D(\GG_k)$ are mutually independent random variables (cf. \cite{IkedaWatanabe}).

Fix $T\in(0,\iy)$ and let $(\Om,\clf,\PP)$ be a complete probability space on which we are given a collection of i.i.d.\
Poisson random measures $\{D_k(ds\, dy\, dz)\}_{k\in\NN_0}$ on $[0,T]\times[0,1]\times\RR_+$ with intensity given by the Lebesgue measure.
Define the filtration $\{\hat{\clf}_t\}_{0\leq t\leq T}$ as
\begin{align*}
	\hat{\clf}_t\doteq\sigma\{D_k((0,s]\times \GG\times B),0\leq s\leq t, \GG\in\clb([0,1]), B\in\clb(\RR_+)\}
\end{align*}
and let $\{\clf_t\}_{0\leq t\leq T}$ be the $\PP$-augmentation of this filtration.
Using the above collection of PRM we now construct certain Point Processes with points in $[0,T]\times [0,1]$ as follows. 

Let $\bar{\clf}$ be the $\{\clf_t\}_{0\leq t\leq T}$-predictable $\sigma$-field on $\Om\times[0,T]$.
Denote by $\bar{\newset}_+$ the class of all $(\bar{\clf}\otimes \clb([0, 1]))/\clb(\RR_+)$-measurable
maps from $\Om\times[0, T]\times [0, 1]$ to $\RR_+$.
For $\varphi\in\bar{\newset}_+$ and each $k\in\NN_0$, define the counting process $D_k^\varphi$ on $[0, T]\times[0, 1]$ by
\begin{align*}
	D_k^\varphi([0, t] \times \GG)
	\doteq
	\int_{[0,t]\times \GG}\one_{[0,\varphi(s,y)]}(z) D_k(ds\, dy\, dz), \text{ for } t \in [0, T],\ \GG \in \clb([0, 1]).
\end{align*}
We regard $D_k^\varphi$ as a controlled random measure, where $\varphi$ is the control process that can be used to produce a
desired intensity. We will write $D_k^\varphi$ as $D^\theta_k$ if $\varphi = \theta$ for some constant $\theta\in\RR_+$.
In particular we will frequently take $\theta = n$.
Note that $D_k^\theta$ is a PRM on $[0, T] \times [0, 1]$ with intensity $\theta\, ds\times dy$.

By using $D_0$ to represent the arrival process and $D_i$ to represent the departure process from queues with $i$ customers, $i \in \Nmb$, we can now give the state evolution of  $X^n$ as follows,
\begin{align}
	X^n_1(t) &= X^n_1(0) + \frac{1}{n} \int_{[0,t]\times[0,1]} \one_{\{X^n_{1}(s-) < 1\}} D_{0}^{n\lambda_n}(ds\,dy)\label{eqn:stateProc1p}\\
	&\qquad - \frac{1}{n} \int_{[0,t]\times[0,1]} \one_{[0,X^n_1(s-)-X^n_2(s-))}(y) D_1^{n}(ds\,dy), \nonumber\\
	X^n_i(t) &= X^n_i(0) + \frac{1}{n} \int_{[0,t]\times[0,1]} \one_{\{X^n_{i-1}(s-) = 1, X^n_{i}(s-) < 1\}} D_{0}^{n\lambda_n}(ds\,dy)\label{eqn:stateProc2p}\\
	&\qquad - \frac{1}{n} \int_{[0,t]\times[0,1]} \one_{[0,X^n_i(s-)-X^n_{i+1}(s-))}(y) D_i^{n}(ds\,dy),\ i\geq 2. \nonumber
\end{align}

 The first integral on the right side of \eqref{eqn:stateProc1p} corresponds to incoming jobs that join an empty queue.  The indicator in the integral captures the fact that this can happen only when an empty queue is available. The second term on the right side of \eqref{eqn:stateProc1p} corresponds to completion of  jobs by a server with only one job in the queue.   The terms in equation \eqref{eqn:stateProc2p}
 are interpreted in an analogous manner.  By introducing {\em reflection terms}  $\eta_i^n$ defined by
\begin{align}
	\eta_i^n(t) & \doteq \frac{1}{n} \int_{[0,t]\times[0,1]} \one_{\{X^n_{i}(s-) = 1\}} D_{0}^{n\lambda_n}(ds\,dy),\ i\geq 1\label{eqn:stateProc3}
\end{align}
one can rewrite the state equation as follows.
Define a \emph{free process} $Y^n$ as
\begin{align}
	Y^n_1(t) & = X^n_1(0) + \frac{1}{n} \int_{[0,t]\times[0,1]} D_{0}^{n\lambda_n}(ds\,dy) - \frac{1}{n} \int_{[0,t]\times[0,1]} \one_{[0,X^n_1(s-)-X^n_2(s-))}(y) D_1^{n}(ds\,dy), \label{eqn:stateProc4} \\
	Y^n_i(t) & = X^n_i(0) - \frac{1}{n} \int_{[0,t]\times[0,1]} \one_{[0,X^n_i(s-)-X^n_{i+1}(s-))}(y) D_i^{n}(ds\,dy)\ i\geq 2.\label{eqn:stateProc5}
\end{align}
Then
\begin{equation}\label{eq:equivrepn}
\begin{aligned}
	X^n_1(t) & = Y^n_1(t) - \eta_1^n(t), \\
	X^n_i(t) & = Y^n_i(t) + \eta_{i-1}^n(t) - \eta_i^n(t), i \ge 2.
\end{aligned}
\end{equation}
Written in this manner, $X^n$ can be viewed as a solution of an infinite-dimensional Skorokhod problem as discussed in the next section.

\subsection{Skorokhod Problem}\label{sec:skorokhod}
We now introduce the  Skorokhod problem that is associated with the system of equations in the last section.
Consider an infinite matrix (namely a map from $\Nmb\times \Nmb$ to $\Rmb$), $R_\infty$, defined as
$$R_{\infty}(i,j) = -\one_{\{j=i\}} + \one_{\{j=i-1, i>1\}}, \mbox{ for } (i,j) \in \Nmb\times \Nmb.$$
Define $\Vmb \doteq (-\infty, 1]$ and consider a $M\in\NN$. 
Let $\Vmb^{\infty}$ (resp. $\Vmb^M$) denote the space of maps from $\Nmb$ (resp. $\{1,\ldots,M\}$) to $\Vmb$ which is equipped with the product topology.
The spaces $\Rmb^{\infty}$ and $\Rmb^M$ are similarly defined with $\Vmb$ replaced by $\Rmb$. 
Let $\Dmb^o_{\Rmb^{\infty}}$ be the subset of $\Dmb_{\Rmb^{\infty}}$ consisting of paths $\psi$ such that $\psi(0) \in \Vmb^{\infty}$.
\begin{definition}
	\label{def:Smap}
	Let $\psi \in \Dmb^o_{\Rmb^{\infty}}$. Then $(\phi,\eta) \in \Dmb_{\Vmb^{\infty} \times \Rmb^{\infty}}$ solves the Skorokhod problem (SP) for $\psi$ associated with  the reflection matrix $R_\infty$ if  the following hold:
	\begin{enumerate}[(i)]
	\item
		$\phi(t)=\psi(t)+R_\infty\eta(t)$ for all $t \in [0,T]$, namely 
		$$\phi_1(t)=\psi_1(t)-\eta_1(t),\;\; \phi_i(t)=\psi_i(t)+\eta_{i-1}(t)-\eta_i(t) \mbox{ for all } i \ge 2 \mbox{ and } t \in [0,T].$$
	\item For each $i \in \Nmb$,
		  (a) $\eta_i(0)=0$, (b) $\eta_i$ is nondecreasing, and (c) $\int_0^T \one_{\{\phi_i(s) < 1 \}}\,d\eta_i(s)=0$.
	\end{enumerate}
\end{definition}

On the domain $\newdomain \subset \Dmb^o_{\Rmb^{\infty}}$ on which there is a unique solution to the SP we define the Skorokhod map (SM) $\Gamma: \newdomain \to \Dmb_{\Vmb^{\infty}}$ as $\Gamma(\psi)=\phi$ if $(\phi,\eta)$ solves the  SP posed by $\psi$.
Also, define the map $\Gammabar \colon \newdomain \to \Dmb_{\Rmb^{\infty}}$ by $\Gammabar(\psi)=\eta$.
For $\psi, \tilde \psi \in \Dmb_{\Rmb^{\infty}}$, let
$$\|\psi - \tilde \psi\|_{\infty} \doteq \sum_{i=1}^\infty \frac{\|\psi_i-\tilde \psi_i\|_\infty }{2^i}.$$

The following result gives the wellposedness and regularity of the above infinite-dimensional Skorokhod problem.

\begin{lemma}
	\label{lem:SMap}
	The SP is well defined on all of $ \Dmb^o_{\Rmb^{\infty}}$ (namely $\newdomain= \Dmb^o_{\Rmb^{\infty}}$) and the SM is Lipschitz continuous in the following sense:
	For all $\psi,\psitil \in  \Dmb^o_{\Rmb^{\infty}}$,
	\begin{equation*}
		\|\Gamma(\psi)-\Gamma(\psi)\|_\infty \le 4 \|\psi-\psitil\|_\infty, \quad \|\Gammabar(\psi)-\Gammabar(\psi)\|_\infty \le 2 \|\psi-\psitil\|_\infty.
	\end{equation*}
\end{lemma}

\begin{proof}
	We first prove uniqueness.
	Fix $\psi \in \Dmb^o_{\Rmb^{\infty}}$.
	Suppose there are two solutions $(\phi,\eta)$ and $(\phitil,\etatil)$.
	For each $M \in \Nmb$, define the matrix $R_M=-I_{M \times M} + P_M$, where $I_{M \times M}$ is the $M \times M$ identity matrix and $P_M(i,j)=\one_{\{j=i-1, i>1\}}$.
	Let $\psi^M, \phi^M, \tilde \phi^M, \eta^M, \tilde \eta^M \in \Dmb_{\Rmb^M}$ be defined as
	$$(\psi^M_i, \phi^M_i, \tilde \phi^M_i, \eta^M_i, \tilde \eta^M_i) = (\psi_i, \phi_i, \tilde \phi_i, \eta_i, \tilde \eta_i), \; i = 1, \ldots M.$$
	Then $(\phi^M, \eta^M)$ solves the $M$-dimensional Skorokhod problem for $\psi^M$ associated with the domain $\Vmb^M$ and reflection matrix $R_M$, namely
	\begin{align*}
		\phi^M &= \psi^M + R_M \eta^M, \; \phi^M \in \Dmb_{\Vmb^M},\\
		 \eta(0)&=0, \eta \mbox{ is right continuous and nondecreasing}, \; \int_{0}^T \one_{\{\phi_i(s)<1\}} d \eta_i(s) =0.
		 \end{align*}
	Furthermore, $(\tilde \phi^M, \tilde \eta^M)$ also solves the same $M$-dimensional Skorokhod problem for $\psi^M$. 
	Since $P_M$ has spectral radius less than $1$, it is well known from \cite{HarrisonReiman1981reflected,DupuisIshii1991lipschitz}  that this $M$-dimensional SP has a unique solution.
	Thus we must have $(\phi^M, \eta^M) = (\tilde \phi^M, \tilde \eta^M)$.
	Since $M \in \Nmb$ is arbitrary, we have the desired uniqueness property.

	Next we prove existence.
	The solutions of the finite dimensional SP have the following consistency property:
	Suppose for some  $M \in \Nmb$, $(\phi^M,\eta^M)$ is the solution to the $M$-dimensional SP for $\psi^M$ (associated with $(\Vmb^M, R_M)$) and let $1 \le m \le M$.
	Define $(\psi^{m,M}, \phi^{m,M}, \eta^{m,M}) \in \Dmb_{\Rmb^{3m}}$ as
	$$(\psi^{m,M}_i, \phi^{m,M}_i, \eta^{m,M}_i) = (\psi^{M}_i, \phi^{M}_i, \eta^{M}_i), \; i = 1, \ldots m.$$
	Then $(\phi^{m,M}, \eta^{m,M})$ solves the $m$-dimensional SP for $\psi^m$ (associated with $(\Vmb^m, R_m)$).
	Fix $\psi \in \Dmb^o_{\Rmb^{\infty}}$ and for $M \in \Nmb$ define $\psi^M$ as before. Let $(\phi^M, \eta^M)$ be the solution
	of   the $M$-dimensional SP for $\psi^M$. Define  $(\phi,\eta) \in   \Dmb_{\Rmb^{\infty}\times \Rmb^{\infty}}$ as 
	 $(\phi_n,\eta_n) \doteq (\phi_n^n, \eta^n_n)$ for $n \in \Nmb$. From the consistency property noted above, $(\phi,\eta)$ is a solution to the infinite-dimensional SP.
	This gives existence.

	Finally we prove the Lipschitz property. Fix $\psi, \tilde \psi \in \Dmb^o_{\Rmb^{\infty}}$ and 
	let $(\phi,\eta)$ and $(\phitil,\etatil)$ be solutions to the infinite-dimensional SP for $\psi$ and $\psitil$ respectively.
	For each $M \in \Nmb$,  since
	\begin{align*}
		\phi_M(t)&=\psi_M(t)+\eta_{M-1}(t)-\eta_M(t), \\
		\phitil_M(t)&=\psitil_M(t)+\etatil_{M-1}(t)-\etatil_M(t),
	\end{align*}
	where we define $\eta_0 = \etatil_0=0$,
	it follows that $(\phi_M,\eta_M)$ and $(\phitil_M,\etatil_M)$ are solutions to the one-dimensional SP
	(associated with the domain $\Vmb = (-\infty, 1]$) for $\psi_M+\eta_{M-1}$ and $\psitil_M+\etatil_{M-1}$, and hence
	\begin{align*}
		\|\eta_M-\etatil_M\|_\infty & \le \|(\psi_M+\eta_{M-1})-({\psitil}_M+{\etatil}_{M-1})\|_\infty \le \|\psi_M-\psitil_M\|_\infty + \|\eta_{M-1}-\etatil_{M-1}\|_\infty, \\
		\|\phi_M-\phitil_M\|_\infty & \le 2\|(\psi_M+\eta_{M-1})-({\psitil}_M+{\etatil}_{M-1})\|_\infty \le 2\|\psi_M-\psitil_M\|_\infty + 2\|\eta_{M-1}-\etatil_{M-1}\|_\infty.
	\end{align*}
	This means
	\begin{align*}
		\|\eta_M-\etatil_M\|_\infty & \le \sum_{i=1}^M \|\psi_i-\psitil_i\|_\infty, \quad \|\phi_M-\phitil_M\|_\infty \le 2\sum_{i=1}^M \|\psi_i-\psitil_i\|_\infty
	\end{align*}
	and hence
	\begin{align*}
		\|\eta-\etatil\|_\infty & = \sum_{k=1}^\infty \frac{\|\eta_k-\etatil_k\|_\infty}{2^k} \le \sum_{k=1}^\infty \sum_{i=1}^k \frac{\|\psi_i-\psitil_i\|_\infty}{2^k} = \sum_{i=1}^\infty \sum_{k=i}^\infty \frac{\|\psi_i-\psitil_i\|_\infty}{2^k} = \sum_{i=1}^\infty \frac{\|\psi_i-\psitil_i\|_\infty}{2^{i-1}} = 2\|\psi-\psitil\|_\infty.
		\end{align*}
		Similarly,
		\begin{align*}
		\|\phi-\phitil\| & = \sum_{k=1}^\infty \frac{\|\phi_k-\phitil_k\|_\infty}{2^k} \le 4\|\psi-\psitil\|_\infty.
	\end{align*}
	This completes the proof. 
\end{proof}

\begin{remark}
	\label{rem:remskor}
	It is easy to verify that if $\psi \in \Dmb^o_{\Rmb^{\infty}}$ is such that $\psi_i$ is continuous (resp. absolutely continuous) for each $i$
	and $\zeta = \Gamma(\psi)$, $\eta=\bar\Gamma(\psi)$, then $\zeta_i$ and $\eta_i$ are continuous (resp. absolutely continuous) for each $i$.
\end{remark}

\subsection{Rate Function}\label{sec:rateFunction}

For each $n \in \mathbb{N}$, $(X^n, Y^n)$ is a $\Dmb_{\mathbb{R}^{\infty}\times \mathbb{R}^{\infty}}$ valued random variable. In this work we  show that as $n\to \infty$, the sequence
$\{(X^n, Y^n)\}_{n\in \mathbb{N}}$ satisfies a LDP in the above space.  We begin by introducing the rate function that will govern this large deviation principle.

Recall that we assume $\lambda_n \to \lambda \in (0,\infty)$ as $n\to \infty$.
Fix $x \in \Vmb^{\infty}$ such that $x_i \ge 0$ for every $i$ and $\sum_{i=1}^{\infty} x_i <\infty$.
Let $\Cmc_T(x)$ be the subset of $\Cmb_{\Rmb^\infty \times \Rmb^\infty}$ consisting of all functions $(\zeta,\psi)$ such that
\begin{enumerate}[(i)]
\item
	$\zeta_i(0)=\psi_i(0)=x_i$. $\zeta_i$ and $\psi_i$ are absolutely continuous on $[0,T]$.
\item For all $t \in [0,T]$,
	$1=\zeta_0(t)\ge \zeta_1(t) \ge \zeta_{2}(t)\geq\ldots\geq 0$.
\item
	$\sup_{t \in [0,T]} \sum_{i=1}^\iy \zeta_i(t) < \infty$.
\item For some $\eta \in \Cmb_{\Rmb^{\infty}}$,
	$(\zeta, \eta)$ solves the Skorokhod problem for $\psi$ associated with the reflection matrix $R_\infty$. Namely, 
	\begin{equation}
		\label{eq:zeta_psi_eta}
		\zeta_i(t) = \psi_i(t) +\eta_{i-1}(t) - \eta_i(t), \; t \in [0,T],\; i\in\NN,
	\end{equation}
	where $\eta_0(t)=0$ and for every $i\ge 1$, $\eta_i(0)=0, \eta_i$ is non-decreasing, and $\int_0^T \one_{\{\zeta_i(s)<1\}} \, \eta_i(ds) = 0$.
\end{enumerate}

Note that property (iii) implies there exists a smallest $M = M(\zeta) \in \Nmb$ such that $$\sup_{t \in [0,T]} \zeta_M(t)<1.$$
Thus  one only needs to consider a $M$-dimensional SP for $\psi^M$ (associated with $(\Vmb^M, R_M)$).

\begin{remark}
	Indeed, an analogous $M$ exists for every realization of $X^n$.
	However, this $M$ depends on the random realization and there is no uniform $M$ that works for all realizations of $X^n$.
	As a result, it is convenient to work with the infinite-dimensional SP introduced in Section \ref{sec:skorokhod} when proving tightness and convergence properties in Section \ref{sec:varWeakCon}.
\end{remark}

We now define the rate function.
Recall $\ell(z) = z\log(z)-z+1$ for $z\ge 0$, and let $\vartheta_0 \doteq \lambda$ and $\vartheta_i \doteq 1$ for $i\in\NN$.
For $(\zeta,\psi) \notin \Cmc_T(x)$, define $I_{T,x}(\zeta,\psi) \doteq \infty$.
For $(\zeta,\psi) \in \Cmc_T(x)$, define
\begin{align}\label{eqn:JSQRateFunction}
	I_{T,x}(\zeta,\psi) \doteq \inf_{\varphi\in\cls_T(\zeta,\psi)}\left\{\sum_{i=0}^\iy\int_{[0,T]\times[0,1]}\vartheta_i\ell(\varphi_i(s,y))\,ds\,dy\right\},
\end{align}
where the set $\Smc_T(\zeta,\psi)$ consists of all $\varphi= (\varphi_i)_{i \in \Nmb_0}$, where each $\varphi_i: [0,T] \times [0,1] \to \Rmb_+$,
 such that
\begin{align}
	\psi_1(t) & = x_1 + \lambda \int_{[0,t]\times[0,1]}\varphi_0(s,y)\,ds\,dy  - \int_{[0,t]\times[0,1]}\one_{[0,\zeta_1(s)-\zeta_2(s))}(y) \varphi_{1}(s,y) \,ds\,dy, \label{eq:psi1}\\
	\psi_i(t) & = x_i - \int_{[0,t]\times[0,1]}\one_{[0,\zeta_i(s)-\zeta_{i+1}(s))}(y) \varphi_{i}(s,y) \,ds\,dy, \quad i \ge 2. \label{eq:psii}
\end{align}
Intuitively \eqref{eq:psi1}--\eqref{eq:psii} represents the ODE limit of the free process $Y^n$ with controls $\varphi$ modulating the rates at which jobs enter and leave the system.
Note that when $\varphi_i$ is taken to be $1$ for each $i$ in the above equations, $\{(\zeta_i,\psi_i)\}_{i\in\NN_0}$ correspond to the law of large numbers limit of the constrained and free processes $\{(X_i^n, Y^n_i)\}_{i\in\NN_0}$.
Clearly, with this choice of $\{\varphi_i\}_{i\in\NN_0}$, the cost on the right side of \eqref{eqn:JSQRateFunction} is zero which verifies that
the rate function evaluated at the LLN limit is $0$.
For a general pair $\{(\zeta_i,\psi_i)\}_{i\in\NN_0}$, the rate function is obtained by considering all controls $\{\varphi_i\}_{i\in\NN_0}$ that produce
the pair $\{(\zeta_i,\psi_i)\}_{i\in\NN_0}$ through the system of equations in \eqref{eq:psi1}--\eqref{eq:psii} and by then taking infimum over the cost for all such controls as on the right side of \eqref{eqn:JSQRateFunction}.

\subsection{Main Result}\label{sec:mainResult}

We now present the main result of this work.
First we introduce the following assumption on the initial values $X^n(0)$.

\begin{assumption}
	\label{assump:1}
	There exist a sequence of $x^n$ and $x$ in $\Vmb^\infty$ such that a.s.
	\begin{align*}
		\mbox{ for every } i \in \NN,\, X_i^n(0) = x_i^n \to x_i \mbox{ as } n \to \infty, \quad
		1 = x_0^n \ge x_1^n \ge x_2^n \ge \dotsb \ge 0, \quad
		\sup_{n\in\NN} \sum_{i=1}^\iy x_i^n<\iy.
	\end{align*}
\end{assumption}

\begin{remark}
	\label{rmk:summable}
	Note that Assumption \ref{assump:1} along with Fatou's lemma imply that $\sum_{i=1}^\iy x_i<\iy$.
\end{remark}

Assumption \ref{assump:1} will be taken to hold throughout this work and $\{x^n\}$ and $x$ as in Assumption \ref{assump:1} will be fixed. Thus we will not note this condition explicitly in our results and will suppress $x$ in the notation when writing $\Cmc_T(x)$ or $I_{T,x}$.

\begin{theorem}\label{thm:mainResult}
	The function $I_T$ defined in \eqref{eqn:JSQRateFunction} is a rate function on $\Dmb_{\Rmb^\infty \times \Rmb^\infty}$.
	The sequence $(X^n,Y^n)$ satisfies a large deviation principle on $\Dmb_{\Rmb^\infty \times \Rmb^\infty}$ with rate function $I_T$.
\end{theorem}

\begin{proof}
	 From the equivalence between a LDP and a Laplace Principle (cf. Section 1.2 of \cite{dupuis2011weak}),  it suffices to establish the following three statements.
	\begin{enumerate}
	\item[(1)]
		Laplace Upper Bound: For all $G\in\CC_b(\Dmb_{\Rmb^\infty \times \Rmb^\infty})$, 
		\begin{align}\label{eqn:LaplaceUpperBound}
			\limsup_{n\to\iy}\frac{1}{n}\log\E e^{-nG(X^n,Y^n)} \leq -\inf_{(\zeta,\psi)\in \Cmc_T}\{I_T(\zeta,\psi)+G(\zeta,\psi)\}.
		\end{align}
	\item[(2)]
		Laplace Lower Bound: For all $G\in\CC_b(\Dmb_{\Rmb^\infty \times \Rmb^\infty})$,
		\begin{align}\label{eqn:LaplaceLowerBound}
			\liminf_{n\to\iy}\frac{1}{n}\log\E e^{-nG(X^n,Y^n)}\geq -\inf_{(\zeta,\psi)\in\Cmc_T}\{I_T(\zeta,\psi)+G(\zeta,\psi)\}.
		\end{align}
		\item[(3)] $I_T$ is a rate function, namely for each $M\in[0,\iy),\{(\zeta,\psi)\in \Cmc_T:I_T(\zeta,\psi)\leq M\}$ is compact.
	\end{enumerate}
	Statements (1) and (2) are proved in Sections \ref{sec:UpperBound} and \ref{sec:LowerBound}, respectively, while the
	 proof of the third
	  statement is given in Section \ref{sec:compactSets}. 
\end{proof}
The LDP given by the above theorem is useful in obtaining estimates for probabilities of various types of rare events in the JSQ system. We now consider one such example. Consider
the critical (heavy traffic) regime, namely $\lambda=1$.
Suppose all queues are of length 1 at time 0 (i.e.\ $X^n_1(0) = x_1 = 1$ and $X^n_i(0) = x_i =0$ for $i\geq 2$).
Consider the problem of estimating the probability that a queue length will be at least $j \ge 3$ at 
some time instant in the time interval $[0,T]$. This corresponds to estimating the probability of $(X^n, Y^n)$ taking values in the following (relatively) open set
\begin{align*}
	G_j = \left\{(\zeta,\psi)\in\bar{\clc}_T| \sup_{t\in[0,T]}\zeta_{j}(t)>0\right\},
\end{align*}
where
\begin{align*}
	\bar{\clc}_T = \{(\zeta,\psi)\in\clc_T|\zeta_1(0) = 1,\ \zeta_i(0)=0 \text{ for } i\geq 2\}.
\end{align*}
We also consider the closed set $F_j$, obtained by a slight enlargement of $G_j$, defined as follows
\begin{align*}
	{F}_j= \left\{(\zeta,\psi)\in\bar{\clc}_T| \sup_{t\in[0,T]}\zeta_{j-1}(t)=1\right\}.
\end{align*}
The following result characterizes the decay rate of probabilities of $G_j$, $F_j$ and shows that for large time intervals the probability of a  queue buildup of length $j$ or higher at any station decays exponentially, approximately, with rate $e^{-\frac{n(j-2)^2}{4T}}$.
\begin{theorem}
	\label{thm:largeqs}
	For every $j\ge 3$,
	\begin{align}
		\lim_{n\to\iy}\frac{1}{n}\log\PP(X^n\in G_j) & =
		\lim_{n\to\iy}\frac{1}{n}\log\PP(X^n\in F_j) \nonumber\\
		& = - T\ell\left( \frac{\frac{j-2}{T}+\sqrt{4+(\frac{j-2}{T})^2}}{2} \right) - T\ell\left( \frac{-\frac{j-2}{T}+\sqrt{4+(\frac{j-2}{T})^2}}{2} \right)\label{eq:eq332}
	\end{align}
	and
	\begin{equation*}
		\lim_{T \to \infty}\lim_{n\to\iy}\frac{T}{n}\log(\PP(X^n\in G_j)) = \lim_{T \to \infty}\lim_{n\to\iy}\frac{T}{n}\log(\PP(X^n\in {F}_j)) = -\frac{(j-2)^2}{4}.
	\end{equation*}
\end{theorem}
Proof of the above theorem is given in Section \ref{sec:examples}. 
\begin{remark}
	\label{rem:opttraj}
	The above result shows that the probability of having a queue of length $j$ or larger at any point in a time interval $[0,T]$ decays  exponentially at rate $e^{-nI(j,T)}$, where $I(j,T)$ is given by the expression in \eqref{eq:eq332}.
	In addition, the proof gives information on  how such an event is likely to occur. The optimal path given by \eqref{eq:optPath}  corresponds to the behavior that the number of jobs per queue increases linearly at rate $(j-2)/T$ until reaching the  queue length of $(j-1)$ at time $T$. Such information can also be used to design accelerated Monte-Carlo sampling algorithms for approximating related probabilities for any fixed sized system by, for example, drawing samples from a proposal distribution which places more weight on sample paths close to this trajectory (see e.g.\ \cite[Part IV]{BudhirajaDupuis2019analysis} and references therein).
\end{remark}
\begin{conjecture}
	\label{rem:otherevent}
	Similar techniques that are used to prove Theorem \ref{thm:largeqs} are expected to be useful for studying decay rates of other types of events as well.
	In particular, we make the following two conjectures. A rigorous analysis of the probability asymptotics of events in the two conjectures is left as an open problem.
	\begin{enumerate}[(1)]
	\item Consider the event that   at some time instant in $[0,T]$, some of the servers are busy with long queues, but the remaining are idle. 
	Such an event signals an undesirable inefficiency or lack of balance in the system. More precisely, the event of interest is 
	$$\{X_1^n(t)=X_2^n(t)= \cdots =X_j^n(t)=c, X^n_{j+1}(t)=0 \mbox{ for some } t \in [0,T]\}$$
	for some $j\ge 2$ and $c \in (0,1)$. This event corresponds to  $n(1-c)$ queues being idle while the remaining $nc$ queues being of length $j$. 
    We conjecture that the most likely manner in which this event occurs is that first, the system reaches a state with $nc$ queues of length $j$ and the remaining $n(1-c)$ queues of length $j-1$ at some time instant
	$t \in [0,T]$, and then those $nc$ queues remain in that state while the other $n(1-c)$ queues decrease down to zero.
	 This result will allow the identification of an optimal trajectory and a characterization of the decay rate of the probability of the above event.

	\item Consider the event that at some time instant in
	$[0,T]$, some of the servers have queues of length $j \ge 2$ while the remaining are of length $1$, namely 
	$$\{X_1^n(t)=1, X_2^n(t)= \cdots =X_j^n(t)=c, X^n_{j+1}(t)=0 \mbox{ for some } t \in [0,T]\}$$
	for some $j\ge 2$ and $c \in (0,1)$.
	This is similar to the first event but simpler to analyze.
	Once more, there is a natural guess for the most likely manner in which this event occurs (which is similar to the conjecture in (1), except that the other $n(1-c)$ queues decrease down to one instead of zero) using which one can identify the optimal trajectory in the path space whose cost characterizes the decay rate of the probability.
	We  conjecture that in this case the optimal trajectory is piecewise linear.
	\end{enumerate}
	
\end{conjecture}

\section{Representation and Weak Convergence of Controlled Processes}\label{sec:varWeakCon}
In this section we give several preparatory results that are needed for the proofs of both the upper and the lower bounds (i.e.  \eqref{eqn:LaplaceUpperBound} and \eqref{eqn:LaplaceLowerBound}).
Section \ref{sec:varRep} presents a variational representation from \cite{budhiraja2011variational} that is the starting point of our analysis.
In Section \ref{sec:tight} we prove tightness of certain families of controls and  controlled processes which arise from the variational representation of Section \ref{sec:varRep}.
Finally, Section \ref{sec:weakconv} presents a result which characterizes the distributional limit points of this collection of processes.

\subsection{Variational Representation}\label{sec:varRep}
Recall that $\bar{\newset}_+$ denotes the class of $(\bar{\clf}\otimes \clb([0, 1]))/\clb(\RR_+)$-measurable
maps from $\Om\times[0, T]\times [0, 1]$ to $\RR_+$.
For each $m\in\NN$ let
\begin{align*}
	\bar{\newset}_{b,m}
	&\doteq \{(\varphi_k)_{k\in\NN_0}:\varphi_k\in\bar{\newset}_+\text{ for all }k\in\NN_0, \text{  for all } (\omega,t,y)\in\Om\times[0,T]\times[0,1]\\
	&\qquad \frac{1}{m}\leq \varphi_k(\om,t,y)\leq m \text{ for }k\leq m\text{ and } \varphi_k(\om,t,y)=1 \text{ for } k>m\}
\end{align*}
and let $\bar{\newset}_b\doteq\cup_{m=1}^\iy\bar{\newset}_{b,m}$.
For any $\varphi\in\bar{\newset}_b$ we denote by $(\Xbar^{n,\varphi},\Ybar^{n,\varphi},\etabar^{n,\varphi})$ the controlled analogues of $(X^n,Y^n,\eta^n)$ obtained by replacing the PRMs in \eqref{eqn:stateProc1p}--\eqref{eqn:stateProc5} with controlled point processes, $D_0^{n\lambda_n\varphi_0}$ and $D_i^{n\varphi_i}, i\in\NN$.
Namely, 
 the state evolution equations for the controlled processes are as follows,
\begin{equation}\label{eq:maincontproc}
\begin{aligned}
	\bar Y^{n, \varphi}_1(t) & = x^n_1 + \frac{1}{n} \int_{[0,t]\times[0,1]} D_{0}^{n\lambda_n\varphi_0}(ds\,dy) - \frac{1}{n} \int_{[0,t]\times[0,1]} \one_{[0,\Xbar^n_1(s-)-\Xbar^n_2(s-))}(y) 
	D_1^{n\varphi_1}(ds\,dy), \\
	\bar Y^{n, \varphi}_i(t) & = x^n_i - \frac{1}{n} \int_{[0,t]\times[0,1]} \one_{[0,\bar X^n_i(s-)-\bar X^n_{i+1}(s-))}(y) D_i^{n\varphi_i}(ds\,dy),\ i\geq 2, \\
	\bar X^{n, \varphi}_1(t)  &= \bar Y^{n, \varphi}_1(t) - \bar \eta^{n, \varphi}_1(t), \; \;
	\bar X^{n, \varphi}_i(t)  = \bar Y^{n, \varphi}_i(t) + \bar \eta^{n, \varphi}_{i-1}(t) - \bar \eta^{n, \varphi}_i(t), i \ge 2,\\
	\bar \eta^{n, \varphi}_i(t) & = \frac{1}{n} \int_{[0,t]\times[0,1]} \one_{\{\bar X^{n, \varphi}_i(s-) = 1\}} D_{0}^{n\lambda_n\varphi_0}(ds\,dy),\ i\geq 1.
\end{aligned}
\end{equation}
When it is clear from context which controls are being used we may simply write $(\Xbar^{n},\Ybar^{n},\etabar^{n})$ to represent the controlled processes.

Let $\vartheta^n_0 \doteq\lambda_n$ and $\vartheta^n_i \doteq 1$ for $i\in\NN$.
The following variational representation will be instrumental in proving both \eqref{eqn:LaplaceUpperBound} and \eqref{eqn:LaplaceLowerBound}.
For a proof we refer the reader to  \cite[Theorem 2.1]{budhiraja2011variational}. We remark that the representation in \cite{budhiraja2011variational} is given for the setting of a single PRM.
	However the result given in the lemma below, which is formulated in terms of a countable collection of PRM, follows immediately  on considering a single PRM on the augmented space $[0,T]\times[0,1]\times \Rmb_+ \times \NN_0$ with intensity $\leb\otimes\varrho$ where $\leb$ is the Lebesgue measure on $[0,T]\times [0,1]\times \Rmb_+$ and $\varrho$ is the counting measure on $\NN_0$. Proof is omitted.
\begin{lemma}\label{lem:varRep}
	Let $G\in\CC_b(\Dmb_{\Rmb^\infty \times \Rmb^\infty})$. Then
	\begin{align}\label{eqn:variationalRep}
		\begin{split}
		-\frac{1}{n} \log \E e^{-nG(X^n,Y^n)} &= \inf_{\varphi^n\in\bar{\newset}_b} \E \left\{\sum_{i=0}^\infty \int_{[0,T]\times[0,1]}\vartheta^n_i \ell(\varphi_i^n(s,y))\,ds\,dy + G(\Xbar^n,\Ybar^n) \right\}.
		\end{split}
	\end{align}
\end{lemma}

\subsection{Tightness}\label{sec:tight}
In this section we prove a key tightness result which says that if the costs are appropriately bounded then the corresponding collection of controls and controlled processes is tight.
We begin by describing the topology on the space of controls.
For $M \in (0,\infty)$, denote by $S_M$ the collection of all  $h= \{h_k\}_{k\in \Nmb_0}$, where $h_k: [0,T]\times [0,1] \to \Rmb_+$ for each $k \in \Nmb_0$ and 
$$
\sum_{k=0}^\iy \int_{[0,T]\times[0,1]} \ell(h_{k}(s,y))\,ds\,dy\leq M.$$
Any $h_k$ as above can be identified with a finite measure $\nu^{h_k}$ on $[0,T]\times [0,1]$ by the following relation
$$\nu^{h_k}(\GG) \doteq \int_{\GG} h_k(s,y)\,ds\,dy, \; \GG \subset \clb([0,T]\times [0,1]).$$
The space $\Mmb$ of finite measures on $[0,T]\times[0,1]$ is equipped with the weak convergence topology and the space $\Mmb^{\infty}$ is equipped with the corresponding product topology.
Using the above identification, each element in $S_M$ can be mapped to an element of the Polish space $\Mmb^{\infty}$ and the space $S_M$ with the inherited topology is compact (see \cite[Lemma A.1]{budhiraja2013large}).

We record the following elementary lemma for future use. 
 See  \cite[{Lemma 3.2}]{BhamidiBudhirajaDupuisWu2019rare} for a proof.
\begin{lemma}
	\label{lem:ellProp}
	Let $\ell(x) = x\log(x)-x+1$. Then the following properties hold for $\ell(x)$:
	\begin{enumerate}
	\item[a)]
		For each $\beta>0$, there exists $\gamma(\beta)\in(0,\iy)$ such that $\gamma(\beta)\to0$ as $\beta\to\iy$ and $x\leq \gamma(\beta)\ell(x)$, for $x\geq\beta\geq 1$.
	\item[b)]
		For $x\geq0$, $x\leq\ell(x)+2$.
	\end{enumerate}
\end{lemma}

The following is the main tightness result of this section. As a convention, we will take $\bar X^n_0(t) =\bar Y^n_0(t) =1$ for all $t \in [0,T]$.
\begin{lemma}\label{lem:tightness}
	Suppose that $\{\varphi^n\}$ is a sequence in $\bar\newset_b$ such that for some $M_0 \in (0,\infty)$
	\begin{align}
		 \sup_{n\in\NN}\sum_{k=0}^\iy\int_{[0,T]\times[0,1]} \ell(\varphi^n_{k}(s,y))\,ds\,dy\leq M_0  \mbox{ a.s. }\label{eqn:controlBound}
	\end{align}
	Denote by $(\bar X^n, \bar Y^n, \bar \eta^n)$ the controlled processes associated with $\varphi^n$, given by \eqref{eq:maincontproc} (replacing $\varphi$ with $\varphi^n$).
 Then, regarding $\varphi^n$ as a $S_{M_0}$ valued random variable, the sequence
	$\{(\Xbar^n,\Ybar^n,\etabar^n,\varphi^n)\}_{n\in\NN_0}$ is tight in $\DD_{\Rmb^{\infty} \times \Rmb^{\infty} \times \Rmb^{\infty}} \times S_{M_0}$. 
	Furthermore the collection $\{(\Xbar^n,\Ybar^n,\etabar^n)\}_{n\in\NN_0}$ is $\CC$-tight. 
\end{lemma}

\begin{proof}
	Since $S_{M_0}$ is compact the tightness of $\{\varphi^n\}_{n \in \Nmb_0}$ is immediate.  From Lemma \ref{lem:SMap} and since $(\bar X^n, \bar \eta^n) = (\Gamma(\bar Y^n), \bar\Gamma(\bar Y^n))$, it suffices
	now to show that $\{\Ybar^n\}_{n\in\NN}$ is $\CC$-tight in $\DD_{\Rmb^{\infty}}$.

	In order to verify tightness of $\{\bar{Y}^n\}_{n\in\NN}$ we appeal to Aldous' tightness criteria (cf. Theorem 2.2.2 in \cite{joffe1986weak}) for each $\{\Ybar^n_i\}_{n\in\NN},\ i \in \Nmb$.
	The tightness of $\{\bar{Y}^n_i(t)\}_{n\in\NN}$ in $\RR$ for each $t\in[0,T]$ follows from the following estimate.
	\begin{align*}
		\limsup_{n \to\iy} \E|\bar{Y}_i^n(t)|
		& \leq \limsup_{n\to \iy} x_i^n + \limsup_{n \to \iy} \E\int_{[0,t]\times[0,1]}[\lambda_n\varphi^n_{0}(s,y)+\varphi^n_{i}(s,y)]\,ds\,dy \\
		& \le x_i +	 (\lambda+1)M_0 + 2(\lambda+1)T,
	\end{align*}
	where the last inequality uses \eqref{eqn:controlBound}, Lemma \ref{lem:ellProp}(b), and Assumption \ref{assump:1}. 
	We next verify the following condition on the fluctuations of $\{\bar{Y}^n\}_{n\in\NN}$,
	\begin{equation}\label{eq:aldkurflu}
		\limsup_{\del\to0}\limsup_{n\to\iy}\sup_{\tau\in\clt^\del}\E[|\bar{Y}_i^n(\tau+\del)-\bar{Y}_i^n(\tau)|]=0,
	\end{equation}
	where $\clt^\del$ is the set of all $[0,T-\del]$-valued stopping times.
	For any $L>1$, it follows from the definition of $\Ybar^n$, the triangle inequality, Lemma \ref{lem:ellProp}(a), and \eqref{eqn:controlBound} that
	\begin{align*}
		&\E[|\bar{Y}_i^n(\tau+\del)-\bar{Y}_i^n(\tau)|]
		\leq \lambda_n\E\int_{[\tau,\tau+\del]\times[0,1]}\varphi^n_{0}(s,y)\,ds\,dy
		+\E\int_{[\tau,\tau+\del]\times[0,1]}\varphi^n_{i}(s,y)\,ds\,dy\\
		&\qquad = \lambda_n\E\int_{[\tau,\tau+\del]\times[0,1]}\varphi^n_{0}(s,y)\one_{\{\varphi^n_{0}(s,y)\leq L\}}\,ds\,dy+\E\int_{[\tau,\tau+\del]\times[0,1]}\varphi^n_{i}(s,y)\one_{\{\varphi^n_{i}(s,y)\leq L\}}\,ds\,dy\\
		&\qquad\quad + \lambda_n\E\int_{[\tau,\tau+\del]\times[0,1]}\varphi^n_{0}(s,y)\one_{\{\varphi^n_{0}(s,y)>L\}}\,ds\,dy+\E\int_{[\tau,\tau+\del]\times[0,1]}\varphi^n_{i}(s,y)\one_{\{\varphi^n_{i}(s,y)> L\}}\,ds\,dy\\
		&\qquad \leq (\lambda_n+1)\del L + (\lambda_n+1)M_0\gamma(L)
	\end{align*}
	and thus
	\begin{align*}
		\limsup_{\del\to0}\limsup_{n\to\iy}\sup_{\tau\in\clt^\del}\E[|\bar{Y}_i^n(\tau+\del)-\bar{Y}_i^n(\tau)|]
		\leq (\lambda+1)M_0\gamma(L).
	\end{align*}
	The property in \eqref{eq:aldkurflu} now follows upon sending $L\to\iy$. This proves the tightness of $\{Y^n\}_{n \in \Nmb}$.
	Finally, $\CC$-tightness follows upon noting that  jump sizes of $Y^n_i$ are bounded by $\frac{2}{n}$.
\end{proof}

\subsection{Characterization of Limit Points}\label{sec:weakconv}
Suppose that $\{\varphi^n\}$ is a sequence as in Lemma \ref{lem:tightness}. Then from the lemma we have the tightness of the vector sequence $\{(\Xbar^n,\Ybar^n,\etabar^n,\varphi^n)\}_{n\in\NN_0}$. In this section we 
characterize the limit points of this sequence.
It will be convenient to consider the following compensated point processes
\begin{align}\label{eqn:compensatedDef}
	\bar{D}^{n \vartheta_k^n\varphi_k^n}_k(ds\,dy)
	\doteq D^{n \vartheta_k^n\varphi_k^n}_k(ds\,dy) - n\vartheta_k^n\varphi_k^n(s,y)\,ds\,dy, \; n \in \NN, \; k \in \NN_0.
\end{align}
Define compensated processes $\ti B^n$ as
\begin{align}
	\ti B^n_1(t)&\doteq\frac{1}{n}\int_{[0,t]\times[0,1]}\bar{D}^{n\lambda_n\varphi^n_0}_{0}(ds\,dy)  - \frac{1}{n}\int_{[0,t]\times[0,1]}\one_{[0,\bar{X}^n_1(s-)-\bar{X}^n_2(s-))}(y)\bar{D}^{n\varphi^n_{1}}_1\left(ds\,dy\right)\label{eqn:tiBdef1}\\
	\ti B^n_i(t) &\doteq \frac{1}{n}\int_{[0,t]\times[0,1]}\one_{[0,\bar{X}^n_i(s-)-\bar{X}^n_{i+1}(s-))}(y)\bar{D}^{n\varphi^n_{i}}_i\left(ds\,dy\right),\qquad i\geq 2 \label{eqn:tiBdef2}
\end{align}
which allows us to write
\begin{align}
	\Ybar^n_1(t) &= X^n_1(0) + \ti B^n_1(t) +  \lambda_n\int_{[0,t]\times[0,1]}\varphi^n_{0}(s,y)\,ds\,dy \label{eq:YbarComp1}\\
	&\qquad  - \int_{[0,t]\times[0,1]}\one_{[0,\bar{X}^n_1(s)-\bar{X}^n_2(s))}(y)\varphi^n_1(s,y)\,ds\,dy,
	\nonumber\\
	\Ybar^n_i(t) &= X^n_i(0) -\ti B^n_i(t) - \int_{[0,t]\times[0,1]}\one_{[0,\bar{X}^n_i(s)-\bar{X}^n_{i+1}(s))}(y)\varphi^n_i(s,y)\,ds\,dy\qquad i \ge 2. \label{eq:YbarComp2}
\end{align}
The following lemma characterizes the limit points of $\{(\Xbar^n,\Ybar^n,\etabar^n,\varphi^n)\}_{n\in\NN_0}$.

\begin{lemma}\label{lem:convergence}
	Suppose that $\{\varphi^n\}$ is a sequence as in Lemma \ref{lem:tightness}.
	Suppose also that the associated sequence $\{(\Xbar^n,\Ybar^n,\etabar^n,\varphi^n)\}_{n\in\NN_0}$ converges along a subsequence, in distribution, to $(\Xbar,\Ybar,\etabar,\varphi)$ given on some probability space $(\Omega^*,\Fmc^*,\PP^*)$.
	Then the following holds $\PP^*$-a.s.
	\begin{enumerate}[(a)]
	\item
Equations \eqref{eq:psi1}--\eqref{eq:psii} are satisfied with $(\zeta, \psi, \varphi)$ replaced by $(\Xbar, \Ybar, \bar\varphi)$.
	\item
		$(\Xbar,\Ybar)\in \clc_T$ and $\varphi\in\cls_T(\Xbar,\Ybar)$.
		In particular, $(\Xbar,\Ybar,\etabar)$ satisfy the following system of equations
		\begin{align}
			\bar{X}_1(t) &= \Ybar_1(t) - \etabar_1(t), \label{eq:XbarLim1} \\
			\bar{X}_i(t) &= \Ybar_i(t) + \etabar_{i-1}(t) - \etabar_{i}(t), \quad i \ge 2, \label{eq:XbarLim2}
		\end{align}
		and for every $i \in \NN$,  $\etabar_i(0)=0$, $\etabar_i$ is non-decreasing, and $\int_0^t \one_{\{\Xbar_i(s)<1\}} \, \etabar_i(ds) = 0$.
	\end{enumerate}
\end{lemma}

\begin{proof}
	Assume without loss of generality that convergence occurs along the whole sequence.
Recall the expression for $\Ybar^n$ given in \eqref{eq:YbarComp1}--\eqref{eq:YbarComp2} and the definition of $\ti B^n$ given in \eqref{eqn:tiBdef1}--\eqref{eqn:tiBdef2}.
	It follows from Doob's inequality and Lemma \ref{lem:ellProp}(b) that for each $k\in\NN_0$
	\begin{align}
		\E\left(\sup_{0\leq t\leq T}|\ti B^n_k(t)|^2\right)
		& \leq \frac{1}{n}\E\int_{[0,T]\times[0,1]}[\lambda_n\varphi^n_{0}(s,y)+ \varphi^n_{k}(s,y)]\,ds\,dy \nonumber\\
		& \leq \frac{1}{n}\E\int_{[0,T]\times[0,1]}[\lambda_n (\ell(\varphi^n_{0}(s,y))+2)+ \ell(\varphi^n_{k}(s,y))+2]\,ds\,dy \nonumber\\
		& \leq \frac{1}{n}(\lambda_n+1)(M_0+2T) \to 0.\label{eq:eq613}
	\end{align}
	By appealing to the Skorokhod representation theorem {(cf.\ \cite[Theorem 6.7]{BillingsleyConv})}, we can assume without loss of generality
  that $(\Xbar^n,\Ybar^n,\etabar^n,\varphi^n, \Btil^n) \to (\Xbar,\Ybar,\etabar,\varphi, 0)$ a.s.\ on $(\Omega^*,\Fmc^*,\PP^*)$, and thus the rest of the argument will be made a.s.\ on $(\Omega^*,\Fmc^*,\PP^*)$.
	From the $\CC$-tightness proved in Lemma \ref{lem:tightness} $(\Xbar,\Ybar, \etabar)$ takes values in $\Cmb_{\RR^\iy\times\RR^\iy \times \RR^\iy}$.
	Using the triangle inequality 
	\begin{align}\label{eqn:conv1}
		\begin{split}
		&\left|\int_{[0,t]\times[0,1]}\one_{[0,\bar{X}^n_i(s)-\bar{X}^n_{i+1}(s))}(y)\varphi^n_i(s,y)\,ds\,dy-\int_{[0,t]\times[0,1]}\one_{[0,\bar{X}_i(s)-\bar{X}_{i+1}(s))}(y)\varphi_i(s,y)\,ds\,dy\right|\\
		&\qquad \leq \int_{[0,t]\times[0,1]}|\one_{[0,\bar{X}^n_i(s)-\bar{X}^n_{i+1}(s))}(y)-\one_{[0,\bar{X}_i(s)-\bar{X}_{i+1}(s))}(y)|\varphi^n_i(s,y)\,ds\,dy\\
		&\qquad\qquad + \left|\int_{[0,t]\times[0,1]}\one_{[0,\bar{X}_i(s)-\bar{X}_{i+1}(s))}(y) (\varphi^n_i(s,y)-\varphi_i(s,y))\,ds\,dy\right|
		\end{split}
	\end{align}
	for each $i\in\NN$.
	Since $\leb_t\{(s,y) : y=\bar{X}_i(s)-\bar{X}_{i+1}(s) \}=0$, where $\leb_t$ is the Lebesgue measure on $[0,t]\times[0,1]$, we have
	\begin{align*}
		|\one_{[0,\bar{X}^n_i(s)-\bar{X}^n_{i+1}(s))}(y)-\one_{[0,\bar{X}_i(s)-\bar{X}_{i+1}(s))}(y)|\to 0
	\end{align*}
	as $n \to \infty$ for $\leb_t$-a.e. $(s,y)\in[0,t]\times[0,1]$.
	From \eqref{eqn:controlBound} and the super-linearity of $\ell$, one has the uniform integrability of
		$(s,y)\mapsto\varphi^n_i(s,y)$
	with respect to the normalized Lebesgue measure on $[0,T]\times[0,1]$.
	The above two observations imply that, as $n\to \infty$,
	\begin{equation}\label{eqn:conv2}
		\int_{[0,t]\times[0,1]}|\one_{[0,\bar{X}^n_i(s)-\bar{X}^n_{i+1}(s))}(y)-\one_{[0,\bar{X}_i(s)-\bar{X}_{i+1}(s))}(y)|\varphi^n_i(s,y)\,ds\,dy\to 0.
	\end{equation}
	Recalling the topology on $S_{M_0}$, the convergence $\varphi^n \to \varphi$ and $\lambda_n\to \lambda$ implies that
	\begin{align}
		& \left|\int_{[0,t]\times[0,1]}\one_{[0,\bar{X}_i(s)-\bar{X}_{i+1}(s))}(y) (\varphi^n_i(s,y)-\varphi_i(s,y))\,ds\,dy\right|
		\to 0, \label{eqn:conv3} \\
		& \lambda_n\int_{[0,t]\times[0,1]}\varphi^n_{0}(s,y)\,ds\,dy \to \lambda\int_{[0,t]\times[0,1]}\varphi_{0}(s,y)\,ds\,dy. \label{eqn:conv3.5}
	\end{align}
	Combining \eqref{eq:YbarComp1}--\eqref{eq:YbarComp2} with \eqref{eq:eq613} and \eqref{eqn:conv1}--\eqref{eqn:conv3.5} completes the proof of (a).

	We now prove part (b).
The validity of \eqref{eq:XbarLim1}--\eqref{eq:XbarLim2} is immediate from the fact that these equalities hold with
$(\bar X, \bar Y, \bar \eta)$ replaced with
$(\bar X^n, \bar Y^n, \bar \eta^n)$.
We now check that $(\bar X, \bar Y)\in\clc_T$ by verify properties (i)--(iv) in the definition of $\clc_T$.
The absolute continuity property in property (i)  follows from the absolute continuity of $\bar Y_i$, which follows from part (a) of the lemma, and property (iv) (whose proof is given below), together with properties of the Skorokhod map.
The remaining statements in properties (i)--(ii) of the definition of $\clc_T$ are  immediate from Assumption \ref{assump:1} and the fact that for every $t$, $1= \bar X^n_0(t) \ge \bar X^n_1(t) \ge \cdots$.
	Property (iii) follows on noting that
	\begin{align*}
		\Emb \left[ \sup_{0\leq t\leq T}\sum_{i=1}^\iy\Xbar_i(t) \right]
		& \le \liminf_{n\to\iy} \Emb \left[ \sup_{0\leq t\leq T} \sum_{i=1}^\iy \Xbar^n_i(t) \right] \\
		& \leq \liminf_{n\to\iy} \sum_{i=1}^\iy x^n_i + \liminf_{n\to\iy} \int_{[0,T]\times[0,1]} \lambda_n\varphi^n_0(s,y) \,ds\,dy
		< \iy,
	\end{align*}
	where the first inequality follows from Fatou's lemma, and the last inequality uses Assumption \ref{assump:1} and \eqref{eqn:controlBound}.
	Finally we verify part (iv).
		%
	%
	Fix $k\in \Nmb$.
	That $\etabar_k(\cdot)$ is nondecreasing follows from the fact that the property holds for all $\etabar_k^n(\cdot)$.
	We now verify that
	\begin{equation}\label{eqn:reflectionproperty}
		\int_0^T\left(1-\bar{X}_k(s)\right)\etabar_k(ds)=0.
	\end{equation}
	Note that $\etabar_k^n(t)$ can only increase when $\Xbar_{k}^n(t-) = 1$ and thus
	\begin{align*}
		\int_0^T \left(1-\Xbar_k^n(t-)\right)\etabar^n_k(ds) =0.
	\end{align*}
	From this we have
	\begin{align}\label{eqn:conv4}
		\begin{split}
		\left|\int_0^T\left(1-\Xbar_k(s)\right)\etabar_k(ds)\right|
		&=   \left|\int_0^T(1-\Xbar_k(s))\etabar_k(ds)-\int_0^T \left(1-\Xbar_k^n(s-)\right)\etabar^n_k(ds)\right|\\
		&\leq \left|\int_0^T\left(1-\Xbar_k(s)\right)\etabar_k(ds)-\int_0^T\left(1-\Xbar_k(s)\right)\etabar^n_k(ds)\right|\\
		&\qquad+\int_0^T \left|\Xbar_k^n(t-)-\Xbar_k(t)\right|\etabar^n_k(dt).
		\end{split}
	\end{align}
	It follows from the fact that $\etabar^n_k\to \etabar_k$ as a finite measure on $[0,T]$ and $\Xbar_k(s)\in\CC_b([0,T]:\RR)$ that
	\begin{align}
		\label{eqn:conv5}
		\left|\int_0^T\left(1-\Xbar_k(s)\right)\etabar_k(ds)-\int_0^T\left(1-\Xbar_k(s)\right)\etabar^n_k(ds)\right|\to0\text{ as }n\to\iy.
	\end{align}
	Finally, continuity of $\Xbar_k$ and uniform convergence of $\Xbar^n_k$ to $\Xbar_k$ gives
	\begin{align*}
		\int_0^T\left|\Xbar_k^n(t-)-\Xbar_k(t)\right|\etabar^n_k(dt)
		\leq \left(\sup_{0\leq t\leq T}\left|\Xbar_k^n(t-)-\Xbar_k(t)\right|\right)\etabar^n_k(T)
		\to 0,
	\end{align*}
	which, combined with \eqref{eqn:conv4} and \eqref{eqn:conv5}, gives \eqref{eqn:reflectionproperty} verifying property (iv). Thus we have shown
	that $(\bar X, \bar Y) \in\clc_T$ a.s.  The fact that $\varphi \in \cls_T(\bar X, \bar Y)$ is now immediate from part (a) of the lemma.
\end{proof}

\section{Laplace Upper Bound}\label{sec:UpperBound}

This section is devoted to the proof of the Laplace upper bound \eqref{eqn:LaplaceUpperBound}.
Fix $G \in \CC_b(\Dmb_{\Rmb^\infty \times \Rmb^\infty})$.
From the variational representation in Lemma \ref{lem:varRep}, for all $n\in\NN$, we can select a control $\ti\varphi^n\in\bar{\newset}_b$ such that
\begin{equation}\label{eqn:upperNopt}
	-\frac{1}{n}\log \E e^{-nG(X^n,Y^n)}\geq \E\left\{\sum_{k=0}^\iy\int_{[0,T]\times[0,1]}\vartheta^n_k\ell(\ti\varphi^n_k(s,y))\,ds\,dy + G(\bar{X}^{n,\ti\varphi^n},\bar{Y}^{n,\ti\varphi^n})\right\}-\frac{1}{n}.
\end{equation}
This shows that
\begin{equation*}
	\sup_{n\in\NN}\E\sum_{k=0}^\iy\int_{[0,T]\times[0,1]}\vartheta^n_k\ell(\ti\varphi^n_k(s,y))\,ds\,dy
	\leq 2\|G\|_\iy+1
	\doteq M_G.
\end{equation*}
By a standard localization argument (see e.g.\ \cite[Proof of Theorem 4.2]{budhiraja2011variational}) it now follows that for any fixed $\sigma>0$ there is a $M_0 \in (0,\infty)$
and a sequence $\varphi^n \in \bar \newset_b$ taking values in $S_{M_0}$ a.s. such that, for every $n$, the expected value on the right side of \eqref{eqn:upperNopt} differs from the same expected value, but with $\tilde \varphi^n$ replaced by $\varphi^n$ throughout, by at most $\sigma$. In particular,
\begin{equation}\label{eqn:UpperBound1}
	-\frac{1}{n}\log \E e^{-nG(X^n,Y^n)}\geq \E\left\{\sum_{k=0}^\iy\int_{[0,T]\times[0,1]}\vartheta^n_k\ell(\varphi^{n}_k(s,y))\,ds\,dy + G(\bar{X}^{n,\varphi^{n}},\bar{Y}^{n,\varphi^{n}})\right\}-\frac{1}{n}-\sigma.
\end{equation}
Now we can complete the proof of the Laplace upper bound.
Since $\varphi^n$ are in $S_{M_0}$ a.s., from Lemma  \ref{lem:tightness} we have the tightness of $(\Xbar^n,\Ybar^n,\etabar^n, \varphi^n)$.
Assume without loss of generality that $(\Xbar^n,\Ybar^n,\etabar^n, \varphi^n)$ converges along the whole sequence, in distribution, to $(\Xbar,\Ybar,\etabar, \varphi)$, given on some probability space $(\Omega^*, \Fmc^*, \Pmb^*)$.
By Lemma \ref{lem:convergence} we have $(\Xbar,\Ybar) \in \Cmc_T$ and  $\varphi \in \Smc_T(\Xbar,\Ybar)$ a.s.\ $\Pmb^*$.
Using \eqref{eqn:UpperBound1}, Fatou's lemma, and the definition of $I_T$ in \eqref{eqn:JSQRateFunction}
\begin{align*}
	\liminf_{n\to\iy} -\frac{1}{n}\log \E e^{-nG(X^n,Y^n)} & \ge \liminf_{n\to\iy}\E\left\{\sum_{k=0}^\iy\int_{[0,T]\times[0,1]}\vartheta_k^n\ell(\varphi^n_k(s,y))\,ds\,dy + G(\bar{X}^n,\bar{Y}^n)-\frac{1}{n}-\sigma\right\} \\
	& \ge \E^*\left\{\sum_{k=0}^\iy\int_{[0,T]\times[0,1]}\vartheta_k\ell(\varphi_k(s,y))\,ds\,dy + G(\bar{X},\bar{Y})\right\}-\sigma\\
	& \ge \inf_{(\zeta,\psi)\in\clc_T} \{I_T(\zeta,\psi) + G(\zeta,\psi)\}-\sigma,
\end{align*}
where the second inequality is a consequence of  a lower semicontinuity property established in Lemma A.1 in \cite{budhiraja2013large}.
Since $\sigma \in (0,1)$ is arbitrary, this completes the proof of the Laplace upper bound.
\qed

\section{Laplace Lower Bound}\label{sec:LowerBound}

This section is devoted to the proof of the Laplace lower bound \eqref{eqn:LaplaceLowerBound}. 
The proof, given in Section \ref{sec:ProofLowerBound}, proceeds by first selecting a trajectory $(\zeta^*, \psi^*)$ which is near-optimal in the infimum on the right side of the lower bound and then an associated control $\varphi^*$ which is near-optimal for the rate function evaluated at $(\zeta^*, \psi^*)$.
One then leverages the results of Section \ref{sec:varWeakCon} to construct a sequence of controls and associated controlled process $(\Xbar^n,\Ybar^n,\varphi^n)$ which are tight and have weak limit points $(\bar \psi, \bar \zeta, \bar \varphi)$ that satisfy $\bar \varphi = \varphi^*  \in \Smc_T(\bar \zeta, \bar \psi)$.
However, it is not guaranteed that $(\zeta^*, \psi^*) = (\bar \zeta, \bar \psi)$ as it is not necessarily true that the control $\varphi^*$ drives a unique set of solutions to the controlled ODE, \mbox{\eqref{eq:psi1}--\eqref{eq:psii}}.
In particular, if $\varphi^*\in\cls_T(\zeta^*, \psi^*)$ and $\varphi^*\in\cls_T(\bar \zeta, \bar \psi)$ then, in general, $(\zeta^*, \psi^*)$ need not equal $(\bar \zeta, \bar \psi)$.
This obstacle is overcome in Lemma \ref{lem:uniqueness}, which says that one can select a trajectory and control $(\zeta, \psi, \varphi)$ suitability close to $(\zeta^*, \psi^*, \varphi^*)$ for which this uniqueness property does hold. 
This uniqueness result is at the technical heart of this work and its proof relies on several other intermediate results which proceed by successively modifying the original trajectories and controls  so that the modified trajectories  remain close to the original trajectories
while changing the cost only slightly, and such that the final set of trajectories and controls do have the desired uniqueness property.
To do this, first Lemmas \ref{lem:uniqueness_preparation_0} and \ref{lem:uniqueness_preparation} form an inductive argument showing that one can approximate any trajectory with one in which the shortest queue length $\pi$ changes only a finite number of times.
This relies on taking any candidate trajectory $(\bar \zeta, \bar\psi)$ and ``smoothing'' out small excursions of the trajectories over which $\pi$ may change an infinite number of times.
Then, in Lemma \ref{lem:uniqueness_preparation_2}, we further modify the trajectory so that it is well-behaved in the neighborhood of the finitely many time instants  at which $\pi$ changes values.
With these properties in hand, we then proceed to the proof of the uniqueness result, Lemma \ref{lem:uniqueness}.
The main step in the proof follows from an argument by contradiction which hinges on introducing ``$\eps$-gaps'' in the spatial thinning used to define the controls for the PRM, see \eqref{eq:gap}.
These gaps alter each control $\varphi_i(t, y)$ so that they are zero in the $\varepsilon$-neighborhood around $\zeta_k(t)-\zeta_{k-1}(t)$ without altering the resulting state trajectory.
The basic idea is that if there were two different solutions to the ODE then there would be some time $\tau$ when they diverged.
In the time right after $\tau$, the two trajectories must be extremely close which due to the $\varepsilon$-gap property 
will say that  the time derivative of the ODE must remain the same in a small interval beyond $\tau$, triggering the contradiction. 
A similar ``$\varepsilon$-gap'' argument in a very different context has recently also been used in \cite{BhamidiBudhirajaDupuisWu2019rare}.
The technical details of this argument rely heavily on the regularity properties established in Lemmas \ref{lem:uniqueness_preparation_2}--\ref{lem:uniqueness_preparation}.
In order to make clear the motivation behind the various lemmas in this section and improve the overall readability we present these lemmas out of order.
An outline of how these various lemmas are used and the organization of this section is as follows:
\begin{itemize}
	\item Proof of the lower bound \eqref{eqn:LaplaceLowerBound} is a consequence of the main uniqueness result, Lemma \ref{lem:uniqueness}. The statement of
	Lemma \ref{lem:uniqueness} and proof of \eqref{eqn:LaplaceLowerBound} based on this lemma are given in Section \ref{sec:ProofLowerBound}.
	\item Lemma \ref{lem:uniqueness} relies on Lemma \ref{lem:uniqueness_preparation_2} which asserts 
	existence of approximating trajectories for which $\pi$ changes at finitely many time instants and that have suitable regularity in the neighborhood of the boundaries of the finitely many pieces of the approximating trajectories over which $\pi$ is constant. The statement of Lemma \ref{lem:uniqueness_preparation_2} and proof of  Lemma \ref{lem:uniqueness} based on this lemma are given in Section \ref{sec:uniqueness}.
	\item Section \ref{sec:uniqueness_preparation} is the longest subsection and it is devoted to the proof of Lemma \ref{lem:uniqueness_preparation_2}. This lemma is a consequence of Lemmas \ref{lem:uniqueness_preparation_0} and \ref{lem:uniqueness_preparation} that construct approximating trajectories in which the shortest queue $\pi$ changes a finite number of times. These lemmas (statements and proofs) are given in Section
	\ref{sec:uniqueness_preparation} as well. 
\end{itemize}

\subsection{Proof of Lower Bound}\label{sec:ProofLowerBound}
The following  lemma is key to the proof of the lower bound \eqref{eqn:LaplaceLowerBound}. It says that,  given a trajectory $(\zeta^*,\psi^*) \in \clc_T$, one can select a trajectory $(\zeta,\psi)$ which is suitably close to $(\zeta^*,\psi^*)$ and a control $\varphi$ such that $(\zeta,\psi)$ is the unique trajectory driven by $\varphi$.
We will equip the space $\Cmb_{\Rmb^{\infty}\times \Rmb^{\infty}}$ with the distance: for $(\zeta, \psi), (\zetatil, \psitil) \in \Cmb_{\Rmb^{\infty}\times \Rmb^{\infty}}$, 
$$d((\zeta, \psi), (\zetatil, \psitil)) \doteq \|(\zeta, \psi) - (\zetatil, \psitil)\|_{\infty} \doteq \|\zeta-\zetatil\|_{\infty} + \|\psi-\psitil\|_{\infty} = \sum_{i=1}^\infty \frac{\|\zeta_i-\zetatil_i\|_\infty \wedge 1}{2^i} + \sum_{i=1}^\infty \frac{\|\psi_i-\psitil_i\|_\infty \wedge 1 }{2^i}$$
where recall that for 
$f \in \Cmb([0,T]: \mathbb{R})$, 
$\|f\|_{\infty} = \sup_{0\le t \le T} |f(t)|$.

\begin{lemma}\label{lem:uniqueness}
	Fix $\sigma\in(0,1)$. Given $(\zeta^*,\psi^*)\in\clc_T$ with $I_T(\zeta^*,\psi^*)<\iy$, there exists $(\zeta,\psi)\in\clc_T$ and $\varphi\in\cls_T(\zeta,\psi)$ such that
	\begin{enumerate}[a)]
	\item
		$\|(\zeta,\psi)-(\zeta^*,\psi^*)\|_\infty \le \sigma$.
	\item
		$\sum_{k=0}^\iy\int_{[0,T]\times[0,1]} \vartheta_k\ell(\varphi_{k}(s,y)) \,ds\,dy \le I_T(\zeta,\psi) +\sigma \leq I_T(\zeta^*,\psi^*)+2\sigma$.
	\item
		If $(\ti\zeta,\ti\psi)$ is another pair in $\clc_T$ such that $\varphi\in\cls_T(\ti\zeta,\ti\psi)$, then $(\ti\zeta,\ti\psi)=(\zeta,\psi)$.
	\end{enumerate}
\end{lemma}

Proof of this lemma will be given in Section \ref{sec:uniqueness}.
We now complete the proof of the lower bound using this result.
Fix $G\in\CC_b(\Dmb_{\Rmb^\infty \times \Rmb^\infty})$ and $\sigma\in(0,1)$. Select a trajectory $(\zeta^*,\psi^*)$ which is $\sigma$-optimal for the RHS of \eqref{eqn:LaplaceLowerBound}, namely
\begin{equation}
	\label{eq:lowerbd_pf1}
	I_T(\zeta^*,\psi^*)+G(\zeta^*,\psi^*) \leq \inf_{(\zeta,\psi)\in\clc_T}\{I_T(\zeta,\psi)+G(\zeta,\psi)\}+\sigma.
\end{equation}
By continuity of $G$ and Lemma \ref{lem:uniqueness}, we can find $(\zetabar,\psibar) \in \clc_T$ and $\varphibar \in S_T(\zetabar,\psibar)$ such that the uniqueness property in  Lemma \ref{lem:uniqueness} holds (with $\varphi$ replaced by $\bar \varphi$) and 
\begin{align}
	\sum_{k=0}^\iy\int_{[0,T]\times[0,1]}\vartheta_k\ell(\varphibar_k(s,y))\,ds\,dy+G(\zetabar,\psibar) & \le I_T(\zetabar,\psibar)+G(\zetabar,\psibar)+\sigma \notag \\
	& \leq I_T(\zeta^*,\psi^*)+G(\zeta^*,\psi^*)+2\sigma. \label{eq:lowerbd_pf2}
\end{align}

Define the sequence of controls $\varphi^n \in \bar{\newset}_b$ as
\begin{align*}
	\varphi^n_i(s,y) & \doteq \frac{1}{n} \one_{\{\varphibar_i(s,y) \le \frac{1}{n}\}} + \varphibar_i(s,y) \one_{\{\frac{1}{n} < \varphibar_i(s,y) < n\}} + n \one_{\{\varphibar_i(s,y) \ge n \}}, \quad i \le n, \\
	\varphi^n_i(s,y) & \doteq 1, \quad i > n.
\end{align*}
Then there is a $M_0 \in (0,\infty)$ such that the sequence $\{\varphi^n\}$ satisfies \eqref{eqn:controlBound}.
Furthermore, it is easily checked that
 $\varphi^n \to \varphibar$ (in $S_{M_0}$).
It then follows from Lemmas \ref{lem:tightness} and \ref{lem:convergence} that $(\Xbar^n,\Ybar^n,\etabar^n,\varphi^n)$ is tight and any limit point $(\Xbar,\Ybar,\etabar,\varphi)$, given on some probability space $(\Om^*,\clf^*,\PP^*)$,
satisfies $(\Xbar, \Ybar) \in \clc_T$ and $\varphi \in \cls_T(\Xbar, \Ybar)$ a.s.
From the fact that $\varphi^n\to\varphibar$ we must have $\varphi=\varphibar$.
Thus $\varphibar\in\cls_T(\Xbar,\Ybar)$  and since we also have $\varphibar\in\cls_T(\zetabar,\psibar)$, we must have  $(\Xbar,\Ybar)=(\zetabar,\psibar)$ a.s. $\PP^*$ from the uniqueness property noted above.
Noting that $\ell(\varphi^n_k(s,y))\leq \ell(\varphibar_k(s,y))$ for all $n\in\NN_0$ and $(s,y)\in[0,T]\times[0,1]$,
it then follows from the variational representation \eqref{eqn:variationalRep} and \eqref{eq:lowerbd_pf1}--\eqref{eq:lowerbd_pf2} that
\begin{align*}
	\limsup_{n\to\iy}-\frac{1}{n}\log\E e^{-nG(X^n,Y^n)}
	&\leq \limsup_{n\to\iy}\E\left\{\sum_{k=0}^\iy\int_{[0,T]\times[0,1]}\vartheta^n_k\ell(\varphi^n_k(s,y))\,ds\,dy+G(\bar{X}^{n},\bar{Y}^{n})\right\}\\
	&\leq \sum_{k=0}^\iy\int_{[0,T]\times[0,1]}\vartheta_k\ell(\varphibar_k(s,y))\,ds\,dy+G(\zetabar,\psibar)\\
	&\leq\inf_{(\zeta,\psi)\in\clc_T}\left\{I_T(\zeta,\psi) + G(\zeta,\psi)\right\}+3\sigma.
\end{align*}
The inequality in \eqref{eqn:LaplaceLowerBound} now follows upon sending $\sigma\to0$. \hfill \qed

\subsection{Proof of Lemma \ref{lem:uniqueness}}\label{sec:uniqueness}
In this section we prove Lemma \ref{lem:uniqueness} using an intermediate result, Lemma \ref{lem:uniqueness_preparation_2}.
Proof of Lemma \ref{lem:uniqueness_preparation_2} is given in Section \ref{sec:uniqueness_preparation}.
Consider $(\zeta, \psi) \in \clc_T$. Informally, we will view $\zeta_k(t)$ as the (asymptotic analogue of) fraction of queues with $k$ or more jobs.
Let $\pi(t) \doteq \max \{k:\zeta_k(t) = 1\}$ represent the shortest queue length at time $t$.
It is easily verified that since $\zeta$ is continuous,
\begin{equation}
	\label{eq:pi_upper_semicts}
	\pi(t) \mbox{ is upper semi-continuous}.
\end{equation}
The following lemma shows that any $(\zeta, \psi) \in \clc_T$ can be suitably approximated by a trajectory for which 
the associated $\pi(t)$ changes only a finite number of times, in an appropriately regular manner.
\begin{lemma}
	\label{lem:uniqueness_preparation_2}
	Fix $\sigma\in(0,1)$. Given $(\zetatil,\psitil)\in\clc_T$ with $I_T(\zetatil,\psitil)<\iy$, there exists $(\zeta,\psi)\in\clc_T$ and $\varphi \in \Smc_T(\zeta,\psi)$ such that
	\begin{enumerate}[i)]
	\item
		There exist $N \in \NN$, $c_i \in \NN_0$ for all $i=0, 1, \ldots N-1$, and a finite partition $0=t_0 < t_1 < \dotsb < t_N = T$, such that $\pi(t)=c_i$ for $t \in (t_i,t_{i+1})$.
	\item
		For any $i$, if $\pi(t_i) > c_i$,  there exists some $\Delta t_i \in (0,t_{i+1}-t_i)$ such that $\varphi_0(t,y)=0$ for $t \in (t_i,t_i+\Delta t_i)$.
	\item
		$\|(\zeta,\psi)-(\zetatil,\psitil)\|_\infty \le \sigma$.
	\item
		$I_T(\zeta,\psi) \leq I_T(\zetatil,\psitil)+\sigma$.
	\item
		$\sum_{k=0}^\iy\int_{[0,T]\times[0,1]} \vartheta_k\ell(\varphi_{k}(s,y)) \,ds\,dy \le I_T(\zeta,\psi) + \sigma$.
	\item
		There exists $M \in \Nmb$ such that $\varphi_k = 1$ for $k \ge M$.
	\end{enumerate}
\end{lemma}

The proof of this lemma is deferred to Section \ref{sec:uniqueness_preparation}.
We now have all of the ingredients needed to prove Lemma \ref{lem:uniqueness}.
To simplify the notation, let for $(\zeta, \psi) \in \clc_T$ and $t\in [0,T]$,
\begin{align}
	\label{eqn:rkdef}
	r_k(t) \doteq \zeta_k(t) - \zeta_{k+1}(t).
\end{align}
Informally, this represents  the fraction of queues with length $k$ at time instant $t$.

\begin{proof}[Proof of Lemma \ref{lem:uniqueness}]
	We first give the basic idea of the proof. 
We will refer to Lemma \ref{lem:uniqueness_preparation_2}(i)--(vi) simply as properties (i)--(vi), since they will be frequently used.
Fix $(\zeta^*, \psi^*) \in \clc_T$ with $I_T(\zeta^*, \psi^*)<\iy$.
We will begin by approximating this trajectory by a more regular trajectory $(\zeta, \psi)$ of the form given  in Lemma \ref{lem:uniqueness_preparation_2}.
Next, we will make additional modifications to the associated control so that one has the desired uniqueness property in part (c).
Ultimately the uniqueness will be argued through a proof by contradiction. 
The purpose of these further modifications to the controls is to ensure that, should there be two possible trajectories driven by the 
modified control, then the time derivative of the state process will be the same in a small time interval following the moment at which 
these trajectories diverge, triggering the contradiction.
This is accomplished by introducing ``$\varepsilon$-gaps'' in the spatial thinning used to define the control process.
The basic idea is to set the controls equal to zero in $\varepsilon$-neighborhoods of each $r_k(t)$ and, in some cases, at the 
boundary between regions, so that a small divergence from $r_k$ does not cause a change in the time derivative of the state process.
 The controls are also reweighted so that they do not change the state trajectory $(\zeta, \psi)$. 
We then show that this modified control increases the cost only slightly and has all the desired properties in Lemma~\ref{lem:uniqueness}.
	
	\noindent {\em Approximating $(\zeta, \psi)$ by a more regular trajectory.}
Fix $\sigma \in (0,1)$.
	Let $(\zeta,\psi) \in \clc_T$ be as in Lemma \ref{lem:uniqueness_preparation_2} with $(\tilde \zeta, \tilde \psi)$ replaced with
	$(\zeta^*, \psi^*)$ and $\sigma$ replaced with $\sigma/3$. Denote the associated control by $\hat{\varphi}$. Thus $\hat{\varphi} \in \cls_T(\zeta,\psi)$.
	It is now immediate that part $(a)$ and the second inequality in part $(b)$ of Lemma \ref{lem:uniqueness} are satisfied.
	Furthermore, since $\ell$ is a convex function and $\ell(1)=0$ we can assume without loss of generality (and no change to the state trajectory) that
	\begin{align}\label{eqn:onlyTime}
		\begin{split}
		\hat{\varphi}_0(t,y) &= \hat{\rho}_0(t) \\
		\hat{\varphi}_{k}(t,y) &= \hat{\rho}_{k}(t)\one_{[0,r_k(s))}(y)+\one_{[r_k(s),1]}(y), \quad k \in \Nmb.
		\end{split}
	\end{align}
	\noindent {\em An $\varepsilon$-gap modification to the thinning function.}
	Fix $\eps>0$ and define a new control $\varphi^\eps$ with an $\eps$-gap around $r_k(s)$, defined as
	\begin{align}\label{eq:gap}
		\begin{split}
		\varphi^\eps_0(t,y) & = \hat{\varphi}_0(t,y),  \\
		\varphi^\eps_{k}(t,y) & = \frac{\hat{\rho}_{k}(t)}{1-\eps}\one_{[0,(1-\eps)r_k(t))}(y)+\one_{[(1+\eps)r_k(t),1]}(y), \quad k \in \Nmb.
		\end{split}
	\end{align}
	This new control $\varphi^\eps$ results in the same trajectory $(\zeta,\psi)$ and the (possible) increase in cost can be estimated as follows,
	\begin{align*}
		&\sum_{k=0}^\iy\int_{[0,T]\times[0,1]} \left[ \vartheta_k\ell(\varphi^\eps_{k}(s,y)) - \vartheta_k\ell(\hat{\varphi}_{k}(s,y)) \right] ds\,dy\\
		&\quad= \sum_{k=1}^\iy\int_0^T\left[(1-\eps)r_k(s)\ell\left(\frac{\hat{\rho}_{k}(s)}{1-\eps}\right) + 2\eps r_k(s)\ell(0)-r_k(s)\ell\left(\hat{\rho_{k}}(s)\right)\right]ds\\
		&\quad=\sum_{k=1}^\iy\int_0^T r_k(s)\left[\hat{\rho}_{k}(s)\log\left(\frac{1}{1-\eps}\right)+\eps\right]ds \\
		& \quad \le \sum_{k=1}^\iy\int_0^T r_k(s)\left[(\ell(\hat{\rho}_{k}(s))+2)\log\left(\frac{1}{1-\eps}\right)+\eps\right]ds \\
		& \quad \le (I_T(\zeta,\psi) + \frac{1}{3}\sigma)\log\left(\frac{1}{1-\eps}\right) + 2T\log\left(\frac{1}{1-\eps}\right) +  T\varepsilon.
	\end{align*}
Since $I_T(\zeta, \psi)<\infty$ there exists an $\varepsilon_0>0$ such that
	\begin{align}
		\label{eq:cost-compare-1}
		\sum_{k=0}^\iy\int_{[0,T]\times[0,1]} \left[ \vartheta_k\ell(\varphi^\eps_{k}(s,y)) - \vartheta_k\ell(\hat{\varphi}_{k}(s,y)) \right] ds\,dy
		\leq \frac{1}{3}\sigma
	\end{align}
	for all $\eps \le \varepsilon_0$.
	
	\noindent{\em A final modification of controls.} We make one last modification.
	Fix $\varepsilon < \varepsilon_0 \wedge \frac{1}{3T}\sigma$.
	Define $\varphi \doteq \varphi^\varepsilon$ except that, for $t \in (t_i,t_{i+1})$, $i \in \{0,1,\dotsc,N-1\}$, if $c_i \ge 2$, set $\varphi_{c_i-1}(t,y)=\one_{[\varepsilon,1]}(y)$.
	This new control $\varphi$ still satisfies property (ii) and results in the same trajectory, namely $\varphi \in \Smc_T(\zeta,\psi)$, since $\pi(t)=c_i$ for $t\in(t_i,t_{i+1})$ by property (i).
	Furthermore, the control $\varphi$ only incurs a small additional cost which can be estimated as follows:
	\begin{align*}
		&\sum_{k=0}^\iy\int_{[0,T]\times[0,1]} \left[\vartheta_k\ell(\varphi_{k}(s,y)) - \vartheta_k\ell(\varphi_{k}^\eps(s,y)) \right] ds\,dy \le T \varepsilon \ell(0) \le \frac{1}{3}\sigma.
	\end{align*}
	Combining this, \eqref{eq:cost-compare-1}, and property (v), yields
	\begin{align*}
		\sum_{k=0}^\iy\int_{[0,T]\times[0,1]} \vartheta_k\ell(\varphi_{k}(s,y)) \,ds\,dy
		&  \le \sum_{k=0}^\iy\int_{[0,T]\times[0,1]} \vartheta_k\ell(\varphi^\eps_{k}(s,y)) \,ds\,dy + \frac{1}{3}\sigma \\
		&  \le \sum_{k=0}^\iy\int_{[0,T]\times[0,1]} \vartheta_k\ell(\hat\varphi_{k}(s,y)) \,ds\,dy + \frac{2}{3}\sigma \\
		&  \le I_T(\zeta,\psi) + \sigma,
	\end{align*}
	which completes the proof of part $(b)$.

	We now show that with the above choice of $\varphi$,  part $(c)$ holds. 
	Suppose there is another pair $(\ti\zeta,\ti\psi)\in\Cmc_T$ such that $\varphi\in\cls_T(\ti\zeta,\ti\psi)$.
	Define time $\tau$ by
	$$\tau \doteq \inf\{t \in [0,T] : (\ti\zeta(t),\ti\psi(t)) \ne (\zeta(t),\psi(t))\} \wedge T.$$
	We argue by contradiction and suppose that $\tau < T$. Then  $\tau \in [t_i,t_{i+1})$ for some $i \in \{0,1,\dotsc,N-1\}$.
	Let $\Ktil \doteq \pi(t_i) \ge K \doteq c_i$ where the inequality is a consequence of the upper semicontinuity of $\pi$ (see \eqref{eq:pi_upper_semicts}).
	By continuity, we have $(\ti\zeta(\tau),\ti\psi(\tau),\etatil(\tau)) = (\zeta(\tau),\psi(\tau),\eta(\tau))$.
	We claim that there exists some $\delta \in (0,t_{i+1}-\tau)$ such that
	\begin{align}
		& \etatil_k(t) = \eta_k(t), \quad t \in [\tau,\tau+\delta], \quad k \in \Nmb. \label{eq:claim_uniqueness_2}
	\end{align}
	Assuming for the moment that the claim holds, we now complete the proof of part $(c)$.
	Define $Y$ and $\Del_k$ by
	\begin{align*}
		Y(t) &\doteq \sum_{k=0}^\infty | \rtil_k(t) - r_k(t) |, \\
		\Delta_k(t) &\doteq \zetatil_k(t) - \zeta_k(t), \quad k \in \Nmb_0
	\end{align*}
	where $r_k$ and $\rtil_k$ are as defined in \eqref{eqn:rkdef} using $\zeta$ and $\zetatil$ respectively.
	We first argue that for $Y(t)$, the differentiation, under the summation over $k$, with respect to $t$ is valid for a.e.\ $t \in [\tau,\tau+\delta]$.
	From the dominated convergence theorem, it suffices to give a summable bound on $\frac{d}{dt}|\rtil_k(t) - r_k(t)|$ for $k \ge M \vee (\Mtil+2)$,
	where $M$ is as in property (vi) and $\Mtil \doteq \min\{j \in \Nmb: \sup_{t \in[0,T]} \zeta_j(t) < 1, \sup_{t \in[0,T]} \zetatil_j(t) < 1\} $ ( which is finite by property (iii) of $\Cmc_T$). 
	Note that for $k \ge M \vee (\Mtil+2)$ and a.e.\ $t \in [0,T]$, $\eta_{k-1}(t)=\eta_k(t)=\eta_{k+1}(t)=\etatil_{k-1}(t)=\etatil_k(t)=\etatil_{k+1}(t)=0$. Hence,
	\begin{align*}
		\frac{d}{dt} |\rtil_k(t)-r_k(t)| & = \frac{d}{dt} |\psitil_k(t) - \psitil_{k+1}(t) - \psi_k(t) + \psi_{k+1}(t)| \\
		& \le \rtil_k(t)+\rtil_{k+1}(t)+r_k(t)+r_{k+1}(t) \le \zetatil_k(t) + \zeta_k(t) \le 2x_k,
	\end{align*}
	where the equality uses \eqref{eqn:rkdef} and \eqref{eq:zeta_psi_eta}, the first inequality uses \eqref{eq:psii} and property (vi), and the last inequality uses \eqref{eq:psii} again.
	Since $x_k$ is summable by Remark \ref{rmk:summable}, we have the desired property of differentiation under the summation.

	Differentiating $Y$ we have that for a.e. $t \in [\tau,\tau+\delta]$,
	\begin{align*}
		Y'(t) & = \sum_{k=0}^\infty \left[ ( \rtil_k'(t) - r_k'(t) ) \one_{\{\rtil_k(t) > r_k(t)\}} + ( r_k'(t) - \rtil_k'(t) ) \one_{\{\rtil_k(t) < r_k(t)\}} \right] \\
		& = \sum_{k=0}^\infty \left[ ( \Delta_k'(t)-\Delta_{k+1}'(t) ) \one_{\{\rtil_k(t) > r_k(t)\}} + ( \Delta_{k+1}'(t)-\Delta_k'(t) ) \one_{\{\rtil_k(t) < r_k(t)\}} \right] \\
		& = \sum_{k=1}^\infty \Delta_k'(t) \left[ \one_{\{\rtil_k(t) > r_k(t)\}} - \one_{\{\rtil_k(t) < r_k(t)\}} - \one_{\{\rtil_{k-1}(t) > r_{k-1}(t)\}} + \one_{\{\rtil_{k-1}(t) < r_{k-1}(t)\}} \right],
	\end{align*}
	where the last line follows from rearranging terms and the fact that $\Delta_0'(t)=0$.
	For any $k \ge 1$, when $\rtil_k(t) > r_k(t)$, we have
	\begin{align*}
		& \one_{\{\rtil_k(t) > r_k(t)\}} - \one_{\{\rtil_k(t) < r_k(t)\}} - \one_{\{\rtil_{k-1}(t) > r_{k-1}(t)\}} + \one_{\{\rtil_{k-1}(t) < r_{k-1}(t)\}} \\
		&\qquad = 1 - \one_{\{\rtil_{k-1}(t) > r_{k-1}(t)\}} + \one_{\{\rtil_{k-1}(t) < r_{k-1}(t)\}} \ge 0
	\end{align*}
	and
	\begin{equation*}
		\Delta_k'(t) = \zetatil_k'(t)-\zeta_k'(t) = \psitil_k'(t) - \psi_k'(t) = \int_0^1 \left( \one_{[0,r_{k}(t))}(y) - \one_{[0,\rtil_{k}(t))}(y) \right) \varphi_{k}(t,y)\,dy \le 0
	\end{equation*}
	for a.e.\ $t \in [\tau,\tau+\delta]$, by \eqref{eq:claim_uniqueness_2} and \eqref{eq:zeta_psi_eta}.
	Similarly, when $\rtil_k(t) < r_k(t)$, we have
	\begin{align*}
		& \one_{\{\rtil_k(t) > r_k(t)\}} - \one_{\{\rtil_k(t) < r_k(t)\}} - \one_{\{\rtil_{k-1}(t) > r_{k-1}(t)\}} + \one_{\{\rtil_{k-1}(t) < r_{k-1}(t)\}} \le 0
	\end{align*}
	and $\Delta_k'(t) \ge 0$ for a.e.\ $t \in [\tau,\tau+\delta]$.
	From these we have $Y'(t) \le 0$ and hence $Y(t)=Y(\tau)=0$ for $t \in [\tau, \tau+\delta]$.
	Therefore $\rtil_k(t)=r_k(t)$ and hence $\zetatil_k(t)=\zeta_k(t)$ for all $t \in [\tau,\tau+\delta]$ and $k \ge 1$.
	Thus, we have shown that $\zetatil(t) = \zeta(t)$, and therefore $\psitil(t) = \psi(t)$, for $t \in [\tau,\tau+\delta]$.
	 This contradicts the definition of $\tau$ and completes the proof of part $(c)$.

	Finally we verify the claim that there exist $\del$ such that \eqref{eq:claim_uniqueness_2} holds.
	Recall that $\tilde K = \pi(t_i)$ and $K=c_i$.
	Consider the following two possible cases: (1) $\tau=t_i$ and $\Ktil > K$ and (2) $\tau \in (t_i,t_{i+1})$ or $\Ktil=K$.

	\underline{Case (1): $\tau=t_i$ and $\Ktil > K$.}
	In this case, we simply take $\delta = \Delta t_i$.
	From property (ii) $\varphi_0(t,y)=0$ for $t \in (\tau,\tau+\delta)$ and so we have $\psi_k'(t) \le 0$ and $\psitil_k'(t) \le 0$ for each $k \ge 1$.
	Therefore, $\eta_k, \tilde \eta_k$ stay constant over $[\tau, \tau+\delta]$ and hence \eqref{eq:claim_uniqueness_2} holds.

	\underline{Case (2): $\tau \in (t_i,t_{i+1})$ or $\Ktil=K$.}
	In this case $\pi(\tau)=K$ and hence $\rtil_K(\tau) = r_K(\tau) > 0$ and $\rtil_k(\tau)=r_k(\tau)=0$ for $k \le K-1$.
	Recall the fixed $\varepsilon$ introduced below \eqref{eq:cost-compare-1}.
	By continuity, there exists some $\delta \in (0,t_{i+1}-\tau)$ such that for $t \in [\tau,\tau+\delta]$,
	\begin{align}
		& \rtil_K(t)>0, \: r_K(t)>0, \label{eq:unique_1} \\
		& |\rtil_K(t) - r_K(t)| \le \varepsilon r_K(t), \label{eq:gap_K}\\
		& \rtil_k(t)<\varepsilon, r_k(t)<\varepsilon, \quad k \le K-1. \label{eq:gap_K-1}
	\end{align}
	The first inequality in \eqref{eq:unique_1} implies that $\pitil(t) \le K$ for $t \in [\tau,\tau+\delta]$, where $\pitil(t) \doteq \max \{k:\zetatil_k(t) = 1\}$.
	Fix $t \in [\tau,\tau+\delta]$.
	Since $\pi(t)=K \ge \pitil(t)$, we have $\zeta_{K+1}(t)<1$, $\zetatil_{K+1}(t)<1$, and hence \eqref{eq:claim_uniqueness_2} holds for $k \ge K+1$ using the fact that $\eta_k, \tilde \eta_k$ do not increase for these coordinates.
	It now remains to show \eqref{eq:claim_uniqueness_2} for $k\le K$.
	We consider the following three different cases.

	\underline{$K=0$:} In this case \eqref{eq:claim_uniqueness_2} holds trivially for $k\le K$.

	\underline{$K=1$:} We only need to check \eqref{eq:claim_uniqueness_2} for $k=1$. Note that, for $t \in [\tau,\tau+\delta]$
	\begin{align*}
		\psitil_1(t) & = \psitil_1(\tau) + \lambda \int_{[\tau,t]\times[0,1]}\varphi_0(s,y)\,ds\,dy - \int_{[\tau,t]\times[0,1]} \one_{[0,\rtil_{1}(s))}(y)\varphi_{1}(s,y)\,ds\,dy \\
		& = \psi_1(\tau) + \lambda \int_{[\tau,t]\times[0,1]}\varphi_0(s,y)\,ds\,dy - \int_{[\tau,t]\times[0,1]} \one_{[0,r_{1}(s))}(y)\varphi_{1}(s,y)\,ds\,dy = \psi_1(t),
	\end{align*}
	where the second line uses \eqref{eq:gap_K} and the $\varepsilon$-gap property of $\varphi_1$ in \eqref{eq:gap}.
	By the uniqueness of solutions of the   one-dimensional SP,  $\etatil_1(t)=\eta_1(t)$ for $t \in [\tau,\tau+\delta]$, which gives \eqref{eq:claim_uniqueness_2} for $k=1$.

	\underline{$K \ge 2$:} In this case $\varphi_{K-1}(t,y) = \one_{[\varepsilon,1]}(y)$ (see below \eqref{eq:cost-compare-1}).
	This and \eqref{eq:gap_K-1} yield,
	\begin{align*}
		\psitil_{K-1}'(t) & = \one_{\{K=2\}} \lambda\int_0^1 \varphi_0(t,y)\,dy - \int_0^1 \one_{[0,\rtil_{K-1}(t))}(y)\varphi_{K-1}(t,y)\,dy = \psi_{K-1}'(t) \ge 0
	\end{align*}
	for a.e.\ $t$.
	Therefore, for $t \in [\tau, \tau+\delta]$, $\psitil_{K-1}(t)=\psi_{K-1}(t)$, $1 \ge \zetatil_{K-1}(t) \ge \zetatil_{K-1}(\tau)=1$, and thus
	\begin{align*}
		\zetatil_k(t) & =\zeta_k(t)=1, \quad k \le K-1, \\
		\psitil_k(t) & = \psi_k(t), \quad k \le K-1, \\
		\psitil_K(t) & = \psitil_K(\tau) - \int_{[\tau,t]\times[0,1]} \one_{[0,\rtil_{K}(s))}(y)\varphi_{K}(s,y)\,ds\,dy \\
		& = \psi_K(\tau) - \int_{[\tau,t]\times[0,1]} \one_{[0,r_{K}(s))}(y)\varphi_{K}(s,y)\,ds\,dy = \psi_K(t),
	\end{align*}
	where the first line uses property (i) and the decreasing order of $\zeta_k$, $\zetatil_k$, the second line uses \eqref{eq:psi1} and \eqref{eq:psii},
	and the last line uses \eqref{eq:gap_K} and the $\varepsilon$-gap property of $\varphi_K$ in \eqref{eq:gap}.
	These together imply
	\begin{align*}
		\etatil_k(t) = \eta_k(t), \quad k \le K,
	\end{align*}
	by the uniqueness of solutions of the $K$-dimensional SP (associated with $(\VV^K, R_K)$; see proof of Lemma \ref{lem:SMap}).
	Hence \eqref{eq:claim_uniqueness_2} holds, which completes the proof.
\end{proof}

\subsection{Proof of Lemma \ref{lem:uniqueness_preparation_2}}\label{sec:uniqueness_preparation}

In this section we present the proof of Lemma \ref{lem:uniqueness_preparation_2}.
A key ingredient is Lemma \ref{lem:uniqueness_preparation} in which we show that one can suitably
approximate a $(\tilde \zeta, \tilde \psi) \in \clc_T$ that has a finite cost by a more regular trajectory for which 
 the length of the shortest queue $\pi(t)$ switches only a finite number of times.
This is accomplished by ``smoothing'' out small excursions during which $\pi(t)$ may change an infinite number of times.
Proof of Lemma \ref{lem:uniqueness_preparation} uses an inductive argument and for ease of presentation we first present the key inductive step separately in Lemma~\ref{lem:uniqueness_preparation_0}. Using this result we then complete the proof of
Lemma~\ref{lem:uniqueness_preparation}.
Finally, we use Lemma~\ref{lem:uniqueness_preparation} to complete the proof of Lemma \ref{lem:uniqueness_preparation_2}.
\begin{lemma}
	\label{lem:uniqueness_preparation_0}
	Fix $\sigma \in (0,1)$ and an integer $K \ge 2$. Suppose $(\zetatil,\psitil)\in\clc_T$ and $I_T(\zetatil,\psitil)<\iy$. Further suppose that there is
	a $\tilde N \in \NN$ and  a finite partition $0=\ttil_0 < \ttil_1 < \dotsb < \ttil_\Ntil = T$  such that on each $(\ttil_i,\ttil_{i+1})$, $\pitil(t)$ is either less than $K$, or some constant $\ctil_i \ge K$.
	Then there exists  a  $(\zeta,\psi)\in\clc_T$ such that
	\begin{enumerate}[a)]
	\item
		There exists a $N \in \NN$ and a finite partition $0=t_0 < t_1 < \dotsb < t_N = T$, such that on each $(t_i,t_{i+1})$, $\pi(t)$ is either less than $K-1$, or some constant $c_i \ge K-1$.
	\item
		$\|(\zeta,\psi)-(\zetatil,\psitil)\|_\infty \le \sigma$.
	\item
		$I_T(\zeta,\psi) \leq I_T(\zetatil,\psitil)+\sigma$.
	\end{enumerate}

\end{lemma}

\begin{proof}
	Let $\sigma, K$ and $(\tilde\zeta, \tilde \psi)\in \clc_T$ be as in the statement of the lemma.
	Fix $\varphitil\in\cls_T(\zetatil,\psitil)$ such that
	\begin{equation}
		\label{eq:cost_K}
		\sum_{k=0}^\iy\int_{[0,T]\times[0,1]} \vartheta_k\ell(\varphitil_k(s,y)) \,ds\,dy \le I_T(\zetatil,\psitil) + \frac{\sigma}{2}.
	\end{equation}
	Let $\etatil = \bar \Gamma(\tilde \psi)$, where $\bar \Gamma$ was introduced below Definition \ref{def:Smap}, that is, $(\zetatil,\etatil)$ solves the SP for $\psitil$ associated with the reflection matrix $R_\infty$.
	Since $(\zetatil_k,\psitil_k)_{k=1}^{K-1}$ are  uniformly continuous on $[0,T]$, there exists some $\varepsilon \le \sigma/4\Ntil K(\lambda+1)$ such that
	\begin{equation}
		\label{eq:uniform_continuity}
		\|(\zetatil_k(s_1),\psitil_k(s_1)) - (\zetatil_k(s_2),\psitil_k(s_2))\| \le \frac{\sigma}{4\Ntil K(\lambda+1)}, \quad \mbox{ for all } 1 \le k \le K-1, \: |s_1-s_2| \le \varepsilon.
	\end{equation}
	From the finiteness of the cost we can assume without loss of generality that $\varepsilon$ is such that
	\begin{equation}
		\label{eq:cost_uniform_continuity}
		\sum_{k=0}^\iy\int_{B\times[0,1]} \vartheta_k\ell(\varphitil_k(s,y)) \,ds\,dy
		\le \frac{\sigma}{4\Ntil K(\lambda+1)}, \mbox{ whenever } \leb(B) \le \varepsilon,
	\end{equation}
	where $\leb$ is the Lebesgue measure on $[0,T]$.
	With above preparation, let $(\zeta(0),\psi(0),\eta(0)) \doteq (\zetatil(0),\psitil(0),\etatil(0))$ and consider the interval $(\ttil_i,\ttil_{i+1}]$ for each $i=0,1,\dotsc,\Ntil-1$.
  In the argument that follows we will inductively construct $(\zeta,\psi,\eta,\varphi)$. 
  The sets $\GGnew, B, C, U, U_j$ will be introduced which depend on $i$, however, for ease of notation we will sometimes suppress this dependence on $i$ when it is clear from context.
	Consider the following two possible cases for $\pitil$.

	\underline{Case 1:} $\pitil(t)=\ctil_i \ge K$ for every $t \in (\ttil_i,\ttil_{i+1})$.
	In this case we define	
	\begin{equation}
		\label{eq:case1}
		\zeta(t)\doteq\zetatil(t), \quad \psi(t)\doteq\psi(\ttil_i)+\psitil(t)-\psitil(\ttil_i), \quad \eta(t)\doteq\eta(\ttil_i)+\etatil(t)-\etatil(\ttil_i), \quad \varphi(t,y)=\tilde \varphi(t,y)
	\end{equation}
	for $t \in (\ttil_i,\ttil_{i+1}]$, $y\in [0,1]$.

	\underline{Case 2:} $\pitil(t)<K$ for every $t \in (\ttil_i,\ttil_{i+1})$.
	Consider the set $U$ of time instants in which the shortest queue is less than $K-1$, namely
	\begin{equation*}
		U \doteq \{ t \in (\ttil_i,\ttil_{i+1}) : \pitil(t) < K-1 \}.
	\end{equation*}
	From the upper semi-continuity of $\pitil(t)$ it follows that $U$ is open and hence $U=\bigcup_{j=1}^\iy U_j$ for some disjoint open intervals $U_j$.
	Since the Lebesgue measure $\leb(U)<\ttil_{i+1}-\ttil_i < \infty$,
	we can express $U = \GGnew \cup B$, where $\GGnew\equiv \GGnew_i=\bigcup_{j=1}^m U_j$ is a union of finitely many  $U_j$'s for some $m \in \Nmb$, and $B\equiv B_i=\bigcup_{j=m+1}^\infty U_j$ with $\leb(B) < \varepsilon$.
	Let $C = (\ttil_i,\ttil_{i+1}) \setminus U$  be the set of time instants in $(\ttil_i,\ttil_{i+1})$ at which the shortest queue length is $K-1$, i.e.
	\begin{equation}
		\label{eq:C}
		C \doteq  \{ t \in (\ttil_i,\ttil_{i+1}) : \pitil(t) = K-1 \},
	\end{equation}
	and define the new trajectory as follows.
	When the shortest queue is of length $K-1$ or on a long excursion from $K-1$ (i.e.\ $t \in \GGnew \cup C$) the trajectory remains unchanged.
	Namely, let
	\begin{align}
		\label{eqn:newtraj1}
		\zeta(t) \doteq \zetatil(t), \quad \varphi(t,y) \doteq \varphitil(t,y), \quad t \in \GGnew \cup C, \quad  y \in [0,1].
	\end{align}
	Over short excursions from  $K-1$, namely when $t\in B$, we  ``smooth out'' the trajectories by setting the shortest queue equal to $K-1$ as follows.
	For $t \in B$ and $y \in [0,1]$, let
	\begin{align}
		\label{eqn:newtraj2}
		\begin{split}
		\zeta_k(t) & \doteq 1, \qquad k \le K-1,\qquad\zeta_k(t) \doteq \zetatil_k(t) < 1, \quad k \ge K, \\
		\varphi_0(t,y) & \doteq 0, \\
		\varphi_{k}(t,y) & \doteq 0, \quad k \le K-1, \qquad \varphi_{k}(t,y) \doteq \varphitil_{k}(t,y), \quad k \ge K.
		\end{split}
	\end{align}
	Having defined $\zeta$ and $\varphi$ over $(\ttil_i,\ttil_{i+1})$, let for $t \in (\ttil_i,\ttil_{i+1})$, 
	\begin{align}
		\label{eqn:newtraj3}
		\begin{split}
		\psi_1(t) & \doteq \psi_1(\ttil_i) + \int_{[\ttil_i,t]\times[0,1]}\varphi_{0}(s,y)\,ds\,dy\\
		&\qquad - \int_{[\ttil_i,t]\times[0,1]}\one_{[0,\zeta_1(s)-\zeta_{2}(s))}(y)\varphi_{1}(s,y)\,ds\,dy, \\
		\psi_k(t) & \doteq \psi_k(\ttil_i) - \int_{[\ttil_i,t]\times[0,1]}\one_{[0,\zeta_k(s)-\zeta_{k+1}(s))}(y)\varphi_{k}(s,y)\,ds\,dy,\qquad k\ge 2, \\
		\eta(t) & \doteq \eta(\ttil_i) + \int_{[\ttil_i,t] \setminus B} \, \etatil(ds) = \eta(\ttil_i) + \int_{[\ttil_i,t] \cap (\GGnew \cup C)} \, \etatil(ds),
		\end{split}
	\end{align}
	and define $(\zeta(\ttil_{i+1}),\psi(\ttil_{i+1}),\eta(\ttil_{i+1}))$ by continuity.

	Now we verify that $(\zeta,\psi,\eta)$ is the required trajectory, namely: (1) $(\zeta, \psi) \in \clc_T$, (2) $\eta = \bar \Gamma(\psi)$, (3) $\varphi \in \cls_T(\zeta, \psi)$
	and, (4) parts (a)--(c) of the lemma are satisfied.

	We will refer to $(\ttil_i,\ttil_{i+1})$ as a type 1 (resp. type 2) interval if it corresponds to Case 1 (resp. Case 2) and begin by making the following observations.
	\begin{itemize}
		\item For $k \ge K$, $(\zeta_k, \psi_k, \eta_k)= (\tilde \zeta_k, \tilde \psi_k, \tilde \eta_k)$. Indeed, the equality of the first coordinate $(\zeta_k= \tilde{\zeta}_k)$ is immediate from
		\eqref{eq:case1}, \eqref{eqn:newtraj1}, and \eqref{eqn:newtraj2}. The second coordinate equality $(\psi_k=\tilde\psi_k)$ follows from the fact that if $\psi_k(\tilde t_i) = \tilde \psi_k(\tilde t_i)$
		then, by \eqref{eq:case1} and the second line of \eqref{eqn:newtraj3}, $\psi_k(t) = \tilde \psi_k(t)$ for all $t \in [\tilde t_i, \tilde t_{i+1}]$. Similarly, the equality for the third
		coordinate $(\eta_k=\tilde\eta_k)$ follows from the fact that if $\eta_k(\tilde t_i) = \tilde \eta_k(\tilde t_i)$, then \eqref{eq:case1}, the third line of \eqref{eqn:newtraj3}, and the fact that $\tilde \eta_k$
		stays constant over $[\tilde t_i, \tilde t_{i+1}]$ for a type 2 interval, implies $\eta_k(t) = \tilde \eta_k(t)$ for all $t \in [\tilde t_i, \tilde t_{i+1}]$.
		\item If $(\tilde t_i, \tilde t_{i+1})$ is a type 1 interval, then $\tilde \zeta_k(\tilde t_i) = \zeta_k(\tilde t_i)$ for all $k \in \NN_0$. Indeed, the only case we need to consider is 
		when $k \le K-1$, $(\tilde t_{i-1}, \tilde t_{i})$ is a type 2 interval, and (making the dependence on $i$ explicit) $\tilde t_i \in \bar B_{i-1}$. 
		In this case $\zeta_k(\tilde t_i)=1$ and since $\tilde \pi(\tilde t_i) \ge \tilde \pi(t) \ge K$ for $t\in (\tilde t_i, \tilde t_{i+1})$, we have $\tilde \zeta_k(\tilde t_i)=1$ as well.
		\item The first two observations together with \eqref{eq:case1}, \eqref{eqn:newtraj1}, and \eqref{eqn:newtraj2} show that $\zeta_k$ is absolutely continuous for every $k \in \NN_0$.
		Also by construction, $\psi_k$ and $\eta_k$ are absolutely continuous as well for every $k \in \NN$. 
		\item $\zeta$ clearly satisfies parts (i)--(iii) in the definition of $\clc_T$.
		\item From the definition of $B$, \eqref{eq:case1}, and the third line of \eqref{eqn:newtraj3} we see that $\tilde \eta_{K-1}(t) = \eta_{K-1}(t)$ for all $t\in [0,T]$.
	\end{itemize}
		From the above observations we see that equations \eqref{eq:zeta_psi_eta} and \eqref{eq:psi1}--\eqref{eq:psii}  hold for all $i\ge K$. We now verify that \eqref{eq:zeta_psi_eta} and \eqref{eq:psi1}--\eqref{eq:psii} hold for $1\le i \le K-1$ as well.
		From the definition of $\etatil$, for $k=1, \ldots K$,  
		\begin{align}
			\zetatil_k(t) &= \psitil_k(t) +\etatil_k(t) - \etatil_k(t), \label{eq:eq512}\\
			 \etatil_k & \mbox{ is non-decreasing, } \etatil_k(0)=0, \quad \int_0^t \one_{\{\zetatil_k(s)<1\}} \, \etatil_k(ds) = 0,\quad k=1,\dotsc,K.\nonumber
		\end{align}
			It follows from \eqref{eqn:newtraj3} that for each $t$ in a type 2 interval $(\ttil_i,\ttil_{i+1}]$ and $k \le K$,
			\begin{align}\label{eqn:smallParts}
				\int_{[\ttil_i,t] } \one_{\{\zeta_k(s)<1\}} \, \eta_k(ds)
				= \int_{[\ttil_i,t] \cap (\GGnew \cup C)} \one_{\{\zeta_k(s)<1\}} \, \etatil_k(ds) = \int_{[\ttil_i,t] \cap (\GGnew \cup C)} \one_{\{\zetatil_k(s)<1\}} \, \etatil_k(ds) = 0,
			\end{align}
			where the first equality on the second line is from \eqref{eqn:newtraj1}. From \eqref{eq:case1} it is clear that the above equality also holds when 
			$(\ttil_i,\ttil_{i+1}]$ is a type 1 interval.
			
			It then remains to show
			\begin{equation}
				\label{eq:claim_1}
				\zeta_k(t) = \psi_k(t) +\eta_{k-1}(t) - \eta_k(t), \quad k \le K-1, \quad t \in (\ttil_i,\ttil_{i+1}).
			\end{equation}
			Once again, if $(\ttil_i,\ttil_{i+1})$ is a type 1 interval, from the observation in the second bullet above, $\tilde \zeta_k(\tilde t_i) = \zeta_k(\tilde t_i)$ and thus from \eqref{eq:case1}
			we see that \eqref{eq:claim_1} is satisfied in this case.  We now show \eqref{eq:claim_1}
			for a type 2 interval, given that \eqref{eq:claim_1} holds for $t=\ttil_i$.
For each $k \le K-1$, we have
			\begin{subequations}
				\begin{align}
					\zeta_k(t)-\zeta_k(\ttil_i) & = \int_{\ttil_i}^t \zeta_k'(s) \,ds\notag
					\\
					& = \int_{[\ttil_i,t] \cap (\GGnew \cup C)} \zeta_k'(s) \,ds \label{eqn:long2} \\
					& = \int_{[\ttil_i,t] \cap (\GGnew \cup C)} \zetatil_k'(s) \,ds \label{eqn:long3} \\
					& = \int_{[\ttil_i,t] \cap (\GGnew \cup C)} (\psitil_k'(s) +\etatil_{k-1}'(s) - \etatil_k'(s)) \,ds \label{eqn:long4} \\
					& = \int_{[\ttil_i,t] \cap (\GGnew \cup C)} \psi_k'(s) \,ds + (\eta_{k-1}(t) - \eta_{k-1}(\ttil_i)) - (\eta_{k}(t)-\eta_{k}(\ttil_i)) \label{eqn:long5} \\
					& = \int_{\ttil_i}^t \psi_k'(s) \,ds + (\eta_{k-1}(t) - \eta_{k-1}(\ttil_i)) - (\eta_{k}(t)-\eta_{k}(\ttil_i)) \label{eqn:long6} \\
					& = (\psi_k(t)-\psi_k(\ttil_i)) + (\eta_{k-1}(t) - \eta_{k-1}(\ttil_i)) - (\eta_{k}(t)-\eta_{k}(\ttil_i)), \notag
				\end{align}
			\end{subequations}			
			where line \eqref{eqn:long2} uses \eqref{eqn:newtraj2},
			line \eqref{eqn:long3} uses \eqref{eqn:newtraj1}, 
			line \eqref{eqn:long4} uses \eqref{eq:eq512},
			line \eqref{eqn:long5} uses \eqref{eqn:newtraj1} and \eqref{eqn:newtraj3},
			and line \eqref{eqn:long6} uses \eqref{eqn:newtraj2} and \eqref{eqn:newtraj3}.
			We have thus shown that \eqref{eq:claim_1}  is satisfied.   
			Next, from \eqref{eq:case1}, \eqref{eqn:newtraj1}, \eqref{eqn:newtraj2}, \eqref{eqn:newtraj3} and \eqref{eqn:smallParts} it is clear that 
			$\int_0^T \one_{\{\zeta_k(s)<1\}} \, \eta_k(ds) = 0$. Combining the above observations we now have that $(\zeta, \psi)\in \clc_T$ and that $\eta = \bar \Gamma(\psi)$ proving statements (1) and (2). Also from \eqref{eq:case1} and  \eqref{eqn:newtraj3} it is clear that $\varphi \in \cls_T(\zeta, \psi)$, proving statement (3). 
			
			Finally we prove statement (4), namely that parts (a)--(c) of the lemma hold. For part $(a)$, note that $\GGnew$ and $B \cup C = (\ttil_i,\ttil_{i+1}) \setminus \GGnew$ are both finite unions of disjoint intervals.
	Also, by construction,
	\begin{equation}
		\label{eq:pi_property}
		\pi(t) = K-1 \mbox{ for } t \in B \cup C = (\ttil_i,\ttil_{i+1})\setminus \GGnew\quad\mbox{ and } \quad \pi(t) < K-1 \mbox{ for } t \in \GGnew.
	\end{equation}
	Therefore part $(a)$ holds.

	We now consider part $(b)$.
	Fix $1 \le k \le K-1$.
	For $t \in (\ttil_i,\ttil_{i+1})$, where the latter is a type 2 interval,
	\begin{align*}
		|\psi_k(t) - \psitil_k(t)|
		& \le |\psi_k(\ttil_i) - \psitil_k(\ttil_i)| + \lambda \int_{B\times[0,1]} \left| \varphi_{0}(s,y) - \varphitil_{0}(s,y) \right| ds\,dy \\
		& \quad + \int_{B\times[0,1]} \left| \one_{[0,\zeta_k(s)-\zeta_{k+1}(s))}(y)\varphi_{k}(s,y) - \one_{[0,\zetatil_k(s)-\zetatil_{k+1}(s))}(y)\varphitil_{k}(s,y) \right| ds\,dy \\
		& \le |\psi_k(\ttil_i) - \psitil_k(\ttil_i)| + \int_{B\times[0,1]} \left( \lambda \varphitil_{0}(s,y) + \varphitil_{k}(s,y) \right) ds\,dy \\
		& \le |\psi_k(\ttil_i) - \psitil_k(\ttil_i)| + \int_{B\times[0,1]} \left( \lambda (\ell(\varphitil_{0}(s,y))+2) + \ell(\varphitil_{k}(s,y))+2 \right) ds\,dy \\
		& \le |\psi_k(\ttil_i) - \psitil_k(\ttil_i)| + \frac{\sigma}{4\Ntil K(\lambda+1)} + 2(\lambda+1)\varepsilon,
	\end{align*}
	where the second inequality uses $\varphi_0=\varphi_k=0$ on $B$, the third uses Lemma \ref{lem:ellProp}(b), and the last uses \eqref{eq:cost_uniform_continuity}.
	Also, the above inequality clearly holds  for a type 1 interval.

	Recalling the definition of $\varepsilon$, we have
	\begin{align*}
		\|\psi_k-\psitil_k\|_\infty & \le \Ntil\left(\frac{\sigma}{4\Ntil K(\lambda+1)} + 2(\lambda+1)\varepsilon\right) \le \frac{3\sigma}{4K}.
	\end{align*}
	It follows from \eqref{eqn:newtraj1} that, on a type 2 interval, $\zeta(t)=\zetatil(t)$ for $t \in \GGnew \cup C$.
	While for $t \in B$, we must have $t \in U_j \doteq (u_j,s_j)$ for some $j \in \Nmb$, such that $u_j \in C$ and $|t-u_j|<\varepsilon$.
	It then follows from \eqref{eqn:newtraj2}, \eqref{eq:C}, and \eqref{eq:uniform_continuity} that
	\begin{align*}
		|\zeta_k(t)-\zetatil_k(t)|
		& = |1 - \zetatil_k(t)| =  |\zetatil_k(u_j) - \zetatil_k(t)| \le \frac{\sigma}{4\Ntil K(\lambda+1)} \le \frac{\sigma}{4K}.
	\end{align*}
	Once again, on a type 1 interval the above inequality holds trivially.  
	Combining above estimates gives
	\begin{align*}
		\|(\zeta,\psi)-(\zetatil,\psitil)\|_\infty & \le \sum_{k=1}^{K-1} \|\zeta_k-\zetatil_k\|_\infty + \sum_{k=1}^{K-1} \|\psi_k-\psitil_k\|_\infty \le \frac{\sigma}{4} + \frac{3\sigma}{4} = \sigma.
	\end{align*}
	This verifies part $(b)$.

	Finally we consider part $(c)$.
	Using \eqref{eqn:newtraj1}, \eqref{eqn:newtraj2}, and the definitions of $B$ and $\varepsilon$ we have for a type 2 interval
	\begin{align*}
		& \sum_{k=0}^\iy \int_{[\ttil_i,\ttil_{i+1}]\times[0,1]} \left[\vartheta_k\ell(\varphi_{k}(s,y)) - \vartheta_k\ell(\varphitil_k(s,y))\right]ds\,dy \\
		&\qquad = \sum_{k=0}^{K-1} \int_{B\times[0,1]} \left[\vartheta_k\ell(\varphi_{k}(s,y)) - \vartheta_k\ell(\varphitil_k(s,y))\right]ds\,dy \\
		& \qquad\le K(\lambda+1)\ell(0) \frac{\sigma}{4\Ntil K(\lambda+1)}
		= \frac{\sigma}{4\Ntil}.
	\end{align*}
	The above bound holds clearly for a type 1 interval.
	From this and \eqref{eq:cost_K} we have
	\begin{align*}
		I_T(\zeta,\psi) & \leq \sum_{k=0}^\iy \int_{[0,T]\times[0,1]} \vartheta_k\ell(\varphi_{k}(s,y))\,ds\,dy \\
		& \le \sum_{k=0}^\iy \int_{[0,T]\times[0,1]} \vartheta_k\ell(\varphitil_k(s,y))\,ds\,dy + \frac{\sigma}{4} \\
		& \le I_T(\zetatil,\psitil) + \sigma.
	\end{align*}
	This gives part $(c)$ and completes the proof of the lemma.
\end{proof}
With Lemma~\ref{lem:uniqueness_preparation_0} in hand, we can now use an inductive argument to prove the following Lemma which will play a key role in the proof of Lemma~\ref{lem:uniqueness_preparation_2}.
\begin{lemma}
	\label{lem:uniqueness_preparation}
	Fix $\sigma\in(0,1)$ and  $(\zetatil,\psitil)\in\clc_T$ with $I_T(\zetatil,\psitil)<\iy$. There exists $(\zeta,\psi)\in\clc_T$ such that
	\begin{enumerate}[a)]
	\item
		There exists a $N \in \Nmb$ and a finite partition $0=t_0 < t_1 < \dotsb < t_N = T$ such that $\pi(t)$ is constant over each $(t_i,t_{i+1})$.
	\item
		$\|(\zeta,\psi)-(\zetatil,\psitil)\|_\infty \le \sigma$.
	\item
		$I_T(\zeta,\psi) \leq I_T(\zetatil,\psitil)+\sigma$.
	\end{enumerate}
\end{lemma}

\begin{proof}
	Let $M \in \Nmb$ be the smallest nonnegative integer such that
	$$ \frac{1}{M+1} \sup_{0 \le t \le T} \sum_{k=0}^M \zetatil_k(t) < 1$$
	and thus
	$$\sup_{0 \le t \le T} \zetatil_M(t) \le \frac{1}{M+1} \sup_{0 \le t \le T} \sum_{k=0}^M \zetatil_k(t) < 1.$$
	Existence of such a $M$ is a consequence of property (iii) of $\clc_T$.
	Therefore $\pitil(t) \doteq \max \{k:\zetatil_k(t) = 1\} < M$ for every $t \in [0,T]$.
	Let $(\zeta^M,\psi^M) \doteq (\zetatil,\psitil)$.
	For $K=M,M-1,\dotsc,2$, apply  Lemma \ref{lem:uniqueness_preparation_0} to $(\zeta^K,\psi^K)$ recursively, with $\sigma$ there replaced by $\sigma/M$, to get $(\zeta^{K-1},\psi^{K-1})$.
	Then $(\zeta,\psi)\doteq(\zeta^1,\psi^1)$ is the desired trajectory, on noting that over each $(t_i,t_{i+1})$, $\pi(t)$ is either less than $1$, which means $\pi(t)=0$, or $\pi(t)$ is some constant $c_i \ge 1$.
\end{proof}

We now prove Lemma \ref{lem:uniqueness_preparation_2} by further modifying the trajectory in Lemma \ref{lem:uniqueness_preparation} so that it has nice properties when $\pi(t)$ changes.

\begin{proof}[Proof of Lemma \ref{lem:uniqueness_preparation_2}]
	Fix $\sigma\in(0,1)$ and $(\tilde \zeta, \tilde \psi) \in \clc_T$ with $I_T(\tilde \zeta, \tilde \psi)<\infty$.
	By Lemma \ref{lem:uniqueness_preparation}, there exists $(\zetabar,\psibar)\in\clc_T$ such that
	\begin{enumerate}[(a)]
	\item
		There exists a $N\in \Nmb$ and a finite partition $0=\tbar_0 < \tbar_1 < \dotsb < \tbar_\Nbar = T$  such that $\pibar(t)=\cbar_i$ is constant over each $(\tbar_i,\tbar_{i+1})$.
	\item
		$\|(\zetabar,\psibar)-(\zetatil,\psitil)\|_\infty \le \frac{\sigma}{16}$.
	\item
		$I_T(\zetabar,\psibar) \leq I_T(\zetatil,\psitil)+\frac{\sigma}{16}$.
	\end{enumerate}
	Here $\pibar(t) \doteq \max\{k : \zetabar_k(t)=1\}$.
	Let $\Mbar \doteq \max\{\pibar(t) : t \in [0,T]\} < \infty$.
	Then with $\bar \eta = \bar \Gamma(\bar \psi)$, $\bar \eta_k(t)=0$ for all $t \in [0,T]$ and $k > \bar M$.
	Choose $\varphibar \in \Smc_T(\zetabar,\psibar)$ that is $\sigma/16$ optimal so that
	\begin{equation}\label{eq:eq859}\sum_{k=0}^\iy\int_{[0,T]\times[0,1]} \vartheta_k\ell(\varphibar_{k}(s,y)) \,ds\,dy \le I_T(\tilde \zeta,\tilde \psi) + \frac{\sigma}{8}.
		\end{equation}
	We will next modify $(\zetabar,\psibar,\varphibar)$ to get $(\zeta,\psi,\varphi)$ that satisfies properties (ii) and (v) in the statement while preserving properties (i), (iii), and (iv) that are satisfied by $(\zetabar,\psibar)$.
	Finally we will make one additional modification that will  guarantee that property (vi) holds as well.

	Since $(\zetabar_k,\psibar_k)_{k=0}^\Mbar$ are  uniformly continuous on $[0,T]$, there exists some $\varepsilon_\Mbar \in (0,\infty)$ such that
	\begin{equation}
		\label{eq:uniform_continuity_1}
		\|(\zetabar_k(s_1),\psibar_k(s_1))_{k=0}^\Mbar - (\zetabar_k(s_2),\psibar_k(s_2))_{k=0}^\Mbar\| \le \frac{\sigma}{8\Nbar} , \quad \mbox{ whenever } |s_1-s_2| \le \varepsilon_\Mbar.
	\end{equation}
	From the finiteness of the cost in \eqref{eq:eq859} 
	$$\sum_{k=0}^{\Mbar+1} \int_{[0,T]\times[0,1]} \vartheta_k \varphibar_k(t,y) \,dt\,dy \le \sum_{k=0}^{\Mbar+1} \int_{[0,T]\times[0,1]} \vartheta_k (\ell(\varphibar_k(t,y))+2) \,dt\,dy < \infty,$$
	 where the first inequality is from Lemma \ref{lem:ellProp} (b). Thus  we can  assume without loss of generality that $\varepsilon_\Mbar$ is small
	 enough so that
	\begin{equation}
		\label{eq:cost_ac}
		\sum_{k=0}^{\Mbar+1} \int_{B\times[0,1]} \vartheta_k \varphibar_k(t,y) \,dt\,dy \le \frac{\sigma}{16\Nbar}, \quad \mbox{ whenever } \leb(B) \le \varepsilon_\Mbar.
	\end{equation}
	Consider the interval $[\tbar_i,\tbar_{i+1})$ for each fixed $i=0,1,\dotsc,\Nbar-1$.
	We will usually suppress the dependence on $i$ (of $K, \Kbar, \delta$ below) for ease of notation.\\
	
	\noindent {\bf First modification of the trajectory.}

	\underline{Case I:}  {\em Either $\pibar(\bar t_i) = \cbar_i$ or $\int_{[\tbar_i,\tbar_i+\Delta \tbar_i]\times[0,1]} \varphibar_0(s,y)\,ds\,dy=0$ for some $\Delta \tbar_i \in (0,\tbar_{i+1}-\tbar_i)$.}  In this case define, for $t \in (\tbar_i,\tbar_{i+1})$
	\begin{equation}
		\label{eq:case1again}
		\zeta(t)=\zetabar(t), \quad \psi(t)=\psi(\tbar_i)+\psibar(t)-\psibar(\tbar_i), \quad \eta(t)=\eta(\tbar_i)+\etabar(t)-\etabar(\tbar_i), \quad \varphi(t)=\varphibar(t).
	\end{equation}
	
	\underline{Case II:}  {\em $\pibar(\bar t_i) > \cbar_i$ and $\int_{[\tbar_i,\tbar_i+\Delta \tbar_i]\times[0,1]} \varphibar_0(s,y)\,ds\,dy>0$
	for every $\Delta \tbar_i \in (0,\tbar_{i+1}-\tbar_i)$}. In this case we  modify $(\zetabar,\psibar,\varphibar)$ as follows.

	Take $\varepsilon = \varepsilon_\Mbar \wedge \frac{\sigma }{16\Nbar(\lambda+1)} \wedge \min_i (\tbar_{i+1} - \tbar_i)$.
	Let $\Kbar \doteq \pibar(\tbar_i) > \cbar_i \doteq K$.
	
	We first claim that we can assume that $(\bar \zeta, \bar \psi)$   and $\bar \varphi \in \cls_T(\bar \zeta, \bar \psi)$ are such that 
	$\varphibar_K(t)=0$ for $t \in (\tbar_i,\tbar_i+\varepsilon)$ whenever $K>0$, and
	parts (a)--(c) and \eqref{eq:eq859} hold with $\sigma/16$  and $\sigma/8$ replaced with $3\sigma/16$.
	
	To see this, suppose $K>0$.
	For $t \in (\tbar_i,\tbar_i+\varepsilon)$, since $\pibar(t)=K$, we have
	\begin{align*}
		\zetabar_k(t) & = 1, \quad k=0,\dotsc,K, \\
		\rbar_k(t) & = 0, \quad k=0,\dotsc,K-1, \quad \rbar_K(t) > 0, \\
		\etabar_k(t) & = \etabar_k(\tbar_i) + \lambda \int_{[\tbar_i,t]\times[0,1]} \varphibar_0(s,y)\,ds\,dy, \quad k=1,\dotsc,K-1.
	\end{align*}
	Note that the last statement is vacuous when $K=1$.
	The first property (with $k=K$) implies that
	\begin{align*}
		\etabar_{K-1}(t)+\psibar_K(t)-\etabar_{K-1}(s)-\psibar_K(s) \ge 0, \quad \forall\, \tbar_i <s < t < \tbar_i+\varepsilon,
	\end{align*}
	and hence for a.e. $t \in (\tbar_i,\tbar_i+\varepsilon)$
	\begin{equation}
		\label{eq:etabar_psibar}
		\etabar_{K-1}'(t)+\psibar_K'(t) = \lambda \int_0^1 \varphibar_0(t,y)\,dy - \int_0^1 \one_{[0,\rbar_K(t))}(y) \varphibar_K(t,y)\,dy \ge 0.
	\end{equation}
	We now modify $(\bar \zeta, \bar \psi, \bar \varphi)$ as follows.
	Replace $\varphibar_K(t,y)$, $\varphibar_0(t,y)$, and $\psibar_K$, 
	for $t \in (\bar t_i, \bar t_i+ \veps]$, by
	\begin{align*}
		& \varphibar_K^{new}(t,y) = 0,\quad  \varphibar_0^{new}(t,y)=\lambda^{-1}(\etabar_{K-1}'(t)+\psibar_K'(t)), \quad y \in [0,1], \\
		& (\psibar_K^{new})'(t) =  \one_{\{K=1\}} \lambda \int_0^1 \varphibar_0^{new}(t,y)\,dy.
	\end{align*}
	For $t \in (\bar t_i+\veps, \bar t_{i+1})$, we set 
	\begin{align*}
		\varphibar_j^{new}(t,\cdot) &= \varphibar_j(t,\cdot),\ j=0,K,\\
		\psibar_K^{new}(t) - \psibar_K^{new}(\bar t_i+\veps) &= \psibar_K(t) - \psibar_K(\bar t_i+\veps).
	\end{align*}
	Also, for $t \in (\bar t_i, \bar t_{i+1})$, define
	\begin{align*}
		(\psibar_j^{new})'(t) 
		= \psibar'_j(t),\ & j \neq 1, K, \\
		 \varphibar_j^{new}(t,\cdot) = \varphibar_j(t,\cdot),\ & j \neq 0,K, \\
	\zetabar_j^{new}(t) = \zetabar_j(t),\ & j \in \Nmb_0 
	\end{align*}	
	and
	$$(\psibar_1^{new})'(t) = \lambda \int_0^1 \varphibar_0^{new}(t,y)\,dy \mbox{ if } K>1.$$
	It is easy to check that the new trajectory is still in $\Cmc_T$ and has the finite partition property $(a)$ as $(\zetabar,\psibar)$.
	
	The contribution to the difference between the two trajectories over the interval $(\bar t_i, \bar t_i+\veps)$ can be estimated as 
	\begin{equation}
		\label{eq:extra_psi_diff}
		\begin{aligned}
		\int_{\tbar_i}^{\tbar_i+\varepsilon} |\psibar_K'(t)-(\psibar_K^{new})'(t)| \,dt +
		\int_{\tbar_i}^{\tbar_i+\varepsilon} |\psibar_1'(t)-(\psibar_1^{new})'(t)|\,dt&\le 2\lambda \int_{[\tbar_i,\tbar_i+\varepsilon]\times[0,1]} \varphibar_0(t,y) \,dt\,dy\\
		& \le \frac{\sigma}{8\Nbar},
		\end{aligned}
	\end{equation}
	where the last inequality is due to \eqref{eq:cost_ac}.
	
	The additional cost of making such a replacement is
	\begin{align}
		\label{eq:extra_cost}
		\begin{split}
			& \int_{[\tbar_i,\tbar_i+\varepsilon]\times[0,1]} \left( \lambda\ell(\varphibar_0^{new}(t,y)) + \ell(\varphibar_K^{new}(t,y)) - \lambda\ell(\varphibar_0(t,y)) - \ell(\varphibar_K(t,y)) \right)dt\,dy \\
			& \qquad \le \int_{\tbar_i}^{\tbar_i+\varepsilon} \left( \lambda\ell\left(\frac{\etabar_{K-1}'(t)+\psibar_K'(t)}{\lambda}\right) + \ell(0) - \lambda\ell\Big(\int_0^1\varphibar_0(t,y)\,dy\Big) - 0\right)dt \\
			& \qquad \le \int_{\tbar_i}^{\tbar_i+\varepsilon} \left( \lambda\ell(0) + \ell(0)\right)dt = (\lambda+1)\varepsilon \le \frac{\sigma}{16\Nbar},
		\end{split}
	\end{align}
	where the first inequality follows from the convexity of $\ell(\cdot)$ and the second inequality uses \eqref{eq:etabar_psibar} and the fact that $\ell(x)-\ell(y) \le \ell(0)$ for $0 \le x \le y$.
	
Combining the contributions to the errors and cost differences over all intervals $(\bar t_i, \bar t_{i+1})$, we have
$$\|(\bar\zeta^{new},\bar\psi^{new})-(\tilde\zeta,\tilde\psi)\|_\infty \le \|(\bar\zeta^{new},\bar\psi^{new})-(\bar\zeta,\bar\psi)\|_\infty +\frac{\sigma}{16} \le \frac{3\sigma}{16}$$
and 
\begin{align*}
I_T(\bar\zeta^{new},\bar\psi^{new}) &\le \sum_{k=0}^\iy\int_{[0,T]\times[0,1]} \vartheta_k\ell(\varphibar_{k}^{new}(s,y)) \,ds\,dy \\
&\le \sum_{k=0}^\iy\int_{[0,T]\times[0,1]} \vartheta_k\ell(\varphibar_{k}(s,y)) \,ds\,dy + \frac{\sigma}{16}\\
&\le I_T(\tilde\zeta,\tilde\psi) + \frac{3\sigma}{16}.
\end{align*}
We have thus proved the claim.

Abusing notation, we denote $(\bar\zeta^{new},\bar\psi^{new}, \bar \varphi^{new})$ once more as $(\bar\zeta,\bar\psi, \bar \varphi)$ and recall that parts (a)--(c) and
\eqref{eq:eq859} hold with $\sigma/16$ and $\sigma/8$ replaced with $3\sigma/16$.

Since $\varphibar_K(t,y) = 0$ for $t \in (\tbar_i,\tbar_i+\varepsilon)$ when $K>0$, we have
	\begin{equation}
		\label{eq:etabar_K}
		\etabar_k(t) = \etabar_k(\tbar_i) + \lambda \int_{[\tbar_i,t]\times[0,1]} \varphibar_0(s,y)\,ds\,dy, \quad k=1,\dotsc,K.
	\end{equation}

	We now return to constructing our modification of 
	 $(\zetabar,\psibar,\varphibar)$ on $[\tbar_i,\tbar_{i+1})$ under Case II.
	Since  $\zetabar_{K+1}(t)<1$ over $(\bar t_i, \bar t_{i+1})$ and $\bar K >K$ in Case II, we have
	\begin{align*}
		0 &> \zetabar_{K+1}(\tbar_i+\varepsilon) - 1 = \zetabar_{K+1}(\tbar_i+\varepsilon) - \zetabar_{K+1}(\tbar_i)\\
		& = \lambda \int_{[\tbar_i,\tbar_i+\varepsilon]\times[0,1]} \varphibar_0(s,y) \,ds\,dy - \int_{[\tbar_i,\tbar_i+\varepsilon]\times[0,1]} \one_{[0,\rbar_{K+1}(s))}(y) \varphibar_{K+1}(s,y) \,ds\,dy.
	\end{align*}
	Let $\delta$ be  the largest value in $(0,\varepsilon)$ such that
	\begin{equation}
		\label{eq:delta}
		\lambda \int_{[\tbar_i,\tbar_i+\varepsilon]\times[0,1]} \varphibar_0(s,y) \,ds\,dy = \int_{[\tbar_i,\tbar_i+\delta]\times[0,1]} \one_{[0,\rbar_{K+1}(s))}(y) \varphibar_{K+1}(s,y) \,ds\,dy.
	\end{equation}
	We  now modify $(\bar \zeta, \bar \psi, \bar \varphi)$ on the time interval $[\tbar_i,\tbar_{i+1})$ as follows. Let
	\begin{align*}
		\varphi_0(t,y) & = 0, \quad t \in [\tbar_i,\tbar_i+\varepsilon], \\
		\varphi_{K+1}(t,y) & = 0, \quad t \in [\tbar_i,\tbar_i+\delta], \\
		 \varphi_{K+1}(t,y) &= \varphibar_{K+1}(t,y) \one_{[0,\rbar_{K+1}(t))}(y), \quad t \in (\tbar_i+\delta,\tbar_i+\varepsilon], \\
		\varphi_k(t,y) & = \varphibar_k(t,y), \quad t \in [\tbar_i,\tbar_i+\varepsilon], \quad k \ne 0,K+1.\\
		\varphi_k(t,y) & = \varphibar_k(t,y), \quad t \in [\tbar_i+\varepsilon, \bar t_{i+1}) \mbox{ for all } k \in \NN_0.
	\end{align*}
	Having made such a modification to the control over each interval $(\bar t_i, \bar t_{i+1})$, consider
 $(\zeta,\psi,\eta)$ driven by $\varphi$, given (on $(\bar t_i, \bar t_{i+1})$) as follows,
	\begin{align}\label{eqn:trajdef}
		\begin{split}
		\eta_k'(t) & = 0, \quad k \ge 1, \quad t \in [\tbar_i,\tbar_i+\varepsilon], \\
		\zeta_k(t) & = \zetabar_k(t) = 1, \quad \psi_k'(t) = 0, \quad k=0,\dotsc,K, \quad t \in [\tbar_i,\tbar_i+\varepsilon], \\
		\zeta_k(t) & = \zetabar_k(t), \quad \psi_k'(t) = \psibar_k'(t), \quad k \ge K+2, \quad t \in [\tbar_i,\tbar_i+\varepsilon], \\
		\zeta_{K+1}(t) & = 1, \quad t \in [\tbar_i,\tbar_i+\delta], \\
		\zeta_{K+1}(t) & = 1 - \int_{[\tbar_i+\delta,t]\times[0,1]} \one_{[0,\rbar_{K+1}(s))}(y) \varphibar_{K+1}(s,y) \,ds\,dy, \quad t \in (\tbar_i+\delta,\tbar_i+\varepsilon], \\
		\psi_{K+1}'(t) & = 0, \quad t \in [\tbar_i,\tbar_i+\delta], \\
		\psi_{K+1}(t) &= \psi_{K+1}(\tbar_i+\delta) - \int_{[\tbar_i+\delta,t]\times[0,1]} \one_{[0,\rbar_{K+1}(s))}(y) \varphibar_{K+1}(s,y) \,ds\,dy, \quad t \in (\tbar_i+\delta,\tbar_i+\varepsilon], \\
		(\zeta'(t),& \psi'(t),\eta'(t)) = (\zetabar'(t),\psibar'(t),\etabar'(t)), \quad t \in (\tbar_i+\varepsilon,\tbar_{i+1}).
		\end{split}
	\end{align}

	We now check that $(\zeta,\psi) \in \Cmc_T$ and $\varphi \in \Smc_T(\zeta,\psi)$.
	For this it suffices to check the evolution of the $K$-th and ($K+1$)-th coordinates on each $[\tbar_i,\tbar_i+\varepsilon]$.
	
	If $K=0$, then $\psi_0' =\zeta_0'=\eta'=0$ by construction, and it is clear that \eqref{eq:zeta_psi_eta} holds for the $K$-th coordinate.
	
	If $K>0$, then $\zeta_K'=\psi_K'=\eta'=\varphi_0=\varphi_K = \varphibar_K=0$ by construction and the claim made below \eqref{eq:case1again} (which has been verified), showing that \eqref{eq:zeta_psi_eta}, \eqref{eq:psi1}, and \eqref{eq:psii} hold once more for the $K$-th coordinate.
	Therefore we have the desired evolution of the $K$-th coordinate.
	
	For $K \in \Nmb_0$ and $t \in [\tbar_i,\tbar_i+\delta]$, since $\zeta_{K+1}'=\psi_{K+1}'=\eta'=\varphi_0=\varphi_{K+1}=0$, we have \eqref{eq:zeta_psi_eta}, \eqref{eq:psi1}, and \eqref{eq:psii} for the $(K+1)$-th coordinate over this interval.
	As for $t \in (\tbar_i+\delta,\tbar_i+\varepsilon]$, note that
	\begin{align*}
		\zeta_{K+1}(t) - \zetabar_{K+1}(t) & = \left( 1 - \int_{[\tbar_i+\delta,t]\times[0,1]} \one_{[0,\rbar_{K+1}(s))}(y) \varphibar_{K+1}(s,y) \,ds\,dy \right) \\
		& \quad - \left( 1 + \lambda\int_{[\tbar_i,t]\times[0,1]} \varphibar_0(s,y) \,ds\,dy - \int_{[\tbar_i,t]\times[0,1]} \one_{[0,\rbar_{K+1}(s))}(y) \varphibar_{K+1}(s,y) \,ds\,dy \right) \\
		& = \int_{[\tbar_i,\tbar_i+\delta]\times[0,1]} \one_{[0,\rbar_{K+1}(s))}(y) \varphibar_{K+1}(s,y) \,ds\,dy - \lambda\int_{[\tbar_i,t]\times[0,1]} \varphibar_0(s,y) \,ds\,dy \\
		& = \lambda \int_{[t,\tbar_i+\varepsilon]\times[0,1]} \varphibar_0(s,y) \,ds\,dy \ge 0,
	\end{align*}
	where the first equality uses \eqref{eq:etabar_K} and the last equality follows from \eqref{eq:delta}.
	The above equality in particular implies that
	\begin{equation}
		\label{eq:t+eps}
		\zeta_{K+1}(\tbar_i+\varepsilon)=\zetabar_{K+1}(\tbar_i+\varepsilon),
	\end{equation}
	and the inequality $\zeta_{K+1}(t) - \zetabar_{K+1}(t)\ge 0$  together with the fact that $\zeta_{K+2}(t) = \zetabar_{K+2}(t)$ for $t\in(\tbar_i,\tbar_{i+1})$ gives $r_{K+1}(t)  \ge \rbar_{K+1}(t)$ for $t \in (\tbar_i+\delta,\tbar_i+\varepsilon]$.
	From this and the definition of $\varphi_{K+1}, \zeta_{K+1}, \psi_{K+1}$, we see that for $t \in (\tbar_i+\delta,\tbar_i+\varepsilon]$,
	\begin{align*}
		\zeta_{K+1}(t) & = 1 - \int_{[\tbar_i+\delta,t]\times[0,1]} \one_{[0,\rbar_{K+1}(s))}(y) \varphibar_{K+1}(s,y) \,ds\,dy \\
		& = \zeta_{K+1}(\tbar_i+\delta) - \int_{[\tbar_i+\delta,t]\times[0,1]} \one_{[0,r_{K+1}(s))}(y) \one_{[0,\rbar_{K+1}(s))}(y) \varphibar_{K+1}(s,y) \,ds\,dy \\
		& = \zeta_{K+1}(\tbar_i+\delta) - \int_{[\tbar_i+\delta,t]\times[0,1]} \one_{[0,r_{K+1}(s))}(y) \varphi_{K+1}(s,y) \,ds\,dy
	\end{align*}
	and similarly
	\begin{equation*}
		\psi_{K+1}(t) = \psi_{K+1}(\tbar_i+\delta) - \int_{[\tbar_i+\delta,t]\times[0,1]} \one_{[0,r_{K+1}(s))}(y) \varphi_{K+1}(s,y) \,ds\,dy.
	\end{equation*}
	From these two displays along with \eqref{eqn:trajdef} and the observation that $\zeta_{K+1}'-\psi_{K+1}'=0=\eta'$ we have verified \eqref{eq:zeta_psi_eta}, \eqref{eq:psi1}, and \eqref{eq:psii} for the $(K+1)$-th coordinate when $t \in (\tbar_i+\delta,\tbar_i+\varepsilon]$.
	This proves that $(\zeta,\psi) \in \Cmc_T$ and $\varphi \in \Smc_T(\zeta,\psi)$.
	
	Since $\delta$ is the largest value in $(0,\varepsilon)$ such that \eqref{eq:delta} holds, from the definition of $\zeta_{K+1}$ we must have $\zeta_{K+1}(t) < 1$ for $t \in (\tbar_i+\delta,\tbar_i+\varepsilon]$.
	This implies that $\pi(t)=K$ is constant for $t \in (\tbar_i+\delta,\tbar_i+\varepsilon)$.
	
	On $[\tbar_i,\tbar_i+\delta]$, since $\varphi_0(t,y)=0$, $\pi(t)$ must be non-increasing, and hence must be a piecewise constant function which can be decomposed into a finite number of intervals.
	Therefore properties (i) and (ii) hold. 
	
	We now estimate the (possible) increase in cost.
	Note that
	\begin{align*}
		& \int_{[\tbar_i,\tbar_i+\varepsilon]\times[0,1]} \left( \lambda\ell(\varphi_0(s,y)) + \ell(\varphi_{K+1}(s,y)) - \lambda\ell(\varphibar_0(s,y)) - \ell(\varphibar_{K+1}(s,y)) \right)ds\,dy \\
		& \qquad\le \int_{[\tbar_i,\tbar_i+\varepsilon]\times[0,1]} (\lambda+1)\ell(0)\,ds\,dy \le (\lambda+1)\varepsilon.
	\end{align*}
	Therefore the cost after making such modifications to each interval $(\tbar_i,\tbar_i+\varepsilon)$ can be bounded as
	\begin{align}
		\label{eq:cost_estimate_uniqueness_lem}
		\begin{split}
		\sum_{k=0}^\iy\int_{[0,T]\times[0,1]} \vartheta_k\ell(\varphi_{k}(s,y)) \,ds\,dy & \le \sum_{k=0}^\iy\int_{[0,T]\times[0,1]} \vartheta_k\ell(\varphibar_{k}(s,y)) \,ds\,dy + (\lambda+1)\Nbar\varepsilon\\
		&  \le I_T(\zetatil,\psitil) + \frac{3\sigma}{16} + \frac{\sigma}{16} = I_T(\zetatil,\psitil) + \frac{\sigma}{4}.
		\end{split}
	\end{align}
The difference in the trajectories is estimated as follows.
	Since $\zeta_{K+1}(t)$ is a monotone decreasing interpolation of $\zeta_{K+1}(\tbar_i)$ and $\zeta_{K+1}(\tbar_i+\varepsilon)$ for $t \in [\tbar_i,\tbar_i+\varepsilon]$, we have
	\begin{align*}
		\|\zeta_{K+1}(t)-\zetabar_{K+1}(t)\| & \le \|\zeta_{K+1}(\tbar_i)-\zetabar_{K+1}(t)\| + \|\zeta_{K+1}(\tbar_i+\varepsilon)-\zetabar_{K+1}(t)\| \\
		& \le \|\zeta_{K+1}(\tbar_i)-\zetabar_{K+1}(\tbar_i)\| + \frac{\sigma}{8\Nbar} + \|\zeta_{K+1}(\tbar_i+\varepsilon)-\zetabar_{K+1}(\tbar_i+\varepsilon)\| + \frac{\sigma}{8\Nbar} \\
		& = \frac{\sigma}{4\Nbar},
	\end{align*}
	where the second inequality uses \eqref{eq:uniform_continuity_1} and the last equality follows from \eqref{eq:t+eps}.
	Therefore
	\begin{equation}
		\label{eq:diff_zeta}
		\|\zeta-\zetabar\|_\iy \le \frac{\sigma}{4\Nbar}.
	\end{equation}
	For each $k=1,\dotsc,\Mbar+1$ and $t \in [\tbar_i,\tbar_i+\varepsilon]$, using the definition of $\psi_k$ one has
	\begin{align*}
		|\psi_k(t)-\psibar_k(t)|
		& \le |\psi_k(\tbar_i)-\psibar_k(\tbar_i)| + \int_{[\tbar_i,\tbar_i+\varepsilon]\times[0,1]} (\lambda \varphibar_0(s,y) + \varphibar_k(s,y) )\,ds\,dy.
	\end{align*}
	As a result
	\begin{align*}
		\|\psi-\psibar\|_\infty \le \sum_{k=1}^{\Mbar+1} \|\psi_k-\psibar_k\|_\infty \le \sum_{i=0}^{\Nbar-1} \sum_{k=1}^{\Mbar+1} \int_{[\tbar_i,\tbar_i+\varepsilon]\times[0,1]} (\lambda \varphibar_0(s,y) + \varphibar_k(s,y) )\,ds\,dy \le \frac{\sigma}{8}
	\end{align*}
	by \eqref{eq:cost_ac}.
	From this and \eqref{eq:diff_zeta} we have
	\begin{equation*}
		\|(\zeta,\psi)-(\zetabar,\psibar)\|_\infty \le \frac{\sigma}{16} + \frac{\sigma}{4\Nbar}  \le \frac{5\sigma}{16}.
	\end{equation*}
	This combined with $\|(\zetabar,\psibar)-(\zetatil,\psitil)\|_\infty \le 3\sigma/16$ gives property (iii) with $\sigma$ replaced by $\sigma/2$.

\noindent {\bf Second modification of the trajectory.}

	We now introduce one last  modification to $(\zeta,\psi,\varphi)$ so that property (vi) holds.
	Take $M > \Mbar + 2$ large enough such that
	\begin{equation}
		\label{eq:extra_M}
		\sum_{k=M}^\infty \frac{\|(\zeta_k,\psi_k)\|_\infty+2}{2^k} \le \frac{\sigma}{2}.
	\end{equation}
	Define $(\zeta^{new},\psi^{new}) \in \Cmc_T$ and $\varphi^{new} \in \Smc_T(\zeta^{new},\psi^{new})$ by
	\begin{align*}
		(\zeta_k^{new},\psi_k^{new},\varphi_k^{new}) & = (\zeta_k,\psi_k,\varphi_k), \quad k < M, \\
		\varphi_k^{new} & = 1, \quad \zeta_k^{new}(t)=\psi_k^{new}(t) = x_k - \int_0^t (\zeta_k^{new}(s)-\zeta_{k+1}^{new}(s)) \,ds, \quad k \ge M.
	\end{align*}
	Note that
	\begin{align}
		& \|(\zeta^{new},\psi^{new})-(\zeta,\psi)\|_\infty = \sum_{k=M}^\infty \frac{\|(\zeta_k^{new},\psi_k^{new})-(\zeta_k,\psi_k)\|_\infty}{2^k} \le \sum_{k=M}^\infty \frac{\|(\zeta_k,\psi_k)\|_\infty+2}{2^k} \le \frac{\sigma}{2}, \label{eq:diff_new} \\
		& \sum_{k=0}^\iy\int_{[0,T]\times[0,1]} \vartheta_k\ell(\varphi_{k}^{new}(s,y)) \,ds\,dy \le \sum_{k=0}^\iy\int_{[0,T]\times[0,1]} \vartheta_k\ell(\varphi_{k}(s,y)) \,ds\,dy, \label{eq:cost_new}
	\end{align}
	where the last line uses $\ell(1)=0$. 
	Once again, abusing notation, we denote $(\zeta^{new},\psi^{new},\varphi^{new})$ as $(\zeta,\psi,\varphi)$.
Clearly, properties (i), (ii), (iii), (iv) and (vi) are satisfied by the above modification.

For property (v), let $\varphi^\sigma \in \Smc_T(\zeta,\psi)$ be any $\sigma/2$ optimal control. Choose $\delta^*$ sufficiently small so that
$$\sum_{k=0}^\iy \sum_{i=0}^{\bar N-1}\int_{[ t_i, t_i+\delta^*]\times[0,1]} \vartheta_k\ell(\varphi_{k}(s,y)) \,ds\,dy \le \frac{\sigma}{2}.$$
Define, for each $i$, $\varphi^*(t)=\varphi(t)$ for $t \in (t_i,t_i+\delta^*)$, and for $t \in [t_i+\delta^*, t_{i+1})$, let $\varphi_k^*(t)=1$ for $k \ge M$, and $\varphi^*_k(t)=\varphi_k^\sigma(t)$ otherwise.

	Then $\varphi^*$ is a $\sigma$ optimal control, without affecting properties (ii) or (vi). Thus properties (i)--(vi) are satisfied completing the proof of the lemma.
%
%
%
%
%
	%
\end{proof}

\section{Compact Sub-level Sets}\label{sec:compactSets}
In this section we prove the third statement in the proof of Theorem \ref{thm:mainResult}, namely the property that $I_T$ is a rate function.
For this we need to show that for every $M \in \NN$, the set $\Gamma_M \doteq \{(\zeta,\psi)\in \Dmb_{\Rmb^\infty \times \Rmb^\infty}:
I_T(\zeta, \psi)\le M\}$ is compact.
Now fix such a $M$ and a sequence $\{(\zeta^n,\psi^n)\} \subset \Gamma_M$. It suffices to show that the sequence has a convergent subsequence with the limit in the set $\Gamma_M$.
From the definition of $I_T$, it follows that $(\zeta^n,\psi^n)\in\clc_T$
%
%
%
and there exists a control $\varphi^n\in\cls_T(\zeta^n,\psi^n)$ such that for every $n$
\begin{equation}
	\label{eqn:controlBoundCompact}
	\sum_{k=0}^\iy\int_{[0,T]\times[0,1]} \vartheta_k\ell(\varphi^n_k(s,u))\,ds\,dy
	\leq I_T(\zeta^n,\psi^n)+\frac{1}{n}
	\leq M+\frac{1}{n}.
\end{equation}
We follow the convention that $\zeta^n_0 = \psi^n_0 = 1$.
We first show pre-compactness of the sequence $\{(\zeta^n,\psi^n, \varphi^n)\}_{n\in\NN_0}$. 
Recall the compact metric spaces $S_N$, for $N \in \NN$, introduced in Section \ref{sec:tight}.
\begin{lemma}
	\label{lem:tightCompact}
	The sequence
	$\{(\zeta^n,\psi^n, \varphi^n)\}_{n\in\NN_0}$ is pre-compact in $\Cmb_{\Rmb^\infty \times \Rmb^\infty}\times S_{M+1}$.
\end{lemma}

\begin{proof}
	%
	Pre-compactness of $\{\varphi^n\}_{n\in\NN_0}$ is  immediate from the compactness of $S_{M+1}$.

	We next prove pre-compactness of $\{\psi^n_k\}$ for each fixed $k \in \NN$.
	 From the definition of $\psi^n$ in \eqref{eq:psi1}--\eqref{eq:psii} we have 
	 $$ |\psi^n_k(t)| \le 1 + \int_{[0,T]\times[0,1]} \lambda \varphi_0^n(t,y)\,dt\,dy 
	 + \int_{[0,T]\times[0,1]}  \varphi_k^n(t,y)\,dt\,dy.$$
	 It then follows from \eqref{eqn:controlBoundCompact} and Lemma \ref{lem:ellProp} that
	 $$\sup_{n\in \NN} \sup_{t\in [0,T]} |\psi^n_k(t)|<\infty.$$
	 We now show that  $\{\psi^n_k\}$ is equicontinuous. 
	Note that for any $0 < t-s \le  \del$ and $K>0$,
	\begin{align*}
		|\psi_k^n(t)-\psi_k^n(s)|
		&\leq \lambda \int_{[s,t]\times[0,1]}\varphi^n_0(u,y)\,du\,dy  + \int_{[s,t]\times[0,1]} \varphi^n_k(u,y) \,du\,dy\\
		&\leq \lambda \int_{[s,t]\times[0,1]}\varphi^n_0(u,y) \one_{\{\varphi^n_0(u,y)>K\}}du\,dy+
		\lambda \int_{[s,t]\times[0,1]}\varphi^n_0(u,y) \one_{\{\varphi^n_0(u,y)\le K\}}du\,dy\\
		&\qquad +\int_{[s,t]\times[0,1]} \varphi^n_k(u,y)\one_{\{\varphi^n_k(u,y)>K\}} \,du\,dy\\
		&\qquad+ \int_{[s,t]\times[0,1]} \varphi^n_k(u,y)\one_{\{\varphi^n_k(u,y)\leq K\}} \,du\,dy\\
		&\leq (\lambda+1)\gamma(K)(M+1) + (\lambda+1)K\delta
	\end{align*}
	where the final inequality above follows from Lemma \ref{lem:ellProp}(a) and \eqref{eqn:controlBoundCompact}.
	Therefore,
	\begin{equation*}
		\lim_{\del\to\iy}\sup_{n\in\NN}\sup_{|t-s|\leq \del}|\psi_k^n(t)-\psi_k^n(s)|\leq (\lambda+1)\gamma(K)(M+1)
	\end{equation*}
	and equicontinuity of $\{\psi^n_k\}$  follows upon sending $K\to\iy$. Pre-compactness of $\{\psi^n_k\}$ for each $k$, and therefore of $\{\psi^n\}$, now follows from
	the Arzela-Ascoli Theorem.
	Pre-compactness of $\{\zeta^n\}_{n\in\NN}$ in $\Cmb_{\Rmb^\infty}$  follows immediately from the precompactness of $\{\psi^n\}$ and the
	Lipschitz property of the Skorokhod map proved in Lemma \ref{lem:SMap}.
\end{proof}

We now characterize the limit points of $(\zeta^n,\psi^n, \varphi^n)$.

\begin{lemma}
	\label{lem:convergencCompact}
	Suppose $(\zeta^n,\psi^n, \varphi^n)$ converges along a subsequence to $(\zeta,\psi, \varphi)\in \Cmb_{\Rmb^\infty \times \Rmb^\infty}\times S_{M+1}$.
	 Then
	\begin{enumerate}
	\item[a)]
		$\sum_{k=0}^\iy\int_{[0,T]\times[0,1]} \vartheta_k \ell(\varphi_k(s,y))\,ds\,dy\leq M$.
	\item[b)]
		For each $t\in[0,T]$,
		\begin{align*}
			\psi_1(t) & = x_1 + \lambda \int_{[0,t]\times[0,1]}\varphi_0(s,y)\,ds\,dy  - \int_{[0,t]\times[0,1]}\one_{[0,\zeta_1(s)-\zeta_2(s))}(y) \varphi_{1}(s,y) \,ds\,dy, \\
			\psi_k(t) & = x_k - \int_{[0,t]\times[0,1]}\one_{[0,\zeta_k(s)-\zeta_{k+1}(s))}(y) \varphi_k(s,y) \,ds\,dy, \quad k \ge 2.
		\end{align*}
	\item[c)]
		$(\zeta,\psi) \in \Cmc_T$ and $\varphi \in \Smc_T(\zeta,\psi)$.
	\end{enumerate}
\end{lemma}

\begin{proof}
	Part $(a)$ is immediate from  \cite[Lemma A.1]{budhiraja2013large}, \eqref{eqn:controlBoundCompact}, and Fatou's lemma.

	We now prove $(b)$. The convergence  
	\begin{equation}\int_{[0,t]\times[0,1]} \varphi^n_0(s,y)\,ds\,dy \to
	\int_{[0,t]\times[0,1]} \varphi_0(s,y)\,ds\,dy
	\label{eq:eq510}
	\end{equation}
	 is immediate from the definition of the topology on $S_{M+1}$ (see the comment above Lemma \ref{lem:tightness}).
	Consider now the integral on the right side of \eqref{eq:psii}. For each  $k \in \NN$,
	\begin{align*}
		&\left|\int_{[0,t]\times[0,1]}\one_{[0,\zeta^n_k(s)-\zeta^n_{k+1}(s))}(y)\varphi^n_k(s,y)\,ds\,dy-\int_{[0,t]\times[0,1]}\one_{[0,\zeta_k(s)-\zeta_{k+1}(s))}(y)\varphi_k(s,y)\,ds\,dy\right|\\
		&\qquad \leq \int_{[0,t]\times[0,1]}|\one_{[0,\zeta^n_k(s)-\zeta^n_{k+1}(s))}(y)-\one_{[0,\zeta_k(s)-\zeta_{k+1}(s))}(y)|\varphi^n_k(s,y)\,ds\,dy\\
		&\qquad\qquad + \left| \int_{[0,t]\times[0,1]}\one_{[0,\zeta_k(s)-\zeta_{k+1}(s))}(y)(\varphi^n_k(s,y)-\varphi_k(s,y))\,ds\,dy \right|.
	\end{align*}
	Using the convergence of
	\begin{align*}
		|\one_{[0,\zeta^n_k(s)-\zeta^n_{k+1}(s))}(y)-\one_{[0,\zeta_k(s)-\zeta_{k+1}(s))}(y)|\to 0
	\end{align*}
	for $\leb_t$-a.e. $(s,y)\in[0,t]\times[0,1]$ and the uniform integrability of
$(s,y)\mapsto \varphi^n_k(s,y)$ with respect to the normalized Lebesgue measure on $[0,t]\times [0,1]$,
	guaranteed by \eqref{eqn:controlBoundCompact} and the superlinearity of $\ell$, we have
	\begin{equation*}
		\int_{[0,t]\times[0,1]}|\one_{[0,\zeta^n_k(s)-\zeta^n_{k+1}(s))}(y)-\one_{[0,\zeta_k(s)-\zeta_{k+1}(s))}(y)|\varphi^n_k(s,y)\,ds\,dy\to 0.
	\end{equation*}
	From the convergence of $\varphi^n \to \varphi$ and recalling the topology on $S_{M+1}$, we have that
	\begin{equation*}
		\int_{[0,t]\times[0,1]}\one_{[0,\zeta_k(s)-\zeta_{k+1}(s))}(y)(\varphi^n_k(s,y)-\varphi_k(s,y))\,ds\,dy
		\to 0.
	\end{equation*}
	This gives the convergence:
	\begin{equation*}
		 \int_{[0,t]\times[0,1]}\one_{[0,\zeta^n_k(s)-\zeta^n_{k+1}(s))}(y)\varphi^n_k(s,y)\,ds\,dy \to
	\int_{[0,t]\times[0,1]}\one_{[0,\zeta_k(s)-\zeta_{k+1}(s))}(y)\varphi_k(s,y)\,ds\,dy.\end{equation*}
	Combining this with \eqref{eq:eq510} gives $(b)$.
	
	Finally consider part $(c)$. The fact that $(\zeta, \psi)$ satisfies property (ii) and  (iv) of $\clc_T$ is an immediate consequence of the
	  fact that $(\zeta^n,\psi^n)$ satisfy these properties and the Lipschitz property of the Skorokhod map proved in Lemma \ref{lem:SMap}.
	  Property (i) follows from this and Remark \ref{rem:remskor}. For property (iii) note that 
  	\begin{align*}
  		 \sup_{0\leq t\leq T}\sum_{i=1}^\iy\zeta_i(t)
  		& \le \liminf_{n\to\iy}  \sup_{0\leq t\leq T} \sum_{i=1}^\iy \zeta^n_i(t) \\
  		& \leq  \sum_{i=1}^\iy x_i + \liminf_{n\to\iy} \int_{[0,T]\times[0,1]} \lambda \varphi^n_0(s,y) \,ds\,dy
  		< \iy,
  	\end{align*}
where the last inequality is from \eqref{eqn:controlBoundCompact} and Remark \ref{rmk:summable}. This completes the proof that $(\zeta, \psi) \in \clc_T$. The fact that
$\varphi \in \cls_T(\zeta, \psi)$ is now immediate from part (b).
\end{proof}

We now return to the proof of compactness of $\Gamma_M$. Consider a sequence $\{(\zeta^n,\psi^n)\} \subset \Gamma_M$. Then Lemma \ref{lem:tightCompact} shows that such a sequence is precompact and Lemma \ref{lem:convergencCompact} shows that any limit point
$(\zeta,\psi)$ of $(\zeta^n,\psi^n)$ is in $\Gamma_M$. This establishes the desired compactness.

\section{Bounds on Probabilities of Long Queues}
\label{sec:examples}
In this section we prove Theorem \ref{thm:largeqs}. Fix $j \ge 3$ and recall the notation $G_j, F_j$ from Section \ref{sec:mainResult}. From the LDP in Theorem \ref{thm:mainResult} and since $G_j \subset F_j$ 
\begin{align*}
	\liminf_{n\to\iy}\frac{1}{n}\log(\PP(X^n\in G_j))&\geq- I_T(G_j),\\
	\limsup_{n\to\iy}\frac{1}{n}\log(\PP(X^n\in G_j))&\leq
	\limsup_{n\to\iy}\frac{1}{n}\log(\PP(X^n\in {F}_j))\leq- I_T(F_j).
\end{align*}
In order to prove the first statement in the theorem we first solve for $I_T(F_j)$ and then show that $I_T({F}_j)=I_T(G_j)$.

Fix $\varepsilon > 0$, $(\zeta,\psi)\in \Fbar_j$, and $\varphi \in \Smc_T(\zeta,\psi)$ with
\begin{equation}\label{eqn:contCostBnd}
	\sum_{k=0}^\iy\int_{[0,T]\times[0,1]}\ell(\varphi_k(s,y))\,ds\,dy \le I_T(\zeta,\psi) + \varepsilon \le I_T(\Fbar_j) + 2\varepsilon.
\end{equation}
Define
\begin{align*}
	\tau_i\doteq\inf_{t\in[0,T]}\{\zeta_{i}(t)=1\}, \quad i \in \Nmb.
\end{align*}
Since all queues are of length one at time $0$, we have that $0 = \tau_1 \le \tau_2 \le \dotsb$ and $\zeta_k(\tau_i) = \one_{\{k \le i\}}$.
%
We can assume without loss of generality that $\tau_{j-1}=T$.
To see this, note that if $\tau_{j-1}<T$, then we can consider the delayed trajectory $(\zetabar,\psibar)$ defined by
\begin{align*}
	(\zetabar(t),\psibar(t)) & =(\zeta(0),\psi(0)), \quad 0 \le t \le T-\tau_{j-1}, \\
	(\zetabar(t),\psibar(t)) & =(\zeta(t-(T-\tau_{j-1}),\psi(t-(T-\tau_{j-1})), \quad T-\tau_{j-1} < t \le T.
\end{align*}
Since the cost over time $[0,T-\tau_{j-1}]$ is zero, we have $I_T(\zetabar,\psibar) \le I_T(\zeta,\psi)$.
Thus henceforth we assume $\tau_{j-1}=T$.

We can further assume without loss of generality that $\zeta_k(t)$ is non-decreasing in $t$ for each $k \in \Nmb$.
To see this, consider the new trajectory $\zetabar$ defined by
\begin{align*}
	\zetabar_k(t) & = \max_{0 \le u \le t} \zeta_k(u), k \in \NN.
\end{align*}
Note that this says that for each $i$ and $ t\in (\tau_i,\tau_{i+1})$,
\begin{align*}
	\zetabar_k(t) & = 1, \quad k \le i; \;\; 
	\zetabar_{i+1}(t)  = \max_{0 \le u \le t} \zeta_{i+1}(u); \;\; 
	\zetabar_k(t)  = 0, \quad k \ge i+2.
\end{align*}

We claim that $\zetabar_{i+1}$ is absolutely continuous and $\zetabar_{i+1}'(t)=\zeta_{i+1}'(t)\one_{\{\zeta_{i+1}(t) = \max_{0 \le u \le t} \zeta_{i+1}(u)\}}$ a.e.\ $t \in (\tau_i,\tau_{i+1})$, for ever $i$. 
Absolute continuity is immediate on noting that $0 \le \zetabar_{i+1}(t_2)-\zetabar_{i+1}(t_1) \le \max_{t_1 \le s \le t_2} \zeta_{i+1}(s) - \zeta_{i+1}(t_1)$ for $\tau_i \le t_1 \le t_2 \le \tau_{i+1}$, for every $i$.
Note also that for a.e.\ $t \in (\tau_i,\tau_{i+1})$ such that $\zeta_{i+1}'(t)$ and $\zetabar_{i+1}'(t)$ exist, if $\zeta_{i+1}(t) < \max_{0 \le u \le t} \zeta_{i+1}(u)$, then $\zetabar_{i+1}'(t)=0$. 
On the other hand, if $\zeta_{i+1}(t) = \max_{0 \le u \le t} \zeta_{i+1}(u)$, we have two possible cases:
\begin{itemize}
\item Case 1: there exists a sequence $t_n \downarrow t$ (i.e.\ $t_n$ approaches $t$ from above) with $\zeta_{i+1}(t_n) = \max_{0 \le u \le t_n} \zeta_{i+1}(u)$.
In this case $$\zetabar_{i+1}'(t) = \lim_{n \to \infty} \frac{\zetabar_{i+1}(t_n)-\zetabar_{i+1}(t)}{t_n-t} = \lim_{n \to \infty} \frac{\zeta_{i+1}(t_n)-\zeta_{i+1}(t)}{t_n-t} = \zeta_{i+1}'(t).$$
\item Case 2: A sequence $t_n$ as in Case 1  does not exist, namely there exists some $t_0 \in (t,\tau_{i+1})$ such that $\zeta_{i+1}(s) < \max_{0 \le u \le s} \zeta_{i+1}(u)$ for all $s \in (t,t_0)$.
Then we must have $\zetabar_{i+1}(s) = \zetabar_{i+1}(t)$ for all $s \in (t,t_0)$, and hence $\zetabar_{i+1}'(t)=0$.
From this and $\zeta_{i+1}(t) = \max_{0 \le u \le t} \zeta_{i+1}(u)$ we have $\zeta_{i+1}(s) \le \zeta_{i+1}(t)$ for all $s \in (0,t_0)$.
Therefore $\zeta_{i+1}'(t)=0=\zetabar_{i+1}'(t).$
\end{itemize}
This proves the claim.

Define the control $\bar \varphi$ over the interval  $(\tau_i,\tau_{i+1})$ as
\begin{align*}
	\varphibar_k(t,y) & = \varphi_k(t,y) \one_{\{\zeta_{i+1}(t) = \max_{0 \le u \le t} \zeta_{i+1}(u)\}} + \one_{\{\zeta_{i+1}(t) < \max_{0 \le u \le t} \zeta_{i+1}(u)\}}, k \in \NN_0
\end{align*}
and define the corresponding $\psibar$  by
\eqref{eq:psi1} and \eqref{eq:psii} using $\zetabar$ and $\varphibar$.
Now we show that $(\zetabar,\psibar) \in \Cmc_T$.
For this it suffices to verify property (iv) of $\Cmc_T$.
Let, for $t \in (\tau_i,\tau_{i+1})$
\begin{align*}
	\etabar_k'(t) & = \eta_k'(t) \one_{\{\zeta_{i+1}(t) = \max_{0 \le u \le t} \zeta_{i+1}(u)\}} + \one_{\{\zeta_{i+1}(t) < \max_{0 \le u \le t} \zeta_{i+1}(u)\}}, \quad k < i, \\
	\etabar_i'(t) & = \eta_i'(t) \one_{\{\zeta_{i+1}(t) = \max_{0 \le u \le t} \zeta_{i+1}(u)\}} + (1+\psibar_i'(t)) \one_{\{\zeta_{i+1}(t) < \max_{0 \le u \le t} \zeta_{i+1}(u)\}}, \\
	\etabar_k'(t) & = \eta_k'(t) = 0, \quad k \ge i+1.
\end{align*}
Note that for a.e.\ $t$, if $\zeta_{i+1}(t) = \max_{0 \le u \le t} \zeta_{i+1}(u)$, then $(\zetabar',\psibar',\etabar')=(\zeta',\psi',\eta')$ and hence
\begin{equation*}
	\zetabar_k'(t) = \psibar_k'(t) +\etabar_{k-1}'(t) - \etabar_k'(t), \quad \one_{\{\zetabar_k(t)<1\}}\etabar_k'(t)=0.
\end{equation*} 
If $\zeta_{i+1}(t) < \max_{0 \le u \le t} \zeta_{i+1}(u)$, using the definition of $(\zetabar,\psibar,\etabar,\varphibar)$ it can be verified that the above equation still holds. In particular to check the equation for $k=i+1$ we use the facts $\psibar_i'(t)+\psibar_{i+1}'(t)=-1$ and $\zetabar_{i+1}'(t)=\zeta_{i+1}'(t)\one_{\{\zeta_{i+1}(t) = \max_{0 \le u \le t} \zeta_{i+1}(u)\}}=0$.
Therefore \eqref{eq:zeta_psi_eta} holds for $(\zetabar,\psibar,\etabar)$.
Thus we have that $(\zetabar,\psibar) \in \Cmc_T$, $\varphibar \in \Smc_T(\zetabar,\psibar)$, and
\begin{equation*}
	I_T(\zetabar,\psibar) \le \sum_{k=0}^\iy\int_{[0,T]\times[0,1]}\ell(\varphibar_k(s,y))\,ds\,dy \le \sum_{k=0}^\iy\int_{[0,T]\times[0,1]}\ell(\varphi_k(s,y))\,ds\,dy.
\end{equation*}
We have therefore shown that one  can assume without loss of generality that $\zeta_k(t)$ is non-decreasing in $t$ for each $k \in \Nmb$.
Henceforth we will assume that this holds.
Note that, in particular this says that  $\zeta_1(t)=1$ for all $t \in [0,T]$.

From \eqref{eq:zeta_psi_eta} we have
\begin{align*}
	\sum_{k=1}^\infty \zeta_k(t) = \sum_{k=1}^\infty \psi_k(t).
\end{align*}
From the above display, \eqref{eq:psi1}, \eqref{eq:psii}, and the assumption that $\pi(t)\doteq\max\{k:\zeta_k(t)=1\} \le j-1$ we have
\begin{align*}
	1 & = \sum_{k=1}^\infty (\zeta_k(\tau_{i+1}) - \zeta_k(\tau_i)) 
	 = \sum_{k=1}^{j-1} (\psi_k(\tau_{i+1}) - \psi_k(\tau_i)) \\
	& = \int_{[\tau_i,\tau_{i+1}]\times[0,1]} \varphi_0(s,y) \,ds\,dy - \sum_{k=1}^{j-1} \int_{[\tau_i,\tau_{i+1}]\times[0,1]} \one_{[0,\zeta_k(s)-\zeta_{k+1}(s))}(y) \varphi_k(s,y) \,ds\,dy.
\end{align*}
Let $\theta_i \doteq (\tau_{i+1} - \tau_i)^{-1}$. 
Note that, since $\zeta_1(t)=1$ for every $t$, 
$$\theta_i \sum_{k=1}^{j-1} \int_{[\tau_i,\tau_{i+1}]\times[0,1]} \one_{[0,\zeta_k(s)-\zeta_{k+1}(s))}(y)  \,ds\,dy = 1.$$
It then follows from Jensen's inequality and the convexity of $\ell$ that
\begin{align}\label{eq:boundBelow}
\begin{split}
& \sum_{k=0}^\infty \int_{[\tau_i,\tau_{i+1}]\times[0,1]} \ell(\varphi_k(s,y)) \,ds\,dy \\
& \quad\ge \int_{[\tau_i,\tau_{i+1}]\times[0,1]} \ell(\varphi_0(s,y)) \,ds\,dy + \sum_{k=1}^{j-1} \int_{[\tau_i,\tau_{i+1}]\times[0,1]} \one_{[0,\zeta_k(s)-\zeta_{k+1}(s))}(y) \ell(\varphi_k(s,y)) \,ds\,dy \\
& \quad\ge \theta_i^{-1} \ell\left( \theta_i\int_{[\tau_i,\tau_{i+1}]\times[0,1]} \varphi_0(s,y) \,ds\,dy \right) \\
& \qquad + \theta_i^{-1} \ell\left( \theta_i \sum_{k=1}^{j-1} \int_{[\tau_i,\tau_{i+1}]\times[0,1]} \one_{[0,\zeta_k(s)-\zeta_{k+1}(s))}(y) \varphi_k(s,y) \,ds\,dy \right).
\end{split}
\end{align}
This quantity can be further bounded from below by 
\begin{equation*}
	\frac{1}{\theta_i}\inf \left\{\ell(a)+\ell(b) : a, b \ge 0, a-b=c>0\right\}
\end{equation*}
where $c=\theta_i$.
Using Lagrange multipliers  one finds that the  the above infimum is achieved at
\begin{equation*}
	a = \frac{c+\sqrt{c^2+4}}{2}, \quad b = \frac{-c+\sqrt{c^2+4}}{2}.
\end{equation*}
Plugging this back into \eqref{eq:boundBelow} gives
\begin{align*}
\sum_{k=0}^\infty \int_{[\tau_i,\tau_{i+1}]\times[0,1]} \ell(\varphi_k(s,y)) \,ds\,dy
\ge \theta_i^{-1} \ell\left( \frac{\theta_i+\sqrt{\theta_i^2+4}}{2} \right) + \theta_i^{-1} \ell\left( \frac{-\theta_i+\sqrt{\theta_i^2+4}}{2} \right).
\end{align*}
From this, Jensen's inequality, the convexity of $\ell$, and the fact that $\tau_1=0$, $\tau_{j-1}=T$, $\sum_{i=1}^{j-2}\theta_i^{-1}=T$, we have, letting
$a_j = (j-2)/T$,
\begin{align*}
	& \sum_{k=0}^\infty \int_{[0,T]\times[0,1]} \ell(\varphi_k(s,y)) \,ds\,dy \notag \\
	& \ge \sum_{i=1}^{j-2} \left[ \theta_i^{-1} \ell\left( \frac{\theta_i+\sqrt{\theta_i^2+4}}{2} \right) + \theta_i^{-1} \ell\left( \frac{-\theta_i+\sqrt{\theta_i^2+4}}{2} \right) \right] \notag \\
	& \ge T\ell\left( \frac{(j-2)+\sum_{i=1}^{j-2}\sqrt{1+4\theta_i^{-2}}}{2T} \right) + T\ell\left( \frac{-(j-2)+\sum_{i=1}^{j-2}\sqrt{1+4\theta_i^{-2}}}{2T} \right) \notag \\
	& \ge T\ell\left( \frac{(j-2)+(j-2)\sqrt{1+4(a_j)^{-2}}}{2T} \right) + T\ell\left( \frac{-(j-2)+(j-2)\sqrt{1+4(a_j)^{-2}}}{2T} \right) \notag \\
 	& = T\ell\left( \frac{a_j+\sqrt{4+(a_j)^{2}}}{2} \right) + T\ell\left(\frac{-a_j+\sqrt{4+(a_j)^{2}}}{2} \right),
\end{align*}
where the last inequality is obtained by the convexity of the function $f(x) \doteq \sqrt{1+4x^2}$ and monotonicity of the functions
$x \mapsto \ell(x+c) + \ell(x-c)$ for $c \ge 0$ and $x >c$.
As $\varepsilon>0$ in \eqref{eqn:contCostBnd} is arbitrary, this shows
$$I_T(\Fbar_j) \ge T\ell\left( \frac{a_j+\sqrt{4+(a_j)^{2}}}{2} \right) + T\ell\left(\frac{-a_j+\sqrt{4+(a_j)^{2}}}{2} \right).$$
On the other hand, note that the above lower bound can  be achieved by a $(\zeta,\psi) \in I_T(\Fbar_j)$ given for $k \in \Nmb$, $t \in [0,T]$, and $y \in [0,1]$ as 
\begin{align}\label{eq:optPath}
\begin{split}
	\zeta_k(t) & = 0 \vee \left( a_jt - (k-2) \right) \wedge 1, \\
	\varphi_0(t,y) & = \frac{a_j+\sqrt{4+(a_j)^2}}{2},\\
	\varphi_k(t,y) & = \frac{-a_j+\sqrt{4+(a_j)^2}}{2} \one_{[0,\zeta_k(t)-\zeta_{k+1}(t))}(y) + \one_{[\zeta_k(t)-\zeta_{k+1}(t),1]}(y), \\
	\psi_k(t) & = \one_{\{k=1\}} + \one_{\{k=1\}} \int_{[0,t]\times[0,1]}\varphi_0(s,y)\,ds\,dy  - \int_{[0,t]\times[0,1]}\one_{[0,\zeta_k(s)-\zeta_{k+1}(s))}(y) \varphi_k(s,y) \,ds\,dy.
\end{split}
\end{align}
Therefore
\begin{equation}
	\label{eq:cost_Fbarj}
	I_T(\Fbar_j)=T\ell\left( \frac{a_j+\sqrt{4+(a_j)^2}}{2} \right) + T\ell\left( \frac{-a_j+\sqrt{4+(a_j)^2}}{2} \right).
\end{equation}

Next we show that  $I_T(F_j)= I_T(G_j)$.
Since $G_j\subset F_j$, we clearly have $I_T(F_j)\leq I_T(G_j)$. 

Now we show the reverse inequality.
Fix $\veps>0$.
Consider the  modification of the trajectory \eqref{eq:optPath}, defined over $[0,T-\veps]$, by replacing $T$ throughout by $T-\veps$.
Then $\zeta_k(T-\varepsilon)=\one_{\{k \le j-1\}}$ and the cost of this trajectory over the interval $[0,T-\veps]$ is given by \eqref{eq:cost_Fbarj} with $T$ replaced by $T-\varepsilon$.
Define, for $t \in [T-\varepsilon,T]$ and $y \in [0,1]$,
\begin{align*} 
	\zeta_k(t) &\doteq \zeta_k(T-\veps) \one_{\{k\ne j\}} + (t-(T-\varepsilon))\one_{\{k=j\}}, \\
	 \varphi_k(t,y)&\doteq\one_{\{k=0 \mbox{ or } k\ge j+1\}},
\end{align*}
so that $\zeta_j(T) > 0$ and
\begin{equation*}
	\sum_{k=0}^\infty \int_{[T-\varepsilon,T]\times[0,1]} \ell(\varphi_k(s,y)) \,ds\,dy = j\ell(0)\varepsilon.
\end{equation*}
This trajectory clearly lies in $G_j$ and its cost (over $[0,T]$) converges to \eqref{eq:cost_Fbarj} as $\varepsilon \to 0$.
This implies $I_T(G_j) \le I_T(F_j)$ proving the reverse inequality and hence showing that $I_T(F_j)= I_T(G_j)$. This completes the proof of the
first part of the theorem.

Finally we consider the second part of the theorem.
Since $\sqrt{4+x^2}=2+o(x)$ and $\ell(x)=\frac{(x-1)^2}{2} + o((x-1)^2)$, as $x \to 0$ and $x\to 1$, respectively, we have
\begin{equation*}
	- T\ell\left( \frac{a_j+\sqrt{4+(a_j)^2}}{2} \right) - T\ell\left( \frac{-a_j+\sqrt{4+(a_j)^2}}{2} \right) = -T \left( \frac{(a_j)^2}{4} + o((a_j)^2) \right) 
\end{equation*}
as $T \to \infty$.
Sending $T\to \infty$, we have
\begin{equation*}
	\pushQED{\qed} \lim_{T \to \infty}\lim_{n\to\iy}\frac{T}{n}\log(\PP(X^n\in G_j)) = \lim_{T \to \infty}\lim_{n\to\iy}\frac{T}{n}\log(\PP(X^n\in F_j)) = -\frac{(j-2)^2}{4}. \qedhere \popQED
\end{equation*}

\bibliographystyle{plain}
\bibliography{references}
{\sc 
\footnotesize
\bigskip
\noindent
\begin{minipage}{1\linewidth}
	A. Budhiraja\\
Department of Statistics and Operations Research\\
University of North Carolina\\
Chapel Hill, NC 27599, USA\\
email: budhiraj@email.unc.edu
\end{minipage}\\
}

{\sc 
\footnotesize
\bigskip
\noindent
\begin{minipage}{0.5\linewidth}
E. Friedlander\\
Department of Ecology and Evolution\\
University of Chicago\\
Chicago, IL 60637, USA\\
email: efriedlander@uchicago.edu
\end{minipage}
\hfill
\begin{minipage}{0.4\linewidth}
R. Wu\\
Department of Mathematics\\
University of Michigan\\
Ann Arbor, MI 48109, USA\\
email: ruoyu@umich.edu
\end{minipage}
}

\end{document}


\title[Errata to JSQ LDP]{Errata to ``Many-Server Asymptotics for Join-the-Shortest-Queue: Large Deviations and Rare Events''}
\author{Amarjit Budhiraja, Eric Friedlander, and Ruoyu Wu}
\date{\today}

\begin{abstract}
	This note corrects an error in \cite[Lemma 5.2]{BudhirajaFriedlanderWu2021many}.
	The main results of the paper remain correct as stated.
	\\
	
\noi {\bf AMS 2020 subject classifications:} 
60F10, 
60J74, 
34H05, 
90B15, 
91B70 

\noi{\bf Keywords:} calculus of variations, diminishing rates, discontinuous statistics, golden ratio, infinite-dimensional Skorokhod problem, JSQ, jump-Markov processes in infinite dimensions, large deviations, load balancing
\end{abstract}

\maketitle

This note corrects an error in the statement and proof of \cite[Lemma 5.2]{BudhirajaFriedlanderWu2021many} and provides a corrected proof of \cite[Lemma 5.1]{BudhirajaFriedlanderWu2021many} (the only result to directly use Lemma 5.2) by applying the corrected version of Lemma 5.2.

The error in \cite[Lemma 5.2]{BudhirajaFriedlanderWu2021many} is in the last page of Section 5 (page 2411) in the step of ``Second modification of the trajectory'' that is aimed at proving part (vi) of Lemma 5.2.
Specifically, the statement $\varphi^{new} \in \Smc_T(\zeta^{new},\psi^{new})$ below (5.38) is not correct, 
because (2.10) does not necessarily hold for $i=M-1$ with $(\psi^{new},\zeta^{new},\varphi^{new})$.
We reformulate the lemma  by simply removing property (vi) and removing all arguments starting from ``Second modification of the trajectory'' of the proof on page 2411, until the end of the section.
Specifically, we replace \cite[Lemma 5.2]{BudhirajaFriedlanderWu2021many}, by the following lemma.

\newtheorem*{Lemma5.2}{Lemma 5.2*}

\begin{Lemma5.2}
	\label{lem:uniqueness_preparation_2}
	Fix $\sigma\in(0,1)$. Given $(\zetatil,\psitil)\in\Cmc_T$ with $I_T(\zetatil,\psitil)<\infty$, there exists $(\zeta,\psi)\in\Cmc_T$ and $\varphi \in \Smc_T(\zeta,\psi)$ such that
	\begin{enumerate}[(i)]
	\item
		There exist $N \in \Nmb$, $c_i \in \Nmb_0$ for all $i=0, 1, \ldots N-1$, and a finite partition $0=t_0 < t_1 < \dotsb < t_N = T$, such that $\pi(t)=c_i$ for $t \in (t_i,t_{i+1})$.
	\item
		For any $i$, if $\pi(t_i) > c_i$,  there exists some $\Delta t_i \in (0,t_{i+1}-t_i)$ such that $\varphi_0(t,y)=0$ for $t \in (t_i,t_i+\Delta t_i)$.
	\item
		$\|(\zeta,\psi)-(\zetatil,\psitil)\|_\infty \le \sigma$.
	\item
		$I_T(\zeta,\psi) \leq I_T(\zetatil,\psitil)+\sigma$.
	\item
		$\sum_{k=0}^\infty\int_{[0,T]\times[0,1]} \vartheta_k\ell(\varphi_{k}(s,y)) \,ds\,dy \le I_T(\zeta,\psi) + \sigma$.
	\end{enumerate}
\end{Lemma5.2}

The original \cite[Lemma 5.2]{BudhirajaFriedlanderWu2021many} was only used to prove Lemma 5.1.
We next show how to instead apply Lemma 5.2* to prove Lemma 5.1, by replacing the proof of part (c) in Lemma 5.1 between (5.8) and line 4 on page 2400 (which originally appealed to Lemma 5.2) with the following arguments.

\begin{proof}[Changes in the proof of Lemma 5.1 starting from (5.8)]
	Let $Y$, $\Delta_k$ and $\Mtil$, be as in the original version (page 2399), i.e.
	\begin{equation*}
		Y(t) \doteq \sum_{k=0}^\infty | \rtil_k(t) - r_k(t) |, \qquad
		\Delta_k(t) \doteq \zetatil_k(t) - \zeta_k(t), \quad k \in \Nmb_0,
	\end{equation*} 
	and $\Mtil \doteq \min\{j \in \Nmb: \sup_{t \in[0,T]} \zeta_j(t) < 1, \sup_{t \in[0,T]} \zetatil_j(t) < 1\}$ (which, recall, is finite by property (iii) of $\Cmc_T$). 
	Then there are no reflection terms in the evolution for $\zeta_k$ or $\zetatil_k$ when $k \ge \Mtil+2$.
	So for $s \in [\tau,\tau+\delta]$, for such $k$,
	\begin{align}
		\int_\tau^s |\Delta_k'(t)|\,dt & = \int_\tau^s |\zetatil_k'(t)-\zeta_k'(t)|\,dt = \int_\tau^s |\psitil_k'(t)-\psi_k'(t)|\,dt \notag \\
		& \le \int_{[\tau,s]\times[0,1]} \left(\one_{[0,\rtil_k(t))}(y) + \one_{[0,r_k(t))}(y)\right) \varphi_k(t,y)\,dt\,dy \notag \\
		& \le \int_{[\tau,s]\times[0,1]} \left(\one_{[0,\rtil_k(t))}(y) + \one_{[0,r_k(t))}(y)\right) [\ell(\varphi_k(t,y))+2]\,dt\,dy \notag \\
		& \le 2 \int_{[0,T]\times[0,1]} \ell(\varphi_k(t,y))\,dt\,dy + 4x_k\delta \to 0 \label{eq:correction1}
	\end{align}
	as $k \to \infty$.
	Here the first line uses (2.7), the second uses (2.10), the third uses Lemma 3.2(b), and the last uses summability guaranteed by Lemma 5.2*(v) and that $\max\{r_k(t), \rtil_k(t)\} \le x_k$ (since $\psi_k$ and $\tilde \psi_k$ are nonincreasing).
	Using the fact that $|f|'=f'\one_{\{f>0\}} - f'\one_{\{f<0\}}$ a.e.\ for an absolutely continuous function $f$, we have
	\begin{align*}
		Y(s)-Y(\tau) 
		& = \sum_{k=0}^\infty \int_\tau^s \left[ ( \rtil_k'(t) - r_k'(t) ) \one_{\{\rtil_k(t) > r_k(t)\}} + ( r_k'(t) - \rtil_k'(t) ) \one_{\{\rtil_k(t) < r_k(t)\}} \right] dt \\
		& = \sum_{k=0}^\infty \int_\tau^s \left[ ( \Delta_k'(t)-\Delta_{k+1}'(t) ) \one_{\{\rtil_k(t) > r_k(t)\}} + ( \Delta_{k+1}'(t)-\Delta_k'(t) ) \one_{\{\rtil_k(t) < r_k(t)\}} \right] dt.
	\end{align*}
    By rearranging terms and using the fact that $\Delta_0'(t) \equiv 0$, we can rewrite the partial sum (up to $K$) as
    \begin{align*}
        & \sum_{k=0}^K \int_\tau^s \left[ ( \Delta_k'(t)-\Delta_{k+1}'(t) ) \one_{\{\rtil_k(t) > r_k(t)\}} + ( \Delta_{k+1}'(t)-\Delta_k'(t) ) \one_{\{\rtil_k(t) < r_k(t)\}} \right] dt \\
        & = \sum_{k=1}^K \int_\tau^s \Delta_k'(t) \left[ \one_{\{\rtil_k(t) > r_k(t)\}} - \one_{\{\rtil_k(t) < r_k(t)\}} - \one_{\{\rtil_{k-1}(t) > r_{k-1}(t)\}} + \one_{\{\rtil_{k-1}(t) < r_{k-1}(t)\}} \right] dt \\
        & \quad + \int_\tau^s \Delta_{K+1}'(t) \left[\one_{\{\rtil_K(t) < r_K(t)\}} - \one_{\{\rtil_K(t) > r_K(t)\}} \right] dt
    \end{align*}
    for $K \in \Nmb$.
    Taking the limit as $K \to \infty$ and observing that the last term vanishes by \eqref{eq:correction1}, we see that
    \begin{align*}
        & Y(s)-Y(\tau) \\
		& = \sum_{k=1}^\infty \int_\tau^s \Delta_k'(t) \left[ \one_{\{\rtil_k(t) > r_k(t)\}} - \one_{\{\rtil_k(t) < r_k(t)\}} - \one_{\{\rtil_{k-1}(t) > r_{k-1}(t)\}} + \one_{\{\rtil_{k-1}(t) < r_{k-1}(t)\}} \right] dt.
    \end{align*}
	
	The rest of the proof is the same 
    (on observing that the term in the parenthesis in the integrand and the term $\Delta_k'(t)$ always have opposite signs on $[\tau, \tau+\delta]$)
    by simply replacing  the sentence ``From these we have $Y'(t) \le 0$ and hence $Y(t)=Y(\tau)=0$ for $t \in [\tau, \tau+\delta]$'' in line 1 on page 2400  with ``From these we have $Y(s)-Y(\tau) \le 0$ and hence $Y(s)=Y(\tau)=0$ for $s \in [\tau, \tau+\delta]$''.

    The key point is that the original argument that relied on the interchange of the order of derivative and infinite sum based on the availability of a summable bound (which in turn relied on the original Lemma 5.2 part (vi)) is not needed in the above argument and thus allows us to bypass the use of Lemma 5.2 part (vi).
\end{proof}

We also note two minor typos in the original proof of Lemma 5.2
\begin{itemize} 
\item The bound $\sigma/16\Nbar$ in (5.27) should be replaced by $\sigma/16\Nbar(\Mbar+1)$ 
\item The bound $\sigma/8$ in the first displayed equation on page 2411 should be replaced by $\sigma/16$.
\end{itemize}



{\sc 
\footnotesize
\bigskip
\noindent
\begin{minipage}{1\linewidth}
	A. Budhiraja\\
Department of Statistics and Operations Research\\
University of North Carolina\\
email: budhiraj@email.unc.edu
\end{minipage}\\
}

{\sc 
\footnotesize
\bigskip
\noindent
\begin{minipage}{0.5\linewidth}
E. Friedlander\\
Department of Mathematics\\
St. Norbert College\\
email: eric.friedlander@snc.edu
\end{minipage}
\hfill
\begin{minipage}{0.4\linewidth}
R. Wu\\
Department of Mathematics\\
Iowa State University\\
email: ruoyu@iastate.edu
\end{minipage}
}